\documentclass{amsart}

\pdfoutput=1
\usepackage{amssymb,microtype,fullpage,hyperref,enumerate,graphicx,nicefrac,xfrac,bm}

\hypersetup{colorlinks=true, citecolor=blue, linkcolor=blue, urlcolor=blue, pdfstartview=FitH}

\def\F{\mathbb{F}}
\def\A{{\mathcal A}}
\def\B{{\mathcal B}}

\def\text#1{\hbox{#1}}
\def\Z{{\mathbb Z}}
\def\R{{\mathbb R}}
\def\N{{\mathbb N}}

\def\C{{\mathbb C}}
\def\Q{{\mathbb Q}}
\def\H{{\mathbb H}}
\def\A{{\mathbb A}}
\def\O{{\mathcal O}}

\def\p{\partial}

\def\injrad{{\rm inj rad\,}}
\def\Irr{{\rm Irr}\,}
\def\Ram{{\rm Ram}\:}

\def\Re{{\mathrm {Re}}}

\def\supp{{\rm supp}\:}
\def\Tr{{\rm Tr}}

\def\Hom{{\rm Hom}}

\def\Res{{\rm Res\:}}

\def\Gal{{\rm Gal}}
\def\Ad{{\rm Ad\:}}
\def\cl{{\rm cl\:}}

\def\Vol{{\rm Vol}}

\def\SL{{\rm SL}}
\def\GL{{\rm GL}}
\def\PGL{{\rm PGL}}
\def\SO{{\rm SO}}

\def\PSL{{\rm PSL}}
\def\Sym{{\rm Sym}}

\def\N{{\rm N}}
\def\Tr{{\rm Tr}}
\def\n{{\rm n}}
 
\def\tr{{\rm tr}\,}

\def\Ind{{\rm Ind\:}}

\def\bs{\backslash}

    \DeclareFontFamily{U}{wncy}{}
    \DeclareFontShape{U}{wncy}{m}{n}{<->wncyr10}{}
    \DeclareSymbolFont{mcy}{U}{wncy}{m}{n}
    \DeclareMathSymbol{\Sh}{\mathord}{mcy}{"58}

\date{\today}
\numberwithin{equation}{section}

\usepackage{bbm}
\usepackage{url}
\usepackage{cite}
\usepackage{tikz-cd}

\def \st {\mu^{\rm st}}
\def\PD {{\rm PD}}

\def \PA{{\rm PA}}

\def \ok{{\frak o}}
\def \st{{\rm st}}
\def \mah{{\rm m}}

\def \m{{\rm m}}
\def \rk{{\rm rk}\,}
\def \B{{\rm B}}
\def \NN{{\mathbb N}}
\newcommand\myurl[1]{\url{#1}}
\author{Mikolaj Fraczyk}
\thanks{This work was supported by a public grant as part of the
Investissement d'avenir project, reference ANR-11-LABX-0056-LMH,
LabEx math\'ematique Hadamard and by ERC Consolidator Grant No. 648017.}
\address{Alfr\'ed R\'enyi Institute of Mathematics, Budapest, Reáltanoda utca 13-15., 1053} \email{fraczyk@renyi.hu}
\curraddr{Department of Mathematics, University of Chicago, Chicago, Il 60615}\email{mfraczyk@math.uchicago.edu}

\date{\today}
\title{Strong limit multiplicity for arithmetic hyperbolic surfaces and $3$-manifolds}

\def\K{{\mathbb K}}
\def\PD{{\rm PD}}

\theoremstyle{plain}
\newtheorem{theorem}{Theorem}
\newtheorem{proposition}{Proposition}
\newtheorem{lemma}{Lemma}
\newtheorem{corollary}{Corollary}

\numberwithin{equation}{section}

\theoremstyle{definition}
\newtheorem{definition}{Definition}
\newtheorem{remark}{Remark}
\newtheorem*{acknowledgement}{Acknowledgement}

\begin{document}

\begin{abstract}
We show that every sequence of torsion-free arithmetic congruence lattices in $\PGL(2,\R)$ or $\PGL(2,\C)$ satisfies a strong quantitative version of the Limit Multiplicity property. We deduce that for $R>0$ in certain range, growing linearly in the degree of the invariant trace field, the volume of the $R$-thin part of any congruence arithmetic hyperbolic surface or congruence arithmetic hyperbolic $3$-manifold $M$ is of order at most $\Vol(M)^{11/12}$. As an application we prove Gelander's conjecture on homotopy type of arithmetic hyperbolic $3$-manifolds: We show that there are constants $A,B$ such that every such manifold $M$ is homotopy equivalent to a simplicial complex with at most $A\Vol(M)$ vertices, all of degrees bounded by $B$. 
\end{abstract}
\maketitle
\tableofcontents
\section{Introduction}

Let $G=\PGL(2,\K)$ where $\K=\R,\C$ and let $X:=\H^2,\H^3$ respectively. We study the geometry and topology of locally symmetric spaces $\Gamma\bs X$ and the growth of multiplicities of irreducible unitary representations in $L^2(\Gamma\bs G)$ as $\Gamma$ varies among the set of all arithmetic lattices in $G$. We give bounds on the volume of the thin part of $\Gamma\bs X$ and the multiplicities of irreducible representations in $L^2(\Gamma\bs G)$ that are of order $\Vol(\Gamma\bs X)^{11/12}$ if $\Gamma$ is torsion-free congruence arithmetic and of order $\Vol(\Gamma\bs X)\Delta_k^{-4/9}$, where $k$ is the invariant trace field, if $\Gamma$ is (non-congruence) torsion-free arithmetic.  As an application we construct triangulations of any arithmetic hyperbolic $3$-manifold $\Gamma\bs X$ with bounded degree and the number of triangles depending linearly in volume, solving a conjecture due to Gelander \cite{Gelander1}. The case of arithmetic lattices with torsion is dealt with in a separate joint work with Jean Raimbault \cite{FraRaim}. Our results and methods are related to the limit multiplicity property (see \cite{FinisKyoto}), whose definition we recall in the next part of the introduction. The limit multiplicity property for sequences of congruence arithmetic lattices in a general semisimple Lie group is a well studied subject \cite{DeWa78,DeWa79,Sarn82,DeGe82,BaMo83,Delo86,Clo86,RS87,Savi89}. Starting with the initial discovery of DeGeorge and Wallach \cite{DeWa78}, the area still remains active \cite{FLM15,FiLa15,7Sam}. So far, with the exception of \cite{Matz19, Raim13, 7Sam}, most of the work focused only on sequences of congruence subgroups of a fixed ambient arithmetic lattice. In contrast, the present paper focuses on non-commensurable families of lattices. In that setting, the work of Abert \textit{et al.}\cite{7Sam} was a major breakthrough as it dealt with the families of non-commensurable lattices for the first time. The methods used in \cite{7Sam} rely crucially on ergodic theory of the actions of higher rank groups and as such cannot be applied to rank $1$ groups.  The papers \cite{Matz19, Raim13} covered certain families of lattices in $\PGL(2,\K)$ (and finite products of such) with invariant trace field \cite[Def 3.3.6]{MaRe03} of bounded degree. Removing the bounded degree assumptions is arguably the main difficulty that we face in the present paper. Perhaps surprisingly, the key ingredient that we use to handle lattices with trace fields of arbitrary degree is a beautiful theorem of Bilu \cite{Bilu} which states that the Galois orbits of points of $\mathbb G_m(\Q^{\rm alg})$ of small Weil height equidistribute on the unit circle. Another crucial ingredient is Theorem \ref{thm:CharacterBounds} giving uniform bounds on the values of irreducible characters of $p$-adic groups on regular elements, which may be of independent interest.

The technical heart of the paper is an upper bound on the part of the geometric side of Selberg's trace formula coming from the non-torsion regular semisimple elements. Let $$\B(R):=\left\{\begin{pmatrix} a& b\\ c& d\end{pmatrix}\in \PGL(2,\K)| \frac{|a|^2+|b|^2+|c|^2+|d|^2}{|ad-bc|}\leq e^R+e^{-R}\right\}, R>0.$$ We prove:
\begin{theorem}\label{mthm:TraceEstimate} There exists a constant $1/2>\eta>0$ with the following property. Let $\Gamma\subset G:=\PGL(2,\K)$ be an arithmetic lattice with invariant trace field $k$ (c.f. \cite[Def. 3.3.6]{MaRe03}). Write $W\subset G$ for the set of regular semisimple \textbf{non-torsion} elements. For every $C>0$ and every $f\in C_c(G)$ with $\|f\|_\infty\leq 1, C>0$ and $\supp f\subset \B(\eta [k:\Q]+C)$ we have 

 $$\left|\sum_{[\gamma]\in \Gamma\cap W}\Vol(\Gamma_\gamma\bs G_\gamma)\int_{ G_\gamma\bs G}f(g^{-1}\gamma g)dg\right| \ll_{C}\Vol(\Gamma\bs G)\Delta_k^{-4/9}.$$ 
 Moreover if $\Gamma$ is a congruence lattice, then 
 $$\left|\sum_{[\gamma]\in \Gamma\cap W}\Vol(\Gamma_\gamma\bs G_\gamma)\int_{ G_\gamma\bs G}f(g^{-1}\gamma g)dg\right| \ll_{C}\Vol(\Gamma\bs G)^{11/12}\Delta_k^{-4/9}.$$ 
\end{theorem}
Vast majority of our work is devoted to the proof of Theorem \ref{mthm:TraceEstimate}. Rest of the results are relatively quick corollaries, which we describe below.
\subsection{Limit multiplicity property}
Let $\Pi( G)$ be the set of equivalence classes of irreducible unitary representations of $G$ and let $\Pi( G)^{\rm temp}$ be the subset of tempered representations. There is a natural topology on $\Pi(G)$ called the Fell topology \cite{Dixmier}. The Plancherel measure $\mu_{\rm Pl}$ \cite{Dixmier} is a Borel measure supported on $\Pi( G)^{\rm temp}$. Let $\Gamma$ be a lattice in $G$ and let $L^2(\Gamma\bs G)_{\rm disc}$ be the orthogonal sum of all irreducible sub-representations of $L^2(\Gamma\bs G)$. We remark that in the case when $\Gamma$ is co-compact we have $L^2(\Gamma\bs G)_{\rm disc}=L^2(\Gamma\bs  G)$. For every $\pi\in \Pi(G)$ let $m(\pi,\Gamma)$ be the multiplicity of $\pi$ in $L^2(\Gamma\bs  G)$, it is always a finite natural number and we have $L^2(\Gamma\bs G)_{\rm disc}=\sum_{\pi\in\Pi( G)}m(\pi,\Gamma)\pi.$ We define the measure 
$$\mu_{\Gamma}:=\frac{1}{\Vol(\Gamma\bs G)}\sum_{\pi\in\Pi(\mathbf G)}m(\pi,\Gamma)\delta_\pi.$$ It is an infinite measure on $\Pi( G)$. Let $(\Gamma_n)_{n\in\mathbb N}$ be a sequence of lattices. We say that $(\Gamma_n)_{n\in\mathbb N}$ has the \textbf{ limit multiplicity property} (see \cite{FinisKyoto}) if 
\begin{enumerate}
 \item For any $f\in C_c(\Pi( G)^{\rm temp})$ we have $\lim_{n\to\infty} \int f d\mu_{\Gamma_n}=\int f d\mu_{\rm Pl};$
 \item For every compact set $A\subset \Pi( G)\setminus \Pi( G)^{\rm temp}$ we have $\lim_{\n\to\infty} \mu_{\Gamma_n}(A)=0$.
\end{enumerate}

Sauvageot \cite{Sauv96} proved\footnote{There is a gap in the proof in \cite{Sauv96}, recently discovered by Nelson and Venkatesh \cite[p. 103 footnote 7]{NelsonVenkatesh}. Fortunately the problem does not appear for $G=\PGL(2,\R),\PGL(2,\C).$} that for a sequence of co-compact lattices $(\Gamma_n)_{n\in\mathbb N}$ the limit multiplicity follows once we know that 
\begin{equation}\label{eq:Sauvageot}
 \lim_{n\to\infty} \Vol(\Gamma\bs G)^{-1}\sum_{[\gamma]\neq 1}\Vol(\Gamma_\gamma\bs G_\gamma) \int_{G_\gamma\bs G}f(g^{-1}\gamma g)dg =0, 
\end{equation}
for every compactly supported test function $f$ on $G$. One can recognize that the left hand side is the geometric side of Selberg's trace formula with contribution of the trivial conjugacy class removed. 

Combining Theorem \ref{mthm:TraceEstimate} with criterion (\ref{eq:Sauvageot}) we immediately obtain:
\begin{theorem}\label{mthm:LMP}
Let $(\Gamma_n)_{n\in\mathbb N}$ be a sequence of pairwise non-conjugate torsion-free co-compact congruence arithmetic lattices in $G$. Then $(\Gamma_n)_{n\in\mathbb N}$ has the limit multiplicity property. 
\end{theorem}
After the first version of this manuscript appeared, Jasmin Matz proved (as a special case of her main result) the limit multiplicity property for sequences of lattices of $\PGL(2,\C)$ of the form $\PGL(2,\O_k)$ where $k$ is an imaginary quadratic number field. Lattices like these are never co-compact and every non co-compact arithmetic  lattice of $\PGL(2,\C)$ is commensurable to one of them. Therefore, it is likely that the assumption of co-compactness in Theorem \ref{mthm:LMP} could be removed. 
In the non co-compact case the limit multiplicity says that the contribution of continuous spectrum becomes negligible. 
We do not investigate this direction here, since it is of different nature from what we do in the proof of Theorem \ref{mthm:TraceEstimate}.

\subsection{Benjamini-Schramm convergence}
Benjamimi-Schramm topology can be defined for any family of metric spaces with reasonably bounded geometry. Since we only deal with hyperbolic orbifolds we will restrict the attention to this family. Let $X=\H^2$ or $\H^3$, let $\Gamma$ be a discrete subgroup of $G:=\PGL(2,\R),\PGL(2,\C)$ respectively. Choose a maximal compact subgroup $K$ of $G$ and identify $X=G/K$. For $x=\Gamma g K\in \Gamma \bs X$ define the \textbf{ injectivity radius } as $\injrad(x)=\sup\left\{r\geq 0|  \Gamma^g\cap \B(r)=\{1\}\right\}.$ Intuitively, this is the supremum of radii $r$ such that $r$-ball around $x$ lifts to an $r$-ball in $X$. The $R$-\textbf{thin} part of the orbifold $\Gamma\bs X$ is given by 
\begin{equation*}
\left(\Gamma\backslash X\right)_{<R}:=\left\lbrace\Gamma gK\mid \B(R)\cap g^{-1}\Gamma g\neq \{1\}\right\rbrace=\left\lbrace \Gamma g K| \injrad(x)<R\right\rbrace 
\end{equation*} 
 The injectivity radius of a quotient $\Gamma\bs X$ is defined as $\injrad(\Gamma\bs X):=\inf_{x\in \Gamma\bs X} \injrad x.$ Given a sequence of lattices $(\Gamma_i)_{i\in \N}$ we consider the sequence of locally symmetric spaces $(\Gamma_i\backslash X)_{i\in \N}$. We say that the sequence $(\Gamma_i\backslash X)_{i\in \N}$ converges Benjamini-Schramm to $X$ if for every $R>0$ we have 
\begin{equation*}
\lim_{i\to\infty}\frac{\Vol((\Gamma_i\backslash X)_{<R})}{\Vol(\Gamma_i\backslash X)}=0 
\end{equation*}
The notion of Benjamini-Schramm convergence originates from the paper \cite{BeSc01} where Benjamini and Schramm introduced it for sequences of graphs of bounded degree. For locally symmetric spaces it was defined and studied in \cite{7Sam}. It is a special case of Benjamini-Schramm convergence for metric spaces with probability measures (see \cite[Chapter 3]{7Sam}).

Using Theorem \ref{mthm:TraceEstimate} we prove:
\begin{theorem}
\label{mthm:BSconv} There exists $\eta>0$ with the following property. Let $\Gamma$ be a torsion-free arithmetic lattice in $G$, with invariant trace field $k$ and let $R=C+\eta[k:\Q]$ for some constant $C>0$. Then 
$$ \Vol((\Gamma\bs X)_{<R})\ll_C \begin{cases}
                                  \Vol(\Gamma\bs X)^{11/12}\Delta_k^{-4/9} &\textrm{ if }\Gamma \textrm{ is a congruence lattice,}\\
                                  \Vol(\Gamma\bs X)\Delta_k^{-4/9} &\textrm{ otherwise.}
                                 \end{cases}$$
                                 
\end{theorem}
We prove it in Section \ref{sec:BSConvergence}. Combined with a well known result of Borel \cite{Bor81} that there exist only finitely arithmetic lattices in $ G$ with bounded covolume we deduce:
\begin{corollary}\label{cor:BSconv} Let $(\Gamma_n)_{n\in\mathbb N}$ be a sequence of pairwise non-conjugate, torsion-free congruence arithmetic lattices in $G$. Then, the sequence of manifolds $(\Gamma_n\bs X)_{n\in\mathbb N}$ converges Benjamini-Schramm to $X$. \end{corollary}
After the first version of this paper appeared, the author together with Jean Raimbault were able extend the above result to all congruence arithmetic lattices by combining estimate from Theorem \ref{mthm:TraceEstimate} with ``soft'' ergodic methods \cite[Theorem A]{FraRaim}. In the previous version of the present paper we proved that every sequence of torsion-free, pairwise non-commensurable arithmetic lattices converges Benjamini-Schramm to $X$. This result has been superseded by \cite[Theorem A]{FraRaim}. Since \cite{FraRaim} refers to that previous version, in Section \ref{sec:BSConvergence} we will briefly indicate how \cite[Theorem A]{FraRaim} can be proven using the present version. 

\subsection{Triangulations of arithmetic hyperbolic $3$-manifolds.}
As an application of Theorem \ref{mthm:BSconv}  we prove Gelander's conjecture \cite[Conjecture 1.3]{Gelander1} for arithmetic $3$-manifolds:
\begin{theorem}\label{mt.Gelander}
There exist absolute positive constants $A,B$ such that every arithmetic, hyperbolic $3$-manifold $M$ is homotopy equivalent to a simplicial complex with at most $A\Vol(M)$ vertices where each vertex has degree bounded by $B$ (if $M$ is compact we can take $B=3104$). 
\end{theorem}
As a simple corollary we obtain:
\begin{corollary}\label{mc.Gelander}
There exists a constant $C>0$ such that any torsion-free arithmetic lattice $\Gamma$ in $\PGL(2,\C)$ admits a presentation 
$$\Gamma=\langle S\mid\Sigma\rangle$$
where $|S|,|\Sigma|$ is bounded by $C\Vol(M)$ and all relations in $\Sigma$ are of length at most $3$.
\end{corollary}
From which we deduce yet another corollary: 
\begin{corollary}\label{cor:TorsionBd}Let $\Gamma$ be a torsion-free, arithmetic lattice in $\PGL(2,\C)$. Then  
$$\log |H_1(\Gamma\bs \mathbb H^3,\Z)_{\rm tors}|\ll \Vol(\Gamma\bs \mathbb H^3).$$
\end{corollary}
We prove Theorem \ref{mt.Gelander} in Section \ref{sec:Gelander}
\subsection{Growth of Betti numbers.}
Matsushima's formula \cite{matsu,BerClo,BW} provides a link between the spectral decomposition of $L^2(\Gamma\bs G)$ and dimensions of cohomology groups $H^i(\Gamma\bs X,\C)$. We use standard notation $b_i(\Gamma\bs X):=\dim_{\C} H_i(\Gamma\bs X,\C)$ and we write $b^{(2)}_i(X)$ for the $L^2$-Betti numbers of $X$.
From Theorem \ref{mthm:LMP} we deduce:
\begin{corollary}\label{mc.Betti}
Let $(\Gamma_n)_{n\in \N}$ be a sequence of pairwise distinct congruence arithmetic co-compact torsion-free lattices in $\PGL(2,\C)$. Then 
$$\lim_{n\to\infty} \frac{b_i(\Gamma_n\bs X)}{\Vol(\Gamma_n\bs X)}=b^{(2)}_i(X)=0.$$
\end{corollary}
We remark that in \cite{FraRaim} this result is extended to all sequences congruence lattices as well as sequences of pairwise non-commensurable arithmetic lattices.
Using a method of DeGeorge and Wallach \cite{DeWa78} and Theorem \ref{mthm:TraceEstimate} we prove 
\begin{theorem}\label{thm:MultiUB}
For any co-compact torsion-free arithmetic lattice $\Gamma\subset \PGL(2,\C)$ with invariant trace field $k$ we have 
$$ \frac{b_1(\Gamma\bs \H^3)}{\Vol(\Gamma\bs \H^3)}\ll [k:\Q]^{-1}.$$
\end{theorem}
The proofs of these results occupy Section \ref{sec:BettiNumbers}.
\subsection{Comparison with previous work}
 The first results and the definition of limit multiplicity property were given by DeGeorge and Wallach \cite{DeWa78,DeWa79}. Let $G$ be a semisimple Lie group and let $(\Gamma_n)_{n\in\mathbb N}$ be a sequence of lattices such that $\injrad (\Gamma_n\bs X)\to\infty$. Then, for every irreducible unitary representation representation $\pi$ we have \cite[Cor. 3.3, Thm. 6.2]{DeWa78}$$\lim_{n\to\infty} \frac{m(\pi,\Gamma)}{\Vol(\Gamma\bs G)}=\begin{cases}d_\pi & \textrm{ if }\pi \textrm{ is discrete series}\\
0 & \textrm{ otherwise,} 
\end{cases}$$ where $d_\pi$ is the formal degree of $\pi$ with respect to the same Haar measure on $G$ that is used to compute $\Vol(\Gamma\bs G)$. 
In the same paper DeGeorge and Wallach conjectured that towers of normal subgroups with trivial intersection have the limit multiplicity property. This was subsequently proven by Delorme \cite{Delo86}.
In the case of families of non-uniform lattices establishing limit multiplicity property has additional difficulty of isolating the discrete part of the spectrum.
This has been successfully done for sequences of principal congruence subgroups of $\SL(N),\GL(N)$, by Finis-Lapid-M\"uller \cite{FLM15}, and later extended to families of all congruence subgroups of a given arithmetic lattice $\Gamma$ in $\SL(N),\GL(N)$ by Finis and Lapid \cite{FiLa15}. The aforementioned results, except \cite{DeWa78}, concerned families of lattices contained in a single ambient lattice. By taking a radically different approach Abert, Bergeron, Biringer, Gelander, Nikolov, Raimbault and Samet \cite{7Sam} proved that the limit multiplicity holds for any sequence of co-compact lattices $(\Gamma_n)_{n\in\mathbb N}$ in a simple higher rank Lie group $G$, provided that the injectivity radius $\injrad(\Gamma_n\bs X)_{n\in\NN}$ is uniformly bounded away from $0$. They prove that by relating the limit multiplicity property to the Benjamini-Schramm convergence. For non-commensurable families of non-uniform lattices, and $G=\PGL(2,\R)^a\times\PGL(2,\C)^b$, the limit multiplicity property for lattices of the form $\PGL(2,\O_k)$ was established recently by Jasmin Matz \cite{Matz19}.

Benjamini-Schramm convergence of locally symmetric spaces was first considered in \cite{7Sam}. One of the main results of that paper asserts that if $G$ is a higher rank Lie groups with property (T) and $X$ is its symmetric space, then \textbf{every} sequence of pairwise non isometric locally symmetric spaces $\Gamma_{n\in\NN}\bs X$ converges Benjamimi-Schramm to $X$. This is deduced from St\"uck-Zimmer theorem on ergodic measure preserving actions of higher rank Lie groups. The methods of \cite{7Sam} do not apply to rank one groups, and for a good reason, since in rank $1$ we do not expect that every sequence converges Benjamini-Schramm to $X$. After the first version of this manuscript appeared, Arie Levit proved \cite{Levit} that any sequence of irreducible, congruence arithmetic lattices in a higher rank group converges Benjamimi-Schramm to $X$. Finally, in case $G=\PGL(2,\R),\PGL(2,\C)$ Jean Raimbault proved that if $(\Gamma_n)_{n\in\mathbb N}$ is a sequence of pairwise non-conjugate congruence arithmetic lattices such that degree of the invariant trace field is at most $3$ then $\Gamma_n\bs X$ converges Benjamini-Schramm to $X$.

Compared to the previous work on Benjamini-Schramm convergence or limit multiplicity property in groups of rank $1$, our result is the first one applicable to families of lattices with unbounded degree of the trace fields. The formulas for the volumes appearing on the both sides of Theorem \ref{mthm:TraceEstimate} involve many quantities that \textit{ a priori } could be exponential in the degree of the invariant trace field. Fortunately, when the degree of the trace field goes to infinity we can start exploiting Diophantine properties of the eigenvalues of conjugacy classes contributing to the trace formula, most notably Bilu equidistribution theorem \cite{Bilu}, to show that either no classes in $\Gamma$ contribute to LHS or many of these \textit{ a priori } exponential quantities must be in fact sub-exponential. This phenomenon is completely new and appears only when we consider sequences of lattices with unbounded degrees of invariant trace fields.

Finally, we should mention the classical result of Chinburg and Friedman on the smallest volume of arithmetic hyperbolic $3$-orbifold \cite{ChFr86} and the recent work of Linowitz, McReynolds, Pollack and Thompson \cite{LMPT2017-1,LMPT2017-2,LMPT2018} that describe in an effective fashion other geometric and topological features of arithmetic hyperbolic $2$ and $3$-manifolds.
\subsection{Outline of the paper}
In Section \ref{sec:QAlg} we recall the definition and basic properties of quaternion algebras and in Section \ref{sec:CongruenceLattices} we explain how quaternion algebras are used to construct and parametrize congruence arithmetic lattices in $\PGL(2,\K)$. We fix the volumes and measures on relevant groups in Section \ref{sec:VolumeConventions}. In Section \ref{sec:Bilu} we prove a quantitative version of Bilu equidistribution theorem and some of its consequences that we will use later. Section \ref{sec:TraceFormulas} is devoted to the proof of Lemma \ref{lem:TFTraceFormula}, which bounds the LHS of Theorem \ref{mthm:TraceEstimate} by a part of the geometric side of the adelic Selberg trace formula. In Section \ref{sec:NAEstimates} we prove upper bounds on the non-Archimedean orbital integrals appearing in RHS of Lemma \ref{lem:TFTraceFormula}. The proof of this bound (Proposition \ref{prop:NonArchBound}) splits into few substantial parts so we wrote its own outline at the beginning of Section \ref{sec:NAEstimates}. In Section \ref{sec:Covolumes} we give covolume formulas for congruence lattices, derived from the Borel volume formula. Next, we give volume formulas for centralizers appearing in the RHS of Lemma \ref{lem:TFTraceFormula}. Section \ref{sec:Archimedean} is where we exploit the support condition $\supp f\subset \B(C+\eta[k:\Q])$. First we give a (standard) estimate of Archimedean orbital integrals (Section \ref{ssec:AOI}), next we give a parametrization of conjugacy classes with non-vanishing orbital integrals and bound their number (Section \ref{sec:ConjClasses}) and finally we use the results from Section \ref{sec:Bilu} to estimate the volumes and orbital integrals involved in the upper bound in Lemma \ref{lem:TFTraceFormula}. We gather all the ingredients and prove Theorem \ref{mthm:TraceEstimate} in Section \ref{sec:ProofMT}. Theorem \ref{mthm:BSconv} is deduced in Section \ref{sec:BSConvergence}, Gelander's conjecture is proved in Section \ref{sec:Gelander} while Theorem \ref{thm:MultiUB} and Corollary \ref{mc.Betti} are proved in Section \ref{sec:BettiNumbers}.

At the first reading we suggest the following order. Start with Sections \ref{sec:QAlg},\ref{sec:CongruenceLattices} to get familiar with the setup. Skim though the Section \ref{sec:TraceFormulas} leading up to proof of Lemma \ref{lem:TFTraceFormula}. Then one can read the proof of Theorem \ref{mthm:TraceEstimate} in Section \ref{sec:ProofMT} and refer to earlier parts as needed. 
\subsection{Notation}
\begin{itemize}
 \item $k,\mathbb A,\mathbb A_\infty,\mathbb A_f$ will usually stand for a number field, the ring of its adeles, infinite adeles and finite adeles respectively;
 \item $\Sigma,\Sigma_\infty,\Sigma_f$ will denote the set of places, Archimedean places, finite places of $k$ respectively;
 \item For $\nu\in \Sigma$ we write $k_\nu$ for the completion of $k$ w.r.t. $\nu$, we reserve letter $\frak p$ for finite places of $k$, we write $|\cdot|_\nu$ for the corresponding multiplicative valuation, normalized so that $|x|_\nu=|x|^{[k_\nu,\R]}$ if $\nu\in \Sigma_\infty$ and $|x|_{\frak p}=q^{-r}$ where $r$ is the power of $\frak p$ dividing $x$ and $q:=|\frak o/\frak p|$ if $\frak p\in \Sigma_f$;
 \item $\frak o_k,\frak o_p$ will denote the ring of integers of $k$ and the ring of integers of $k_\frak p$ respectively;
 \item $\Delta_k$ is the absolute value of the field discriminant of $k$, $\Delta_{l/k}$ is the relative discriminant of a field extension $l/k$;
 \item $R_k,\cl(k),\rho_k,w_k$ is the regulator, the class group, the residue of Dedekind zeta function and the number of roots of $1$ in $k$;
 \item $[k:\Q], r_1, r_2, s=r_1+r_2$ will be the degree, the number of real places, the number of complex places and the total number of Archimedean places of $k$;
 \item $\pi_k(X), X\geq 0$ denotes the number of prime ideals $\frak p\subset \frak o$ with $N(\frak p)\leq X$;
 \item For a finite extension $l/k$ we write $\Tr_{l/k},\N_{l/k}\colon l\to k$ for the trace and norm;
 \item Let $\alpha$ be an algebraic number. The Mahler measure of $\alpha$ is $M(\alpha):=|a|\prod_{\alpha'\sim \alpha}\max\{1, |\alpha'|\}$ where $a$ is the leading coefficient of the minimal polynomial of $\alpha$ and $\alpha'$ are the Galois conjugates of $\alpha$. We define the logarithmic Mahler measure as $\m(\alpha):=\log M(\alpha)$ and write $h(\alpha):=m(\alpha)/[\Q(\alpha):\Q]$ for the Weil height of $\alpha$; 
 \item $\Gal(k)$ is the absolute Galois group of $k$ with profinite topology, $\Gal(l/k)$ is the Galois group of a Galois extension $l/k$;
 \item Let $\chi$ be a finite dimensional representation of $\Gal(k)$, we write $\frak f_\chi$, $L(s,\chi),\rho_\chi$ for the Artin conductor of $\chi$, the Artin $L$-function of $\chi$ and the leading term of the Laurent expansion of $L(s,\chi)$ at $s=1$;
 \item Let $G$ be a profinite group with open subgroup $H$. We write $\Ind_H^G\rho$ for the induction of a representation $\rho$ of $H$ to $G$;
 \item Let $\rho$ be a representation of a profinite group $G$. We write $\chi_\rho$ for the character of $\rho$;
 \item $\Irr\, G$ is the set of equivalence classes of unitary representations of $G$;
 \item Let $G$ be a group, we write $G_\gamma$ for the centralizer of $\gamma\in G$, likewise for algebraic groups; 
 \item If $\mathbf G$ is a reductive group with Lie algebra $\frak g$ we will write $\Ad\colon \mathbf G\to \GL(\frak g)$ for the adjoint map;
 \item Write $\mathbf G^{rs}$ for the set of regular semisimple elements of $\mathbf G$;
 \item For a torus $\mathbf T\subset \mathbf G$, write $\Phi(\mathbf G,\mathbf T)$ for the set of roots of $\mathbf G$ with respect to $\mathbf T$. For $\gamma\in \mathbf T$ the Weyl discriminant is defined as $\Delta(\gamma):=\prod_{\lambda\in \Phi(\mathbf G,\mathbf T)}(1-\lambda(\gamma));$ 
 \item For an algebraic torus $\mathbf T$, defined over a field $F$ we write $X^*(\mathbf T)$ for the  module of characters and $X^*(\mathbf T)_F$ for the module of $F$-rational characters;
 \item If $A$ is a quaternion algebra over  $k$, write $\Ram A,\Ram_\infty A,\Ram_f A$ for the sets of ramified places, ramified Archimedean places and ramified finite places;
 \item Write $\tr,\n\colon A\to k$ for the reduced trace and reduced norm maps;
 \item Let $k_{A}^\times:=\{x\in k^\times\mid \, (x)_\nu>0 \textrm{ for } \nu\in \Ram_\infty A\}$, where $(-)_\nu$ is the real embedding of $k$ corresponding to $\nu$;
 \item Write $\H^2,\H^3$ for the hyperbolic $2-$ and $3-$spaces;
 \item For a locally symmetric space $M$ and $x\in M$ write $\B(x,R)$ for the open ball of radius $R$ around $x$, $\injrad x$ for the injectivity radius around $x$ and $(M)_{\leq R}=\{x\in M| \injrad x\leq R\}$ for the $R$-thin part of $M$; 
 \item All logarithms are in base $e$;
 \item We use Landau's big-O notation $f=O_A(g)$ or Vinogradov's notation $f\ll_A g$ if there exists a constant $C>0$ dependent on $A$ such that $|f|\leq C|g|$;
 \item We shall use small $o$-notation: $f(x)=o_A(g(x))$ if $\lim_{x\to\infty} \frac{f}{g}=0$ and the speed of convergence depends only on $A$; 
 \item For any fixed $C,\eta>0$ the expression $\exp(o_{C,\eta}(1))$ is a quantity that tends to $1$ as $[k:\Q]\to \infty$, and $\Delta_k^{o_{C,\eta}(1)}$ is a quantity that grows slower than any positive power of $\Delta_k$, as $\Delta_k\to \infty$. In particular, by Minkowski's lower bound,  $\exp(o_{C,\eta}([k:\Q]))\ll \Delta_k^{o_{C,\eta}(1)}$;
 \item For any set $U$ we will write $\mathbf 1_U$ for its characteristic function;
 \item For a topological space $X$ we will write $C_c(X)$ for the space of continuous compactly supported functions on $X$, $\pi_1(X)$ for the fundamental group of $X$ (if it's connected) and $b_i(X)=\dim_\C H_i(X,\C)$ for the $i$-th Betti number of $X$;
 \item $\langle v,w\rangle$ will stand for the Hermitian or inner  product of two vectors in a complex or real Hilbert space;
 \item If $\Gamma\subset G$ is a lattice in $G$ write $R_\Gamma f$ for the operator
$$
R_\Gamma f(\Phi)(x):=\int_{G}f(g)\Phi(xg)dg \textrm{ for } \Phi\in L^2(\Gamma\backslash G).
$$
\end{itemize}

\section{Background}
\subsection{Quaternion algebras}\label{sec:QAlg}
We follow the notation from \cite{MaRe03}. Let $F$ be a field of characteristic different than $2$, let $a,b\in F^\times$. The quaternion algebras over $F$ are the algebras $\left(\frac{a,b}{F}\right)$ defined by 
$$\left(\frac{a,b}{F}\right)=F+ {\bf i}F+{\bf j}F+{\bf ij}F, \ \ {\bf i}^2=a,{\bf j}^2=b, {\bf ij}=-{\bf ji}.$$ 
If $x=x_1+{\bf i}x_2+{\bf j}x_3+{\bf ij}x_4$, we define the {\bf conjugate} $\overline x:=x_1-{\bf i}x_2-{\bf j}x_3-{\bf ij}x_4$, the {\bf trace} $\tr x:=x+\overline x$ and the {\bf norm} $\n x:=x\overline x.$ Let $A$ be a quaternion algebra over $F$ and let $E/F$ be an extension of $F$. Write $A_E:=A\otimes_F E$. We say that $A$ {\bf splits over $E$} if $A_E\simeq M(2,E)$. 
From now on let $k$ be a number field and let $A$ be a quaternion algebra over $k$. For every place $\nu\in \Sigma_k$ put $A_\nu:=A\otimes_k k_\nu$. If $k_\nu\neq \C$, then there are exactly two isomorphism classes of quaternion algebras over $k_\nu$. One is split and the other one is the unique division quaternion algebra over $k_\nu$. We shall denote the latter by $D_\nu$. With the induced topology, the group $D_\nu^\times$ is compact modulo center. Define the {\bf ramification set} of $A$ as $\Ram A:=\{\nu\in \Sigma_k| A_\nu\simeq D_\nu\}$. Write $\Ram_f A=\Ram A\cap \Sigma_f, \Ram_\infty=\Ram A\cap \Sigma_\infty$. Write $\PA^\times$ for the algebraic group $A^\times/ k^\times$. It is an inner form of $\PGL(2,k)$ and we have 
$$ \PA^\times(k_\nu)\simeq \begin{cases}
                      \PGL(2,k_\nu) & \textrm{ if } \nu\in \Ram A,\\
                      {D}_\nu^\times/ k_\nu^\times & \textrm{ otherwise.}
                     \end{cases}$$
An element $x\in A$ is semisimple (regular) if it is diagonalizable (and regular) as a matrix in $A_{\overline{k}}\simeq M(2,\overline{k}).$ For any regular semisimple element $x\in A$ the centralizer of $x$ in $A$ is $k[x]$. It is either a quadratic extension of $k$ or it is isomorphic to $k^2$. We call $x$ \textbf{elliptic} and \textbf{split} respectively. We say an element $\gamma\in \PA^\times$ is semisimple, regular, elliptic or split in $\PA^\times$ if any preimage in $A^\times$ is. Finally, for a regular semisimple $x\in A^\times$ we have $\tr x=\Tr_{k(x)/k}(x)$ and $\n x=N_{k(x)/k}(x).$ If $\gamma\in \PA^\times(k)$ is regular semisimple and $\gamma^2\neq 1$, then the centralizer of $\gamma$ is a connected torus isomorphic to $\Res_{k(\gamma)/k}\mathbb G_m/\mathbb G_m$. 

\subsection{Congruence lattices}\label{sec:CongruenceLattices}
Fix $\K=\R$ or $\C$. Let $k$ be a number field with Archimedean places $\nu_1,\ldots, \nu_{s}$ and let $A$ be a quaternion algebra over $k$. We say that the pair $(k,A)$ is \textbf{admissible} if $k_{\nu_1}\simeq \K$, $k_{\nu_i}\simeq \R$ for $i=2,\ldots s$ and $A_{\nu_1}\simeq M(2,\K), A_{\nu_i}\simeq D_\R$ for $i=2,\ldots,s$. The algebra $D_\R$ is the algebra of Hamilton quaternions. Throughout the paper we will work with a fixed admissible pair $k,A$ and we shall write $\mathbf G=\PA^\times$ and $\mathbf G_\gamma$ for the centralizer of an element $\gamma\in \mathbf G(k)$. We have 
$$\mathbf G(\mathbb A)=\prod_{\nu\in \Sigma}^* \mathbf G(k_\nu)\simeq \PGL(2,\K)\times \SO(3)^{s-1}\times \mathbf G(\mathbb A_f).$$
Let $p\colon \mathbf G(\mathbb A)\to \PGL(2,\K)$ be the projection onto the first coordinate. For any open compact subgroup $V\subset \mathbf G(\mathbb A_f)$ we define 
\begin{equation}\label{def:GammaV}\Gamma_V:=\mathbf G(k)\cap\left(\mathbf G(\mathbb A_\infty)\times V\right).\end{equation}
When talking about lattices of $\PGL(2,\K)$ we will identify $\Gamma_V$ with its projection $p(\Gamma_V)\subset \PGL(2,\K)$. This is justified, as the projection $p$ is injective on $\mathbf G(k).$ The lattice $\Gamma_V$ is a \textbf{congruence arithmetic lattice} and every congruence arithmetic lattice in $\PGL(2,\K)$ is, up to conjugation, of this form.  For the reader's convenience we sketch the proof of this fact. Let $\Gamma$ be a congruence arithmetic lattice of $\PGL(2,\K)$. By definition, it contains a principal congruence subgroup as a finite index subgroup. Principal congruence subgroups are always of the form $p(\Gamma_V)$ for $V=\mathbb A^\times(1+\frak n\O)/\mathbb A^\times$, where $\O$ is an order of $A$ and $\frak n$ is an ideal of $\O_k$. By \cite[Proposition 3.1.10]{Zimmer} for every $\gamma\in \Gamma$ the conjugation map $x\mapsto \gamma^{-1}x\gamma$ is defined over $k$, hence $\Gamma\subset \mathbf{G}(k)$\footnote{ We used here the fact that $\mathbf G(k)$ is an adjoint group.}. These two observations combined prove that $\Gamma=\Gamma_{V'}$ for some open subgroup $V'\subset \mathbf G(\mathbb A_f)$ containing $V$ as a finite index subgroup. We recall that a lattice $\Gamma$ is called \textbf{arithmetic} if it is commensurable with a congruence arithmetic lattice. As before, \cite[Proposition 3.1.10]{Zimmer} implies that every arithmetic lattice $\Gamma$ commensurable with $\Gamma_V$ is contained in $\mathbf G(k)$. An important upshot of this discussion is that: 
\begin{lemma} Every arithmetic lattice $\Gamma\subset \PGL(2,\K)$ is contained, up to conjugacy, in $\Gamma_U$ for some choice of admissible pair $k,A$ and a maximal open compact subgroup $U$ of $\mathbf G(\mathbb A_f)$.
\end{lemma}
We remark that this observation used the fact that $\PGL(2,\K)$ is adjoint, and in general the analogue is not true for non adjoint groups. The field $k$ is uniquely determined by $\Gamma_V$ because it is the \textbf{ invariant trace field } of $\Gamma_V$, defined as in \cite[Definition 3.3.6]{MaRe03}. 

We will need a description of the maximal compact subgroups of $\mathbf G(k_\frak p)$ at each finite place. Recall that $\mathbf G(k_\frak p)\simeq \PGL(2,k_p)$ if $\frak p\not\in \Ram A$ and $\PD^\times_\frak p$ otherwise. Up to conjugacy, the maximal compact subgroups of $\mathbf G(k_\frak p)$ are as follows
\begin{align}
\nonumber U_\frak p^0&:=\PGL(2,\ok_\frak p) & \textrm{ if } \frak p\not\in \Ram A,\\
\label{eq:MaxCompDefs}U_\frak p^1&:=\left\langle \begin{pmatrix} 0 & 1\\ \pi & 0\end{pmatrix}, \begin{pmatrix} \ok_\frak p^\times & \ok_\frak p\\ \frak p & \ok_\frak  p^\times\end{pmatrix} \right\rangle/ k_\frak p^\times &  \textrm{ if } \frak p\not\in \Ram A,\\
\nonumber U^2_\frak p&:=\PD^\times_\frak p & \textrm{ if }  \frak p\in \Ram A.
\end{align}
Any maximal compact subgroup $U$ of $\mathbf G(\mathbb A_f)$ is $\mathbf G(\mathbb A_f)$-conjugate to $\prod_{\frak p}U^{i_\frak p}_\frak p$ where $i_\frak p\in \{0,1,2\}$ and $i_\frak p=0$ for almost all primes $\frak p$. Throughout the paper we will work with a single maximal compact $U=\prod_{\frak p}U_\frak p$ and we will denote by $S$ the set of primes $\frak p$ where $U_\frak p$ is conjugate to $U_\frak p^1$. 
\begin{remark}
 The lattices $\Gamma_U$ with $U$ as above correspond to the lattices $\Gamma_{S,\O}$ in the notation from \cite[Definition 11.4.2]{MaRe03}. As proved in \cite[Theorem 11.4.3]{MaRe03} all maximal lattices in $\PGL(2,\K)$ are conjugate to some $\Gamma_U$, although not every lattice of this form must be maximal (c.f. \cite[p. 10, Remark]{Bor81}).
\end{remark}

\subsection{Volume conventions}\label{sec:VolumeConventions}
Let $X=\H^2$ if $\K=\R$ and $X=\H^3$ if $\K=\C$. We endow $X$ with a standard Riemannian volume form 
$$d\Vol=\frac{dxdy}{y^2} \textrm{ if } X=\H^2 \textrm{ and } d\Vol=\frac{dx_1dx_2dy}{y^2} \textrm{ if } X=\H^3.$$
Let $K$ be a maximal compact subgroup of $\PGL(2,\K)$ and identify $X=\PGL(2,\K)/K$. Let $\mu_K$ be the unique Haar probability measure on $K$. We equip $\PGL(2,\K)$ with 
the measure $d\mu_{\PGL(2,\K)}^{\st}=dk d\Vol$. In this way, when $\Gamma$ is a lattice of $\PGL(2,\K)$, we have $\mu_{\PGL(2,\K)}^{st}(\Gamma\bs\PGL(2,\K))=\Vol(\Gamma\bs X).$ We call $\mu_{\PGL(2,\K}^{\st}$ the \textbf{ standard measure } on $\PGL(2,\K)$ . To shorten notation we will drop the subscript $\PGL(2,\K)$ and just write $\mu^{\st}$. 

In a similar fashion we will define the standard measures of $\mathbf G(k_\nu)$ for any place $\nu\in \Sigma$. If $\mathbf  G(k_\nu)$ is compact we define $\mu^{\st}$ to be the unique Haar probability measure. That leaves just the case $\mathbf G(k_\nu)\simeq \PGL(2,k_\nu)$. If $k_\nu=\K$ we take the definition from previous paragraph. If $\nu=\frak p$, then the measure $\mu^{\st}$ will implicitly depend on the choice of a maximal compact subgroup $U_\frak p\subset \mathbf G(k_\frak p)$. We put $d\mu^{\st}_{\mathbf G(k_\frak p)}:=d\mu_{U_\frak p}d\mu_{\mathbf G(k_\frak p)/U_\frak p}$ where $\mu_{U_\frak p}$ is the unique Haar probability measure on $U_\frak p$ and $\mu_{\mathbf G(k_\frak p)/U_\frak p}$ is the counting measure on the coset space $\mathbf G(k_\frak p)/U_\frak p$. Finally, we define the standard measure $\mu^{\st}$ on $\mathbf G(\mathbb A)$, relative to maximal compact subgroup $U=\prod_{\frak p}U_\frak p$ by $\mu^{\st}_{\mathbf G(\mathbb A)}=\prod_{\nu\in \Sigma}\mu^{\st}_{\mathbf G(k_\nu)}.$ 
\begin{remark}
 The dependence on $U$ can be seen as a drawback of the definition of $\mu^{\st}$. We prefer to keep it that way because the group $U$ will be fixed throughout the paper and this choice of measure simplifies the trace formulas in Section \ref{sec:TraceFormulas} as well as the volume formulas in Section \ref{sec:Covolumes}.
\end{remark}

We will also need to define the standard measures for the centralizers of regular semisimple elements in $\mathbf G$. Let $\gamma$ be a regular seimisimple element of $\mathbf G(k)$ and let $x\in \gamma k^\times\subset A^\times.$ The centralizer of $x$ in $A$ is $k[x]$ which is either $k\times k$ or a quadratic field extension of $k$.  Then $\mathbf G_\gamma\simeq \mathbb G_m$ and $\Res_{k(x)/k}\mathbb G_m/\mathbb G_m$ respectively. We give a definition of standard measure valid for any algebraic torus $\mathbf T$ over $k$:

Let $\nu\in \Sigma$ be any place of $k$. The torus $\mathbf T(k_\nu)$ has a unique maximal compact subgroup $$\mathbf T(k_\nu)^b:=\{t\in \mathbf T(k_\nu)|\, |\xi(t)|_{\nu}=1 \textrm{ for every }\xi\in X^*(\mathbf T)_{k_\nu}\}.$$ Fix a basis $\xi_1,\ldots,\xi_m$ of the sub-module of $k_\nu$-rational characters of $\mathbf T$. The map 
$$\mathbf T(k_\nu)\ni t\mapsto (\log |\xi_1(t)|_\nu,\ldots,\log |\xi_m(t)|_\nu)\in \begin{cases}
                                                                          \R^m &\textrm{ if }\nu\in \Sigma_\infty,\\
                                                                          (\log q) \Z^m &\textrm{ if }\nu=\frak p\textrm{ with }\frak o_\frak p/\frak p\simeq \F_q.
                                                                         \end{cases}$$ induces an isomorphism $\mathbf T(k_\nu)/\mathbf T(k_\nu)^b$ with $\R^m, (\log q) \Z^m$ respectively. The standard measure $\mu^{\st}$ on $\mathbf T(k_\nu)$ is the gluing of Lebesgue, resp. counting measure on $\mathbf T(k_\nu)/\mathbf T(k_\nu)^b$ with the Haar probability measure on $\mathbf T(k_\nu)^b$.

The standard measure on $\mathbf T(\mathbb A)$ is given by $$\mu^{\st}_{\mathbf T(\mathbb A)}:=\prod_{\nu\in\Sigma}\mu^{\st}_{\mathbf T(k_\nu)}.$$ The definition of the standard measure ``commutes'' with Weil restriction, i.e. the natural isomorphism $\mathbf T(\mathbb A_k)\simeq \Res_{k/\Q}\mathbf T(\mathbb A_\Q)$ preserves the standard measures. 

One important feature of our definition of the standard measure on the torus $\mathbf T$ is that in the case when $\mathbf T$ is $k$-anisotropic we have $\Vol(\mathbf T\bs \mathbf T(\mathbb A))=\frac{h_\mathbf T R_\mathbf T}{w_\mathbf T}$ where $h_\mathbf T,R_\mathbf T$ are respectively the class number and the regulator of $\mathbf T$, as defined by Ono in \cite{Ono1}, and $w_\mathbf T$ is the size of the torsion subgroup of $\mathbf T(k)$.
\section{Bilu equidistribution and consequences}\label{sec:Bilu}
A celebrated theorem of Yuri Bilu \cite{Bilu} states that if we have a sequence of algebraic numbers $(\alpha_n)_{n\in \N}$ such that $\deg \alpha_n\to\infty$ and $h(\alpha_n)\to 0$, then the algebraic conjugates of $\alpha_n$ equidistribute on the unit circle. More precisely for every $f\in C_c(\C)$ we have 
$$\lim_{n\to\infty} \frac{1}{[\Q(\alpha_n):\Q]}\sum_{\sigma\in\Hom(\Q(\alpha_n),\C)}f(\sigma(\alpha_n))=\frac{1}{2\pi}\int_{0}^{2\pi}f(e^{i\theta})d\theta.$$
We will not be using Bilu's theorem directly but the phenomenon lying behind it is one of the main ingredients of this paper. It is precisely this equidistribution that allows us to get estimates that are uniform in the degree of the invariant trace field $k$. In the sequel we will be using following consequences of a quantitative version of Bilu equidistribution (Lemma \ref{lem:BiluMoments}), highly tailored towards our applications. 
Let $\alpha$ be an algebraic number. Throughout this section we will write $H=\Hom(\Q(\alpha),\C), H_1:=\{\sigma\in H|\, |\sigma(\alpha)|=1\}, H_2:=\{\sigma\in H|\, |\sigma(\alpha)|\neq 1\}$ and $d:=[\Q(\alpha):\Q]=|H|.$
\begin{lemma}\label{lem:BiluFinal}
Let $C>0, 1/2>\eta>0$ and let $\alpha$ be an algebraic integer such that $\mah(\alpha)\leq C+\eta[\Q(\alpha):\Q]$ and $|H_2|\leq 8$. Then 
\begin{enumerate}
 \item $|\sum_{\sigma\in H_1}\sigma(\alpha)|\ll \eta^{1/2}d+o_{C,\eta}(d).$
 \item $\log |\N_{\Q(\alpha)/\Q}(1-\alpha)|\ll (-\eta^{1/2}\log\eta )d+o_{C,\eta}(d).$
\end{enumerate}
\end{lemma}
\begin{remark}The constant $8$ appearing in the statement is just what we will need in the paper but the same inequalities hold for any uniform bound on $|H_2|$. \end{remark}
\begin{lemma}\label{lem:BiluZeta}
For any $\sigma>1$ there exist decreasing functions $\kappa_1, \kappa_{2,\sigma}:\R_{>0}\to\R_{>0}, i=1,2$ with $\lim_{t\to 0}\kappa_1(t)=\lim_{t\to 0}\kappa_{2,\sigma}(t)=0$ with the following property. Let $C>0,1/2>\eta>0$, let $\alpha$ be an algebraic integer such that $m(\alpha)\leq \eta d+C$ and $|H_2|\leq 8$. Then
\begin{enumerate}
 \item For any $X>0$ we have $\pi_{\Q(\alpha)}(X)\ll_{X} \kappa_1(\eta)d+o_{X,C,\eta}(d).$
 \item For any $s=\sigma +it, \sigma>1$  we have $\log |\zeta_{\Q(\alpha)}(s)|\leq \kappa_{2,\sigma}(\eta)d+o_{C,\eta}(d).$
\end{enumerate}
\end{lemma}

The conclusion of Lemma \ref{lem:BiluZeta} is non-trivial only when $d$ tends to infinity. First we show how Lemma \ref{lem:BiluZeta} follows from Lemma \ref{lem:BiluFinal}. Next, we prove some general estimates in the spirit of Bilu equidistribution, from which we will deduce Lemma \ref{lem:BiluFinal} as a simple corollary. 

\begin{proof}[Lemma \ref{lem:BiluZeta}]
There is only finitely many rational primes $p\leq X$ so it will be enough to prove that the number $n_p$ of primes $\frak p$ dividing $p$ with $\N(\frak p)\leq X$ satisfies $n_p\ll_{X}\kappa_1(\eta)d+o_{C,\eta}(d)$. Let $L$ be the least common multiple of $1,2,\ldots, \lfloor \frac{\log X}{\log p}\rfloor.$ 
 Every prime $\frak p|p$ with $\N(\frak p)\leq X$ divides $\alpha^{p^L}-\alpha$ so $$n_p\log p\leq \sum_{\substack{\frak p|p\\ \N(\frak p)\leq X}}\log \N(\frak p)\leq \log |\N_{\Q(\alpha)/\Q}(\alpha)|+\log |\N_{\Q(\alpha)/\Q}(\alpha^{p^L-1}-1)|.$$ We have $[\Q(\alpha):\Q(\alpha^{p^L-1})]\leq p^L$ and $\m(\alpha^{p^L-1})\leq p^L\m(\alpha)$. Put $C'=p^LC, \eta'=p^{2L}\eta$. We have 
$\m(\alpha^{p^L-1})\leq C'+\eta'[\Q(\alpha^{p^L-1}):\Q]$. If $p^{-2L}/2\leq \eta$, then the trivial bound $n_p\leq d\ll_X (-\eta^{1/2}\log \eta)d$ corresponds to $\kappa_1(\eta)=-\eta^{1/2}\log \eta$. Therefore, it is enough to treat the case $0<\eta<p^{-2L}/2$, that is $0<\eta'<1/2$. We have 
\begin{align*} \log |\N_{\Q(\alpha)/\Q}(\alpha)|&\leq \eta d+ C. 
\end{align*}
Lemma \ref{lem:BiluFinal} applied to $\alpha^{p^L-1}$ yields
\begin{align*}
\log |\N_{\Q(\alpha)/\Q}(\alpha^{p^L-1}-1)| &\ll [\Q(\alpha):\Q(\alpha^{p^L-1})]\left((-\eta'^{1/2}\log(\eta'))[\Q(\alpha^{p^L-1}):\Q]+o_{\eta',C'}([\Q(\alpha^{p^L-1}):\Q])\right)
\end{align*}
Therefore
\begin{align*}
n_p\log p\leq \log |\N_{\Q(\alpha)/\Q}(\alpha)|+\log |\N_{\Q(\alpha):\Q}(\alpha^{p^L-1}-1)|&\ll_X (-\eta^{1/2}\log\eta) d+o_{C,\eta}(d).
\end{align*}
This proves the first inequality. 

Write $C_1(X)$ for the implicit constant from the first point. For any $M>0$ we have
\begin{align*}
\log |\zeta_{\Q(\alpha)}(\sigma+it)|&\leq \log|\zeta_{\Q(\alpha)}(\sigma)|\leq 2\sum_\frak p \frac{1}{N(\frak p)^\sigma}\leq 2\sum_{N(\frak p)\leq M}\frac{1}{N(\frak p)^\sigma}+ 2d\sum_{M< p}\frac{1}{p^\sigma}\\
&\leq 2C_1(M)\left(\kappa_1(\eta)d+o_{M,C,\eta}(d)\right) +\frac{2M^{1-\sigma}d}{\sigma-1}.
\end{align*}
There is a function $M_0(\eta)$ with $M_0(\eta)\to \infty$ as $\eta\to 0$ such that $C_1(M_0(\eta))\kappa_1(\eta)\to 0$ as $\eta\to 0$. Put $\kappa_{2,\sigma}(\eta):=2C_1(M_0(\eta))\kappa_1(\eta)+2M_0(\eta)^{1-\sigma}(\sigma-1)^{-1}.$ Choosing $M=M_0(\eta)$ we get 
$$\log |\zeta_{\Q(\alpha)}(\sigma+it)|\leq \kappa_{2,\sigma}(\eta)d+o_{C,\eta}(d).$$ Second point is proved.
\end{proof}
We now prove a sequence of statements that will imply Lemma \ref{lem:BiluFinal}.
\begin{lemma}\label{lem:BiluMoments}
Let $\alpha$ be an algebraic integer. Then for every $0<t<1$
$$\sum_{n\geq 1}\frac{e^{-nt}}{n}\left|\sum_{\sigma\in H_1}\sigma(\alpha)^n\right|^2 \leq d\left(\frac{td}{2}+2\mah(\alpha)-\log t+ 2|H_2|\log 2+\frac{1}{2}\right).$$
\end{lemma}
\begin{proof}
 The discriminant of the minimal polynomial of $\alpha$ is a non zero integer. Hence
\begin{align}\nonumber \sum_{\sigma\neq\sigma'\in H}\log|\sigma(\alpha)-\sigma'(\alpha)|&\geq 0\\
 \sum_{\sigma\neq\sigma'\in H_1}\log|\sigma(\alpha)-\sigma'(\alpha)|&\geq -2\sum_{\substack{\sigma\in H_1\\ \nonumber \sigma'\in H_2}}\log|\sigma(\alpha)-\sigma'(\alpha)|-\sum_{\sigma\neq\sigma'\in H_2}\log|\sigma(\alpha)-\sigma'(\alpha)|\\ \label{eq:BiluMomentsEq1}
&\geq -2|H|\sum_{\sigma'\in H_2}(\log^+|\sigma'(\alpha)|+\log 2)=-2|H|(\mah(\alpha)+|H_2|\log 2).
\end{align}
Let $0<t<1$. Using Lemma \ref{lem:LogTrick} and the identity $|1-\sigma(\alpha)\overline{\sigma'(\alpha)}|=|\sigma(\alpha)-\overline{\sigma'(\alpha)}|$ for $\sigma,\sigma'\in H_1$ we get
\begin{align*}\sum_{\sigma\neq \sigma'\in H_1}\log|1-e^{-t}\sigma(\alpha)\overline{\sigma'(\alpha)}|&\geq -{|H_1|\choose 2}t+\sum_{\sigma\neq\sigma'\in H_1}\log|\sigma(\alpha)-\sigma'(\alpha)|.\\ 
\end{align*}
We have  $\log|1-e^{-t}|\geq \log t-\frac{1}{2}$ for $0<t<1$, so by (\ref{eq:BiluMomentsEq1}) and the previous inequality we get
\begin{align*}\sum_{\sigma,\sigma'\in H_1}\log|1-e^{-t}\sigma(\alpha)\overline{\sigma'(\alpha)}|&\geq -|H|\left(\frac{t|H|}{2}+2m(\alpha)+2|H_2|\log 2-\log t+\frac{1}{2}\right).
\end{align*}
It remains to expand the left hand side as a power series in $e^{-t}$ 
$$-\sum_{\sigma,\sigma'\in H_1}\log|1-e^{-t}\sigma(\alpha)\overline{\sigma'(\alpha)}|=\frac{1}{2}\sum_{\sigma,\sigma'\in H_1}\sum_{n\neq 0}\frac{e^{-t|n|}}{|n|}\sigma(\alpha)^n\overline{\sigma'(\alpha)^n}=\sum_{n=1}^\infty \frac{e^{-nt}}{n}\left|\sum_{\sigma\in H_1}\sigma(\alpha)^n\right|^2.$$ The lemma follows upon noting that $|H|=d.$
\end{proof}
\begin{corollary}\label{col:BiluTrace}
$$\left|\sum_{\sigma\in H_1}\sigma(\alpha)\right|\leq e^{1/2}d^{1/2}\left(3\mah(\alpha)+\log d-\log^+\mah(\alpha)+2|H_2|\log 2+1\right)^{1/2}.$$
\end{corollary}
\begin{proof}
In Lemma \ref{lem:BiluMoments} take $t=d^{-1}$ if $\mah(\alpha)<1$, $t=\frac{\mah(\alpha)}{d}$ if $1\leq \mah(\alpha)<d$ and $t=1$ otherwise. 
\end{proof}

\begin{lemma}\label{lem:BiluNorm}
Let $\alpha$ be an algebraic integer. Then,
 $$\log|\N_{\Q(\alpha)/\Q}(1-\alpha)|\leq \begin{cases}
                                           |H_2|+2\mah(\alpha)+d^{1/2}(\frac{1}{2}+\log^+\frac{d}{\mah(\alpha)})^{1/2} & \\ (3\mah(\alpha)+\log^+\frac{d}{\mah(\alpha)}+2|H_2|\log 2+\frac{1}{2})^{1/2} & \textrm{ if } \mah(\alpha)\geq 1\\
                                           |H_2|+2+d^{1/2}(3+\log d+2|H_2|\log 2) & \textrm{ otherwise.}
                                          \end{cases}
$$
\end{lemma}
\begin{proof} We keep notation from the proof of Lemma \ref{lem:BiluMoments}
 \begin{align*}\log|\N_{\Q(\alpha)/\Q}(1-\alpha)|&=\sum_{\sigma\in H}\log|1-\sigma(\alpha)|\leq \sum_{\sigma\in H_2}(\log^+|\sigma(\alpha)|+\log 2)+ \sum_{\sigma\in H_1}\log|1-\sigma(\alpha)|\\
 &\leq |H_2|+\mah(\alpha)+\sum_{\sigma\in H_1}\log|1-\sigma(\alpha)|.\end{align*}

Let $0<t<1$. By Lemma \ref{lem:LogTrick} we can bound the last expression from above by 
\begin{align*}
 |H_2|+\mah(\alpha)+\frac{|H_1|t}{2}+\sum_{\sigma\in H_1}\log|1-e^{-t}\sigma(\alpha)|.
\end{align*}
We expand the last sum as a series in $e^{-t}$, use Cauchy-Schwartz, then Lemma \ref{lem:BiluMoments} and the inequality $-\log(1-e^{-t})\leq \frac{1}{2}-\log t$: 
\begin{align*}\left|\sum_{\sigma\in H_1}\log|1-e^{-t}\sigma(\alpha)|\right| &\leq \left(\sum_{n\geq 1}\frac{e^{-nt}}{n}\right)^{1/2}\left(\sum_{n\geq 1}\frac{e^{-nt}}{n}\left|\sum_{\sigma\in H_1}\sigma(\alpha)^n\right|^2\right)^{1/2}\\
&\leq (\frac{1}{2}-\log t)^{1/2}|H|^{1/2}\left(\frac{t|H|}{2}+2\mah(\alpha)-\log t+ 2|H_2|\log 2+\frac{1}{2}\right)^{1/2}.
\end{align*}
To get the Lemma we put $t=1/|H|$ if $\mah(\alpha)<1$, $t=\mah(\alpha)/|H|$ if $1\leq\mah(\alpha)<|H|$ and $t=1$ otherwise. We arrive at the inequalities from the lemma by  simple algebraic manipulations, case by case.
\end{proof}
Lemma \ref{lem:BiluFinal} follows now from Corollary \ref{col:BiluTrace} and Lemma \ref{lem:BiluNorm}.

\begin{lemma}\label{lem:LogTrick}
For $0<t$ and $|z|=1, z\neq 1$ we have $\log|1-e^{-t}z|\geq \log|1-z|-\frac{t}{2}$. 
\end{lemma}
\begin{proof}
\begin{align*}
\log|1-e^{-t}z|-\log|1-z|&=\int_0^t\frac{\p}{\p s}\log|1-e^{-s}z|ds=\frac{1}{2}\int_0^t\left(\sum_{n\neq 0}e^{-|n|s}z^n\right)ds\\ &=-\frac{t}{2}+\frac{1}{2}\int_0^t \left(\sum_{n=-\infty}^\infty e^{-|n|s}z^n\right)ds=-\frac{t}{2}+\frac{1}{2}\int_0^t\frac{1-e^{-2s}}{|1-e^{-s}z|^2}ds>-\frac{t}{2}.
\end{align*}
\end{proof}

\section{Trace formula for congruence lattices}\label{sec:TraceFormulas}
Let us fix $\K=\R,\C$, an admissible pair $k,A$ and let $s=r_1+r_2$ be the number of infinite places of $k$. Recall we write $G=\PGL(2,\K)$ and $\mathbf G=\PA^\times$. We fix a maximal compact subgroup $U\subset \mathbf G(\mathbb A_f)$ and reserve letter $V$ for open subgroups of $U$. We endow $\mathbf G(\mathbb A)$ with the standard measure $\mu^{\st}$, defined as in Section \ref{sec:VolumeConventions}. Let $W$ be the set of non-torsion regular semisimple elements of $G$. The goal of this part is to estimate, for $\Gamma=\Gamma_V$, the sum $\sum_{[\gamma]_{\Gamma}\subset W}\Vol(\Gamma_\gamma\bs G_\gamma)\O_\gamma(f)$ by an expression akin to the geometric side of the adelic trace formula where the summation is taken over the conjugacy classes in $\mathbf G(k)$ rather than in $\Gamma$. The regular semisimple conjugacy classes  in $\mathbf G(k)$ can be then easily parametrized using Skolem-Noether theorem, which we do in Section \ref{sec:ConjClasses}.
\subsection{Class groups}
For any open compact subgroup $V$ of $\mathbf G(\mathbb A_f)$ we define the \textbf{class group} $\cl(V):=\mathbf G(k)\bs \mathbf G(\mathbb A_f)/V.$ It will play a prominent role in the trace formulas developed in the Section \ref{sec:CongTraceForm}. A priori it is not clear why this should be a group. We will show this as a consequence of the strong approximation property for the group $A^1:=\SL_1(A)$. Let $k_A^\times:=\{x\in k^\times| \, (x)_\nu>0 \textrm{ for } \nu\in \Ram_\infty A\}.$
\begin{lemma}\label{lem:ClassGroup}
The double quotient $\cl(V)$ is a finite abelian group of exponent $2$. Moreover, if $U=\prod_{\frak p}U_\frak p$ is a maximal compact subgroup of $\mathbf G(\mathbb A_f)$ and $S=\{\frak p| U_\frak p\sim U_\frak p^1\}$, then $$\cl(U)\simeq \mathbb A_f^\times / (\mathbb A_f^\times)^2k_A^\times  \prod_{\frak p\not\in S\cup\Ram_f A}\frak o_\frak p^\times \prod_{\frak p\in S\cup\Ram_f A}k_\frak p^\times.$$
\end{lemma}
\begin{proof}
Write $\tilde V$ for the preimage of $V$ in $A^\times(\mathbb A_f)$. 
Note that $A^1(k), A^1(\mathbb A_f)$ are the kernels of the norm map $\n$ restricted to $A^\times(k), A^\times(\mathbb A_f)$ respectively. 
Since $A$ is admissible we have $\nu_1\not\in \Ram A$. By the strong approximation property \cite[Theorem 7.7.5]{MaRe03} the group $A^1(k)$ is dense in $A^1(\mathbb A_f)$. Therefore 
\begin{equation}\label{eq:ClassNumber1}\cl(V)=\mathbf G(k)\bs \mathbf G(\mathbb A_f)/V\simeq A^\times(k)\bs A^\times(\mathbb A_f)/\tilde V\simeq \n(A^\times(k))\bs \n(A^\times(\mathbb A_f))/\n(\tilde V).\end{equation}
By Eichler's theorem on norms \cite[Theorem 7.4.1]{MaRe03} $\n(A^\times(k))=k_A^\times.$ Since $\n(\GL(2,k_\frak p))=\n(D_\frak p)=k_\frak p^\times$ we get $\n(A^\times(\mathbb A_f))=\mathbb A_f^\times.$ Finally $\n(\mathbb A_f^\times)\subset \n(\tilde V)$ so $\cl(V)$ is a quotient of $k^\times_A\bs \mathbb A_f^\times/ (\mathbb A_f^\times)^2$ which is an abelian group of exponent $2$. Finiteness will follow from the second assertion, as we can take $V\subset U$ with $[U:V]<\infty$. Let $U=\prod_{\frak p}U_\frak p$ be as in the statement of the lemma. We have $\n(\tilde U_\frak p)=k_\frak p^\times$ if $\frak p\in S\cup\Ram_f A$ and $\n(\tilde U_\frak p)=(k_\frak p^\times)^2\frak o_\frak p^\times$ otherwise. Therefore 
$$\n(\tilde U)=(\mathbb A_f^\times)^2 \prod_{\frak p\not\in S\cup\Ram_f A}\frak o_\frak p^\times \prod_{\frak p\in S\cup\Ram_f A}k_\frak p^\times.$$ By (\ref{eq:ClassNumber1}) we have $\cl(U)=\n(A^\times)\bs \n(A^\times(\mathbb A_f))/\n(\tilde U))=k_A^\times \bs \mathbb A_f^\times/ \n(\tilde U)$. Lemma follows now from the formula for $\n(\tilde U)$. 
\end{proof}
For later use we show: 
\begin{lemma}\label{lem:SecondIneq}
Let $\gamma\in \mathbf G(k)$ be a regular semisimple non-torsion element\footnote{The statement is actually true all regular semisimple elements, but in the torsion case the centralizer might be bigger than what we wrote in the proof.}. Let $V$ be an open subgroup of $\mathbf G(\mathbb A_f)$. Then $|\mathbf G(\mathbb A)/\mathbf G(k)V\mathbf G_\gamma(\mathbb A)|\leq 2$.
\end{lemma}
\begin{proof}
Let $l:=k(\gamma)\subset A$. $\mathbf G_\gamma(k)$ in $\mathbf G(k)$ is the image of $l^\times=k(\gamma)^\times$ modulo center. Like in the proof of Lemma \ref{lem:ClassGroup} the strong approximation theorem yields a bijective map 
$$\n\colon \mathbf G(\mathbb A)/\mathbf G(k)V\mathbf G_\gamma(\mathbb A)\to \n(A^\times(\mathbb A))/ \n(A^\times(k))\n(\mathbb A(\gamma)^\times)\n(\tilde V).$$
By the weak approximation and Eichler norm theorem \cite[Theorem 7.7.5]{MaRe03} $\n(A^\times(\mathbb A))/ \n(A^\times(k))\simeq \mathbb A^\times/k^\times$. On the other hand $\n(\mathbb A(\gamma)^\times)=\N_{l/k}((\mathbb A\otimes_k l)^\times)$ if $l/k$ is quadratic or $\n(\mathbb A(\gamma)^\times)=\mathbb A^\times$ if $l=k\times k$ (in which case $A$ must be split). By the second inequality of global class field theory \cite[5.1(a)]{MilneCFT} $|\mathbb A^\times/k^\times \N_{l/k}((\mathbb A\otimes_k l)^\times)|\leq 2$. The lemma follows.
\end{proof}

\subsection{Trace formula}\label{sec:CongTraceForm}
Let $f\in C_c(G)$, let us fix an open subgroup $V$ of $U$ and let $\Gamma=\Gamma_V$. We put $$f_\mathbb A:=f\times \mathbf 1_{\SO(3)}^{s-1}\times[U:V]\mathbf 1_V\in C_c(\mathbf G(\mathbb A)).$$ 
We identify the set of unitary characters $\chi\colon\mathbf G(\mathbb A)\to \C^\times$ that are trivial on $\mathbf G(k)V\mathbf G(\mathbb A_\infty)$ with $\widehat{\cl}(V)$. 
\begin{lemma}\label{lem:TraceForm0}
Suppose either that $\mathbf G$ is $k$-anisotropic or that $f$ is supported on set of regular semisimple elements of $G$. Then, 
$$\int_{\Gamma_V\bs G} \left(\sum_{\gamma\in \Gamma_V}f(g^{-1}\gamma g)\right)dg=\frac{1}{|\cl(V)|}\sum_{\chi\in \widehat{\cl}(V)}\int_{\mathbf G(k)\bs \mathbf G(\mathbb A)} \left(\sum_{\gamma\in \mathbf G(k)}\chi(g)f_\mathbb A(g^{-1}\gamma g)\right)dg.$$
\end{lemma}
\begin{proof} The right hand side equals
\begin{align*}\int_{\mathbf G(k)\bs \mathbf G(k)V\mathbf G(\mathbb A_\infty)}\left(\sum_{\gamma\in \mathbf G(k)}f_\mathbb A(g^{-1}\gamma g)\right)dg=& \int_{\mathbf G(k)\cap V\mathbf G(\mathbb A_\infty)\bs V\mathbf G(\mathbb A_\infty)} \left(\sum_{\gamma\in \mathbf G(k)\cap \mathbf G(\mathbb A_\infty)V} f_\mathbb A(g^{-1}\gamma g)\right)dg\\=& \mu^{\rm st}_{\mathbf G(\mathbb A_f)}(V)[U:V]\int_{\Gamma_V\bs G} \left(\sum_{\gamma\in \Gamma_V}f(g^{-1}\gamma g)\right)dg.\end{align*} The lemma follows since the standard measure is chosen so that $\mu^{\rm st}_{\mathbf G(\mathbb A_f)}(U)=1$.
\end{proof}
Arguing as in \cite[p. 7-9]{Clay03} we get
\begin{equation}\label{eq:TraceEq1}
 \int_{\mathbf G(k)\bs \mathbf G(\mathbb A)} \left(\sum_{\gamma\in \mathbf G(k)}\chi(g)f_\mathbb A(g^{-1}\gamma g)\right)dg=\sum_{[\gamma]\subset \mathbf G(k)}\Vol(\mathbf G_\gamma(k)\bs \mathbf G_\gamma(\mathbb A))\sum_{\substack{\chi\in \widehat{\cl}(V)\\ \chi|_{\mathbf G_\gamma(\mathbb A)}=1}}\O_\gamma^\chi(f_\mathbb A),
\end{equation}
where  $\O_\gamma^\chi(f_\mathbb A):=\int_{\mathbf G_\gamma(\mathbb A)\bs \mathbf G(\mathbb A)}\chi(g)f_\mathbb A(g^{-1}\gamma g)dg$. Using Lemmas \ref{lem:SecondIneq} and \ref{lem:TraceForm0} we now deduce:
\begin{lemma}\label{lem:TFTraceFormula} Let $W$ be the set of non-torsion regular semisimple elements of $G$. For any non-negative function $f\in C_c(G)$ we have
\begin{equation}
\sum_{[\gamma]\subset \Gamma\cap W}\Vol(\Gamma_\gamma\bs G_\gamma)\O_\gamma(f)\leq \frac{2}{|\cl(V)|}\sum_{[\gamma]\subset \mathbf G(k)\cap W}\Vol(\mathbf G_\gamma(k)\bs \mathbf G_\gamma(\mathbb A))\O_\gamma(f_\mathbb A).
\end{equation}
\end{lemma}
\begin{proof}
By Lemma \ref{lem:TraceForm0} and (\ref{eq:TraceEq1}) for every $h\in C_c(G)$ supported on the set of regular semisimple elements we have 
\begin{equation}\label{eq:Trace1}\sum_{[\gamma]\subset \Gamma}\Vol(\Gamma_\gamma\bs G_\gamma)\O_\gamma(h)=\frac{1}{|\cl(V)|}\sum_{[\gamma]\subset \mathbf G(k)}\Vol(\mathbf G_\gamma(k)\bs \mathbf G_\gamma(\mathbb A))\sum_{\substack{\chi\in \widehat{\cl}(V)\\ \chi|_{\mathbf G_\gamma(\mathbb A)}=1}}\O^\chi_\gamma(h_\mathbb A).\end{equation}
 Let $\gamma_0\in W\cap \Gamma$ and take a sequence $(U_n)_{n\in \N}\subset W$ of shrinking open sets with $\bigcap_{n=0}^\infty U_n=[\gamma_0]_G$. For each $n\in \N$ choose a continuous function $\phi_n\colon G\to \R_{\geq 0}$ with $\supp \phi_n\subset U_n$ and $\phi_n|\gamma_0^G=1$. Put $h_n:=\phi_nf$. The conjugacy classes classes of regular semisimple elements of $\Gamma$ are separated in $G$ so for $n$ big enough $[\gamma_0]_G$ is the only class of $\Gamma$ intersecting $\supp h_n$. Similarly, for $n$ big enough $[\gamma_0]_{\mathbf G(\mathbb A)}$ is the only conjugacy class of $\mathbf G(k)$ intersecting $\supp (h_n)_\mathbb A$. Hence, if we apply (\ref{eq:Trace1}) to $h_n$ and take a limit as $n\to\infty$, we get
$$\Vol(\Gamma_{\gamma_0}\bs G_{\gamma_0})\O_{\gamma_0}(f)=\frac{1}{|\cl(V)|}\Vol(\mathbf G_{\gamma_0}(k)\bs \mathbf G_{\gamma_0}(\mathbb A))\sum_{\substack{\chi\in \widehat{\cl}(V)\\ \chi|_{\mathbf G_{\gamma_0}(\mathbb A)}=1}}\O^\chi_{\gamma_0}(f_\mathbb A).$$
The number of characters $\chi\in \widehat{\cl(V)}$ that are trivial on $\mathbf G_{\gamma_0}(k)$ is at most $2$ by Lemma \ref{lem:SecondIneq} so
$$\Vol(\Gamma_{\gamma_0}\bs G_{\gamma_0})\O_{\gamma_0}(f)\leq \frac{2}{|\cl(V)|}\Vol(\mathbf G_{\gamma_0}(k)\bs \mathbf G_{\gamma_0}(\mathbb A))\O_{\gamma_0}(f_\mathbb A).$$
We deduce the lemma by taking the sum over all $[\gamma_0]\subset \Gamma\cap W$. 
\end{proof}
\begin{remark} 
Our reasoning in this section can be compared to the stabilization procedure employed by Langlands and Labesse in \cite{LL}.
\end{remark}
\section{Non-Archimedean estimates}\label{sec:NAEstimates}  The orbital integrals appearing in Lemma \ref{lem:TFTraceFormula} decompose as $\O_\gamma(f_\mathbb A)=[U:V]\O_\gamma(f)\O_\gamma(\mathbf 1_V)$. This is an immediate consequence of the definition of $f_\mathbb A$ as a product of the Archimedean part $f\times \mathbf 1_{\SO(3)}^{s-1}$ and the non-Archimedean part $[U:V]\mathbf 1_V$. In this section we estimate $\O_\gamma(\mathbf 1_V)$ in terms of the norm of Weyl discriminant $\N_{k/\Q}(\Delta(\gamma))$, index $[U:V]$ and certain arithmetic invariants of $\mathbf G$. This can be compared with \cite{FiLa14}, where a bound of the form  $\O_\gamma(\mathbf 1_\mathbf V)\ll [U:V]^{-\delta}$ was obtained for a fixed  $U\subset\mathbf G(\mathbb A_f),1\neq \gamma\in \mathbf G(k)$ and the family of open subgroups $V$ of $U$. The proof  in \cite{FiLa14} proceeds via a beautiful structure theorem for open subgroups of $U$ (an ``approximation principle'' \cite[Thm 2.2]{FiLa14}). The implicit constants in \cite{FiLa14} depend on $\mathbf G, U$ so unfortunately we cannot use these bounds without working out this dependence explicitly. Instead, we take another approach relying on character bounds for irreducible representations of $U_\frak p$ and the twist representation zeta function of $U$. Our estimate takes the following form:
\begin{proposition}\label{prop:NonArchBound}
Let $\gamma\in \mathbf G(k)$ be regular semisimple. Equip $\mathbf G(\mathbb A_f)$ with the standard measure relative to $U$ (see Section \ref{sec:VolumeConventions}). Then for every $m\geq 6$ 
$$\O_\gamma(\mathbf 1_V)\leq |\N_{k/\Q}(\Delta(\gamma))|^{7}J_1[U:V]^{-1/2m} J_2^{1/2m}J_3^{1/2m}$$ where $J_1=2^{3\pi_k(65)+|\Ram_f A|/2+|S|/2}, J_2=2^{2[k:\Q]+|S|+|\Ram_f A|}|\cl(k)/\cl(k)^2|\frac{|\cl(V)|}{|\cl(U)|}$ and  \\ $ J_3=\zeta_k(2)^{24} 162^{|S|}2^{|\Ram_f A|}\prod_{\frak p\in S\cup\Ram_f A}(N(\frak p)+1)^3$.
\end{proposition}
We recall that $\pi_k(X)$ is the number of primes $\frak p$ of $\O_k$ with $\N(\frak p)\leq X$ and $S$ is the set of primes $\frak p$ where $U_\frak p\simeq U_\frak p^1$ (c.f. Section \ref{sec:CongruenceLattices}). With our methods, the exponent $7$ next to $\N_{k/\Q}(\Delta(\gamma))$ can be brought arbitrarily close to $3/2$, at the cost of increasing other factors. Later, we are going to show that in our application the norm $\N_{k/\Q}(\Delta(\gamma))$ is quite small so we can afford a large exponent. The constant $J_3$ is very far from being optimal. However, do not need to optimize it because we can reduce its impact by taking sufficiently big $m$.
We prove the Proposition \ref{prop:NonArchBound} in Section \ref{sec:ProofPropNE}, where we introduce the main ingredients of the argument as black boxes. These are Lemmas \ref{lem:CharacterOI},\ref{lem:RepzetaBound}, and \ref{lem:AbelianizationBound}. We prove them in Sections \ref{sec:CharacterOI},\ref{sec:RepZetaBound} and \ref{sec:AbelianizationBound} respectively.
\subsection{Proposition \ref{prop:NonArchBound}}\label{sec:ProofPropNE}
First, we note that $\O_\gamma(\mathbf 1_V)$ remains unchanged when we conjugate $V$ by elements of $\mathbf G(\mathbb A_f)$. Therefore 
$$\O_\gamma(\mathbf 1_V)=\int_U\O_\gamma(\mathbf 1_V^u)du=\frac{1}{[U:V]}\O_\gamma(\chi_{\Ind_V^U 1}),$$ where $\chi_V^U 1$ is the character of the induced representation $\Ind_V^U1$. For any representation $(\rho,H_\rho)$ of $U$ write $\dim \rho^V$ for the dimension of $V$-fixed subspace $H_\rho^V$. By the Frobenius reciprocity we have
\begin{equation}\label{eq:FrobeniusRec}\O_\gamma(\chi_{\Ind_V^U 1})=\sum_{\rho\in \Irr U}\dim \rho^V \O_\gamma(\chi_{\rho}).\end{equation} This quantity will be compared with $[U:V]=\sum_{\rho\in \Irr U}\dim \rho^V\dim \rho$. Working with characters of irreducible representations of $U$ has the following advantage: they always decompose as product of characters of irreducible representations of local groups $U_\frak p$. This will allow us to easily deduce the global bounds from the local ones. We will show:
\begin{lemma}\label{lem:CharacterOI}
For every irreducible representation $\rho\in \Irr U$ we have
$$\O_\gamma(|\chi_\rho|)\leq |\N_{k/\Q}(\Delta(\gamma))|^7 J_1 (\dim \rho)^{1/2},$$ where $J_1=2^{3\pi_k(65)+|\Ram_f A|/2+|S|/2}$.
\end{lemma}
In order to show that $\O_\gamma(\chi_{\Ind_V^U1})$ is much smaller than $[U:V]=\sum_{\rho\in \Irr U}\dim \rho^V \dim \rho$ we will need some information about the growth of dimensions of representations of $U$. Indeed, after we plug Lemma \ref{lem:CharacterOI} into the right hand side of (\ref{eq:FrobeniusRec}) we see that the savings come from high dimensional representations. This heuristic will be made precise using the twist representation zeta function of $U$. 
\begin{definition}\label{def:TwistRepZeta}
Let $K$ be a profinite group. For two representations $\rho,\rho'\in \Irr K$ we write $\rho\sim \rho'$ if there exists a one dimensional character $\theta$ such that $\rho\simeq \rho'\otimes\theta$. We define the \textbf{ twist representation zeta function } $\zeta_K^*(s)$ as the formal Dirichlet series:
$$\zeta_K^*(s):=\sum_{\rho\in\Irr K/\sim}\frac{1}{(\dim\rho)^s}.$$
\end{definition}
For more background on twist zeta functions and their older cousins, ordinary representation zeta functions, we refer to \cite{HasaStasinski2018}. We prove an explicit upper bound on $\zeta^*_U(s)$, which takes the form 
\begin{lemma}\label{lem:RepzetaBound} For $s\geq 4$ we have  $$\zeta_U^*(s)\leq J_3:=\zeta_k(2)^{24} 162^{|S|}2^{|\Ram_f A|}\prod_{\frak p\in S\cup\Ram_f A}(N(\frak p)+1)^3.$$ \end{lemma} The bound above is far from optimal\footnote{In the previous version of the manuscript we proved a bound $\zeta_U^*(7)\leq\zeta_k(2)\prod_{\frak p\in S\cup \Ram_f A}(N(\frak p)+1)$.}, but it can be deduced fairly quickly from the existing literature. We do not optimize the bound because what we shall see in final bound is $J_3^{1/2m}$ and by taking big $m$ we can offset this error. 

The last ingredient needed to prove Proposition \ref{prop:NonArchBound} is an upper bound on the index of the image of $V$ in the abelianization of $U$. Write $j\colon U\to U^{\rm ab}$ for the abelianization map. We have
\begin{lemma}\label{lem:AbelianizationBound}
$$[j(U):j(V)]\leq J_2:= 2^{2[k:\Q]+|S|+|\Ram_f A|}|\cl(k)/\cl(k)^2|\frac{|\cl(V)|}{|\cl(U)|}.$$
\end{lemma}

With Lemmas \ref{lem:CharacterOI},\ref{lem:RepzetaBound} and \ref{lem:AbelianizationBound} at our disposal we can now bound $\O_\gamma(\chi_{\Ind_V^U1}).$ Choose a normal open subgroup $N\triangleleft U$ such that $N\subset V$. All representations $\rho$ contributing to the formula (\ref{eq:FrobeniusRec}) factor through $U/N$. Write $\overline U=U/N, \overline V=V/N$ and identify the subset of irreducible representations of $U$ that factor through $\overline{U}$ with $\Irr\overline U$. For each $\rho\in \Irr \overline U$ let $a_\rho:=\{\theta\in \Irr\overline U| \rho\otimes\theta\simeq \rho\}.$ The reason why we pass to $\overline U$ is purely technical: we need to ensure that $a_\rho$ is finite. Write $\rho\sim \rho', \rho,\rho'\in \Irr \overline U$ if $\rho'\simeq \theta\otimes \rho$ for some one dimensional character $\theta$ of $\overline U$. 

\begin{align*}
\O_\gamma(\chi_{\Ind_V^U1})=\sum_{\rho\in \Irr \overline U}\dim \rho^{\overline V} \O_\gamma(\chi_{\rho})=&\sum_{\rho\in \Irr \overline U/\sim}\frac{1}{a_\rho}\sum_{\dim \theta=1}\dim (\rho\theta)^{\overline V} \O_\gamma(\chi_\rho \theta)\\
\leq&\sum_{\rho\in \Irr \overline U/\sim}\frac{\O_\gamma(|\chi_\rho|)}{a_\rho}\sum_{\theta\in \Irr\overline U^{\rm ab}}\dim (\rho\theta)^{\overline V}.
\end{align*}
If two characters $\theta,\theta'$ agree on $V$, then obviously $\dim (\rho\theta)^V=\dim(\rho\theta)^V$. Write $\overline j:\overline U\to \overline U^{\rm ab}$ for the abelianization map. Put $C_\rho:=\sum_{\xi\in \Irr\: \overline j(\overline V)}\dim (\rho\xi)^V$, where $\rho\xi$ is understood as the tensor product of $\xi\circ \overline j$ and the restriction $\rho|_{V}$. The last sum becomes:
\begin{align}\label{eq:PreHolder}
\sum_{\rho\in \Irr \overline U/\sim}\frac{\O_\gamma(|\chi_\rho|)[\overline j(\overline U):\overline j(\overline V)]C_\rho }{a_\rho}\leq |\N_{k/\Q}(\Delta(\gamma))|^7 J_1[\overline j(\overline U):\overline j(\overline V)]\sum_{\rho\in \Irr \overline U/\sim}\frac{(\dim\rho)^{1/2}C_\rho }{a_\rho},
\end{align} by Lemma \ref{lem:CharacterOI}. Let $m>0$ be such that $\zeta_U^*(m-2)<\infty$. To bound the rightmost sum we use H\"older's inequality and the fact that $C_\rho\leq \dim\rho$:
\begin{align*}
\sum_{\rho\in \Irr \overline U/\sim}\frac{(\dim\rho)^{1/2}C_\rho }{a_\rho}=&\sum_{\rho\in \Irr \overline U/\sim}\left(\frac{(\dim\rho)\, C_\rho }{a_\rho}\right)^{\frac{2m-1}{2m}}\left(\frac{C_\rho}{a_\rho(\dim\rho)^{m-1}}\right)^{\frac{1}{2m}}\\
\leq& \left(\sum_{\rho\in \Irr \overline U/\sim}\frac{(\dim\rho) C_\rho }{a_\rho}\right)^{\frac{2m-1}{2m}}\left(\sum_{\rho\in \Irr \overline U/\sim}\frac{C_\rho}{a_\rho(\dim\rho)^{m-1}}\right)^{\frac{1}{2m}}\\
\leq& \left(\frac{[\overline{U}:\overline V]}{[\overline j(\overline U):\overline j(\overline V)]}\right)^{\frac{2m-1}{2m}}\left(\sum_{\rho\in \Irr \overline U/\sim}\frac{1}{(\dim\rho)^{m-2}}\right)^{\frac{1}{2m}}\\
\leq& \left(\frac{[U:V]}{[j(U):j(V)]}\right)^{\frac{2m-1}{2m}}\zeta_U^*(m-2)^{\frac{1}{2m}}.
\end{align*}
We plug it into (\ref{eq:PreHolder}) to get 
$$[U:V]\O_\gamma(\mathbf 1_V)=\O_\gamma(\chi_{\Ind_V^U 1})\leq  |\N_{k/\Q}(\Delta(\gamma))|^7 J_1 [U:V]^{\frac{2m-1}{2m}}[j(U):j(V)]^{\frac{1}{2m}}\zeta_U^*(m-2)^{\frac{1}{2m}}.$$
To finish the proof of the proposition it remains to use the bounds $[j(U):j(V)]\leq J_2$ and $\zeta_U^*(m-2)\leq J_3$. These are Lemmas \ref{lem:AbelianizationBound} and \ref{lem:RepzetaBound} respectively.

\subsection{Local orbital integrals} In this section we provide the bounds on $\O_\gamma(\mathbf 1_{U_{\frak p}^i})$ following closely Langlands' exposition \cite[p.39-52]{BaseChangeGL2}. 
The groups $U_\frak p^i$ are the representatives of all conjugacy classes of maximal compact subgroups in $\mathbf G(k_\frak p)$, described in Section \ref{sec:CongruenceLattices}. We recall that we compute the all the orbital integrals with respect to the standard measure $\mathbf G_\gamma(k_\frak p)$ and the standard measure on $\mathbf G(k_\frak p)$ relative to $U_\frak p^i$, so the measure of $U_\frak p^i$ is always $1$.
\begin{lemma}\label{lem:LocalOI}
Let $\gamma\in \mathbf G(k_\frak p)$ be a regular semisimple non-torsion element\footnote{ non-torsion just to avoid disconnected centralizers.}. If $\gamma$ is not split in $\mathbf G(k_\frak p)$ write $l_\frak p=k_\frak p[\gamma]$. Let $q$ be the cardinality of the residue field of $k_\frak p$. If $\O_\gamma(\mathbf 1_{U_\frak p^i})$ does not vanish, then the eigenvalues of $\Ad \gamma$ must be units lying in a quadratic extension of $k_\frak p$. We have the following formulas/bounds on $\O_\gamma(\mathbf 1_{U_\frak p^i})$
\begin{center}
\renewcommand{\arraystretch}{1.4}
  \begin{tabular}{| c | l | l | p{5cm}| }\hline
                  & $\gamma$ split                       & $l_\frak p/k_\frak p$ unramified & $l_\frak p/k_\frak p$ ramified\\ \hline
    $U_\frak p^0$ & $|\Delta(\gamma)|_{k_\frak p}^{-1/2}$ & $\frac{q+1}{q-1}|\Delta(\gamma)|_\frak p^{-1/2}-\frac{2}{q-1}$  & $\leq 2\frac{|\Delta(\gamma)|_\frak p^{-1/2}q^{1/2}-1}{q-1}$ if $|\Delta(\gamma)|_\frak p<1$ and $0$ if $|\Delta(\gamma)|_\frak p=1$\\ \hline
    $U_\frak p^1$ & $|\Delta(\gamma)|_{k_\frak p}^{-1/2}$ & $\frac{q+1}{q-1}|\Delta(\gamma)|_\frak p^{-1/2}-\frac{q+1}{q-1}$& $\leq 2\frac{|\Delta(\gamma)|_\frak p^{-1/2}q^{1/2}-1}{q-1}-1$\\ \hline
    $U_\frak p^2$ & -                                    & $1$&$1$\\ \hline
  \end{tabular}
\end{center}
\end{lemma}
\begin{proof}
For $U_\frak p^0$, the exact computations were made in \cite[p.39-52]{BaseChangeGL2} by studying the action of $\gamma$ on the Bruhat-Tits building $\frak X$ of $\SL(2,k_\frak p)$. For the sake of the reader we sketch the argument and at the same time compute the orbital integrals of $U_\frak p^1$. Write $\mu^{\rm st}$ for the standard measure on $\mathbf G(k_\frak p)$ and let $\mathbf T$ be the torus centralizing $\gamma$. 
In the split case we can assume, by conjugating $\gamma$, that $\mathbf T$ is the diagonal subgroup of $\PGL(2,k_\frak p)$. Let $\varpi$ be a uniformiser of $k_\frak p$. Write $\Lambda$ for the discrete subgroup of $\mathbf T$ generated by the class of $\begin{pmatrix}\varpi & 0\\0 & 1\end{pmatrix}$ and let $\frak A$ be the unique apartment of $\frak X$ stabilized by $\mathbf T$. The maximal compact subgroup of $\mathbf T(k_\frak p)$ is a fundamental domain of $\Lambda$, so $\mu^{\rm st}(\Lambda\bs\mathbf T(k_\frak p))=1$. Therefore $$\O_\gamma(\mathbf 1_{U_\frak p^i})=\int_{\Lambda\bs \mathbf G(k_\frak p)}\mathbf 1_{U_\frak p^i}(g^{-1}\gamma g)dg=\sum_{g\in \Lambda\bs \mathbf G(k_\frak p)/U_\frak p^i}\mathbf 1_{U_\frak p^i}(g^{-1}\gamma g)=|\Lambda\bs \{gU_\frak p^i|\gamma gU_\frak p^i=gU_\frak p^i\}|.$$ 
To go from the integral to counting cosets we have used the fact that $\mu^{\rm st}(U_\frak p^i)=1$ and that $\Lambda$ intersects trivially all the conjugates of $U_\frak p^i$. It is well known that $U_\frak p^i$ is conjugate to the stabilizer of a vertex in $\frak X$ if $i=0$ or an edge in $\frak X$ if $i=1$. Write $\frak X_0,\frak X_1$ for the set of vertices and edges of $\frak X$ respectively. Choose an edge $e\in \frak A$ and let $p_0$ be one of it's endpoints. The group $\Lambda$ acts freely and transitively on vertices and edges of $\frak A$ so the set $\mathcal F=e\cup \{x\in \frak X| {\rm dist}(x,\frak A)={\rm dist}(x,p_0)\}$ is a fundamental domain of $\Lambda$ acting on $\frak X$. If we delete all the edges in $\frak A$, then $\mathcal F$ is the connected component containing $p_0$ with edge $e$ added back to it. By \cite[Lemma 5.2]{BaseChangeGL2} the subset of $\frak X$ fixed by $\gamma$ is formed by precisely those points of $\frak X$ that are at distance at most $r$ from $\frak A$ where $q^r=|\Delta(\gamma)|_\frak p^{-1/2}, r\in\NN$. Let $\mathcal F_i$ be the set of vertices or edges in $\mathcal F$ for $i=0,1$ respectively. Since $\mathcal X$ is the $q+1$ regular tree we have $$|\Lambda\bs \{gU_\frak p^i|\gamma gU_\frak p^i=gU_\frak p^i\}|=|\mathcal F_i^\gamma|=q^r.$$ 
Now let us assume that $\gamma$ is elliptic. In that case $T(k_\frak p)$ is compact so we have $$\O_\gamma(\mathbf 1_{U_\frak p^i})=\sum_{g\in \mathbf G(k_p)/U_\frak p^i} \mathbf 1_{U_\frak p^i}(g^{-1}\gamma g)=\{gU_\frak p^i| \gamma gU_\frak p^i=gU_\frak p^i.\}$$
 Hence, $\O_\gamma(\mathbf 1_{U_\frak p^i}), i=0,1$ is respectively the number of vertices or edges of $\frak X$ stabilized by $\gamma$. Write $|\Delta(\gamma)|_\frak p^{-1/2}=q^r, r\in \frac{1}{2}\NN$. The subset of $\frak X$ fixed by $\gamma$ is a ball of radius $r$ centered around a vertex if $l_\frak p/k_\frak p$ is unramified (see \cite[p.48-49,]{BaseChangeGL2}) or a ball of radius $r-\delta, 0\leq \delta\leq r$ around  a midpoint of an edge if $l_\frak p/k_\frak p$ is ramified (see \cite[p.49-50]{BaseChangeGL2}). The number $\delta$ is the distance of $\frak X$ to the unique apartment of the Bruhat-Tits building of $\SL(2,l_\frak p)$ stabilized by $\mathbf T(l_\frak p)$.  The desired bounds follow by counting vertices or edges in $\frak X^\gamma$ and noting that in the ramified case when $i=1$ and $|\Delta(\gamma)|_\frak p=1$ there is one edge that is stabilized but not pointwise fixed.

If $i=2$, then $\mathbf G(k_\frak p)=U_\frak p^2$. There are no split elements and for every elliptic $\gamma$ we have $\O_\gamma(\mathbf 1_{U_\frak p^2})=1.$
\end{proof}
\begin{corollary}\label{cor:LocalOIBound}
Let $\gamma\in\mathbf G(k_\frak p)$ be a regular semisimple non-torsion element. We have $\O_\gamma(\mathbf 1_{U_\frak p^i})\leq 3|\Delta(\gamma)|_\frak p^{-1/2}$ if $|\Delta(\gamma)|_\frak p<1$ and $\O_\gamma(\mathbf 1_{U_\frak p^i})\leq 1$ if $|\Delta(\gamma)|_\frak p=1$. 
\end{corollary}


\subsection{Character bounds}\label{sec:CharacterOI}
\begin{theorem}\label{thm:CharacterBounds}
Let $\mathbf H$ be a connected reductive group defined over a non-Archimedean local field $F$, let $\rk\, \mathbf H$ be the absolute rank of $\mathbf H$ and let $W$ be the absolute Weyl group of $\mathbf H$. For every open compact subgroup $U\subset \mathbf H(F)$ and every irreducible complex representation $\rho$ of $U$ we have $$|\chi_\rho(\gamma)|\leq |\Delta(\gamma)|_F^{-1}2^{\dim \mathbf H-{\rm rk}\, \mathbf H}|W|.$$
\end{theorem}
\begin{proof}
The proof is inspired by an unpublished note \cite{Larsen} where Michael Larsen proved an analogous result for finite groups of Lie type. To lighten the notation write $\chi:=\chi_\rho$ and extend it by $0$ to $\mathbf H(F)$. 
Following Serre\footnote{The formula is stated there for central simple division algebras but the same statement holds for any reductive algebraic group.} \cite[Formule (21)]{Serre78} we shall use the Weyl integration formula. We choose the unique Haar measure $dh$ on $\mathbf H(F)$ giving mass $1$ to $U$ and for any torus $\mathbf T$ we fix compatible measures $dx,dt$ on $\mathbf H(F)/\mathbf T(F), \mathbf T(F)$ in such a way that $U\cap \mathbf T(F)$ has measure $1$. For any torus $\mathbf T$ put $ W_\mathbf T=N(\mathbf T(F))/\mathbf T(F)$ and let $W$ be the absolute Weyl group. For any continuous compactly supported function $\phi$ on $\mathbf H(F)$ we have
\begin{equation}\label{eq:WeylForm}
\int \phi(g) dg=\sum_{[\mathbf T]_{\mathbf H(F)}\subset \mathbf H}\frac{1}{|W_{\mathbf T}|}\int_{\mathbf H(F)/\mathbf T(F)}\int_{\mathbf T(F)}|\Delta(t)|_F \phi(xtx^{-1})dt\,dx.
\end{equation}
The sum is taken over the set of $\mathbf H(F)$-conjugacy classes of maximal tori of $\mathbf H$ defined over $F$. Let $\mathbf T_0$ be the connected component of the centralizer of $\gamma$. Applying (\ref{eq:WeylForm}) we get 
\begin{align*}1=\int_U|\chi(h)|^2dh\geq& \frac{1}{|W_{\mathbf T_0}|}\int_{U/(\mathbf T_0(F)\cap U)}\left(\int_{U\cap \mathbf T_0(F)}|\Delta(t)|_{F}|\chi(t)|^{2}dt\right)dx\\
=&\frac{1}{|W_{\mathbf T_0}|}\int_{\mathbf T_0(F)\cap U }|\Delta(t)|_{F}|\chi(t)|^{2}dt.
\end{align*}
In particular 
\begin{equation}\label{e.Ineq11}
|W|\geq |W_{\mathbf T_0}|\geq \int_{\mathbf T_0(F)\cap U}|\Delta(t)|_F |\chi(t)|^2dt,
\end{equation}
We shall approximate $|\Delta(t)|_F$ by an integral combination of unitary characters of $\mathbf T_0(F)\cap U$. Let $\Phi$ be the root system of $\mathbf H$ relative to $\mathbf T_0$ and choose a set of positive roots $\lambda_1,\ldots,\lambda_{l}$ with $l=(\dim \mathbf H-{\rm rk}\,\mathbf H)/2$. For any $t\in \mathbf T_0(F)$
$$\Delta(t)=\prod_{i=1}^l(1-\lambda_i(t))(1-\lambda_i(t)^{-1}).$$
Let $E$ be the extension of $F$ generated by $\lambda_i(\gamma), i=1,\ldots,l$. The image of $\mathbf T_0(F)\cap U$ via any character $\lambda\in X^*(\mathbf T_0)$ is a compact subgroup of $E^\times$ so $\lambda_i(\mathbf T_0(F)\cap U)\subset \frak o_{E}^\times$ for $i=1,\ldots,l$. For any $i=1,\ldots,l$ we pick a unitary character $\theta_i\colon \frak o_{E}^\times\to \C^\times$ such that $\theta_i(\lambda_i(\gamma))\neq 1$ but $\theta_i(\lambda_i(t))=1$ is trivial on the subgroup defined by 
\begin{equation}\label{e.Condition11}
\{t\in \mathbf T_0(F)\cap U|\,  |1-\lambda_i(t)|_F<|1-\lambda_i(\gamma)|_F\}.
\end{equation}  Define $\Theta:\mathbf T_0(F)\cap U\to \C$ by 
$$\Theta(t)=\prod_{i=1}^l(1-\theta_i(\lambda_i(t))).$$
 $\Theta(\gamma)$ is a non-zero algebraic integer so there is a Galois automorphism that makes it least one in absolute value. Hence, by composing $\theta_i$'s with a Galois automorphism we can guarantee that $|\Theta(\gamma)|\geq 1$. 
Since $\theta_i$ are trivial on the subgroups (\ref{e.Condition11}) we have $\Theta(t)= 0$ for all $t$ with $|\Delta(t)|_F<|\Delta(\gamma)|_F$. It follows that $|\Theta(t)|^2/|\Delta(t)|_F \leq |\Theta(t)|^2/|\Delta(\gamma)|_F $. We combine this inequality with (\ref{e.Ineq11}) to get
\begin{equation}\label{e.Condition2}
\int_{\mathbf T_0(F)\cap U}|\Theta(t)|^2|\chi(t)|^2dt\leq \sup_{t\in \mathbf T_0(F)\cap U}|\Theta(t)|^2|\Delta(\gamma)|_F^{-1}|W|\leq 2^{2l}|\Delta(\gamma)|_F^{-1}|W|.
\end{equation}
By construction, $\Theta(t)\chi(t)$ is an integral combination of unitary characters of $\mathbf T_0(F)\cap U$. We can write 
\begin{equation*}
\Theta(t)\chi(t)=\sum_{\xi}c_\xi \xi(t) \textrm{ with }  
c_\xi\in\Z.
\end{equation*}
By (\ref{e.Condition2}) we have $\sum_{\xi}c_\xi^2\leq 2^{2l}|\Delta(\gamma)|_F^{-1}|W|$. In particular
\begin{equation*}\label{e.VB2}
|\chi(\gamma)|\leq \left|\Theta(\gamma)\chi(\gamma)\right|=\left|\sum_{\xi}c_\xi \xi(\gamma)\right|\leq\sum_{\xi}|c_\xi|\leq 2^{\dim \mathbf H-\rk \mathbf H}|\Delta(\gamma)|_F^{-1}|W|.
\end{equation*}
\end{proof}
From Theorem \ref{thm:CharacterBounds} and Corollary \ref{cor:LocalOIBound} we immediately deduce:
\begin{corollary}\label{cor:LocalCharacterOI} 
Let $U_\frak p =U_\frak p^i$ where $i\in\{0,1,2\}$. Then for every irreducible representation $\rho_\frak p$ of $U_\frak p$ we have 
$$\O_\gamma(|\chi_{\rho_\frak p}|)\leq \min\{\dim\rho_\frak p, 8|\Delta(\gamma)|_\frak p^{-1}\}\begin{cases} 3|\Delta(\gamma)|_\frak p^{-1/2}& {\rm if } |\Delta(\gamma)|_\frak p<1,\\
                                     1 & {\rm if} |\Delta(\gamma)|_\frak p=1.
                                     \end{cases}
$$
\end{corollary}
\

\subsection{Quasi-randomness and an estimate of $\O_\gamma(\mathbf 1_V)$}
In this section we use Corollary  \ref{cor:LocalCharacterOI} and a form of \textit{quasi-randomness} of the family of groups $ U_\frak p^i$ to establish Lemma \ref{lem:CharacterOI}. What is usually meant by quasi-randomness of a family of finite groups $\{H_j\}$ is the fact that there exists $c>0$ such that the non-trivial representations $\rho$ of $H_j$ satisfy $\dim \rho\geq |H_j|^c$. A bit weaker property holds for the family $U_\frak p^i$ as $\frak p$ varies among finite places of $k$. 
\begin{lemma}\label{lem:QuasiRand}
Let $q$ be the cardinality of the residue field of $k_\frak p$ and let $\rho$ be an irreducible representation of $U_\frak p^i$ where $i=0,1,2.$
\begin{enumerate}
 \item If $i=0$, then $\dim \rho=1$ or $\dim \rho\geq q-1$.
 \item If $i=1$, then $\dim \rho\leq 2$ or $\dim\rho\geq q-1$.
 \item If $i=2$, then $\dim\rho\leq 2$ or $\dim \rho\geq q+1.$
\end{enumerate}
\end{lemma}
\begin{proof}
 (1) is a special case of \cite[Theorem 1]{Bardestani14}. Let $K$ be the image of $\begin{pmatrix}1+\frak p& \frak o_\frak p\\ \frak p & 1+\frak p \end{pmatrix}$ in $\PGL(2,k_\frak p)$. It is a normal pro-$p$ subgroup (where $p$ is the residual characteristic of $k_\frak p$) of $U_\frak p^1$. Matrix computations show that $K$ is the normal closure of $N:=\begin{pmatrix}1& \frak o_\frak p\\ 0 & 1\end{pmatrix}$ in $U_\frak p^1$. We deduce that either $\rho|_K$ is trivial or $\rho|_N$ contains a non-trivial character $\psi$ of $N$. In the former case $\rho$ factors through the dihedral group $U_\frak p^1/ K\simeq D_{q-1}$. In that case $\dim \rho\leq 2$ follows from the explicit description of irreducible representations of dihedral groups. Assume now that $\rho$ contains a non-trivial character $\psi$ of $N$. Let $A=\begin{pmatrix}\frak o_\frak p^\times& 0\\0 & \frak o_\frak p^\times \end{pmatrix}$. Then $A$ normalizes $N$ and the orbit of $\psi$ under conjugation by $A$ has at least $q-1$ elements. Accordingly, Clifford's theorem yields $\dim \rho\geq q-1$. This proves (2). In the last case $U_\frak p^2=\mathbf G(k_\frak p)$ so we can read the possible dimensions of $\rho$ from \cite[Proposition 6.5]{Cara84}. The lowest ones are $1,2$ followed by $q+1$, which proves (3).
\end{proof}

We can start proving Lemma \ref{lem:CharacterOI}.
\begin{proof}[Lemma \ref{lem:CharacterOI}] Let $\rho$ be an irreducible representation of $U$. Then $\rho=\prod_{\frak p}\rho_\frak p$ where $\rho_\frak p$ is an irreducible representation of the local factor $U_\frak p$ and $\rho_\frak p=1$ for almost all primes $\frak p$. Let $ Q_1:=\{\frak p|\, |\Delta(\gamma)|_\frak p<1\}, Q_2:=\{\frak p|\, 8|\Delta(\gamma)|_\frak p^{-1}> \dim\rho_\frak p\}.$ Using Corollary \ref{cor:LocalCharacterOI} we get 
\begin{align}\label{eq:I1I2Decomp}\nonumber\O_\gamma(|\chi_\rho|)=&\prod_{\frak p}\O_\gamma(|\chi_\rho|)\leq 3^{|Q_1|}|\N_{k/\Q}(\Delta(\gamma))|^{1/2}\prod_{\frak p\in Q_2}\dim\rho_\frak p\prod_{\frak p\not\in Q_2}8|\Delta(\gamma)|_\frak p^{-1}\\ 
\leq & |\N_{k/\Q}(\Delta(\gamma))|^{3/2}3^{|Q_1|}\prod_{\frak p\in Q_2}(\dim\rho_\frak p|\Delta(\gamma)|_\frak p)\prod_{\frak p\not\in Q_2}8.
\end{align}
Put $I_1:=|\N_{k/\Q}(\Delta(\gamma))|^{3/2}3^{|Q_1|}, I_2=\prod_{\frak p\in Q_2}(\dim\rho_\frak p|\Delta(\gamma)|_\frak p)\prod_{\frak p\not\in Q_2}8$. We have 
$|Q_1|\leq \log |\N_{k/\Q}(\Delta(\gamma))|/\log 2$ so $I_1\leq |\N_{k/\Q}(\Delta(\gamma))|^{3/2+\log 3/\log 2}\leq |\N_{k/\Q}(\Delta(\gamma))|^{4}$. We turn to $I_2$
\begin{align}\label{eq:I2Bound}
\nonumber I_2\leq& \prod_{\frak p\in Q_1}8\prod_{\dim\rho_\frak p<8}\dim \rho_\frak p\prod_{\dim\rho_\frak p\geq 8}8\leq 2^{3|Q_1|}\prod_{\dim\rho_\frak p<64}\dim\rho_\frak p\prod_{\dim\rho_\frak p\geq 64}(\dim\rho_\frak p)^{1/2}\\
=& 2^{3|Q_1|}(\dim\rho)^{1/2}\left(\prod_{\dim\rho_\frak p<64}\dim\rho_\frak p\right)^{1/2}.
\end{align}
We bound the product of low dimensions using Lemma \ref{lem:QuasiRand}. Indeed, by Lemma \ref{lem:QuasiRand} the representation $\rho_\frak p$ with dimension $>1$ and at most $64$ can occur only if $\N(\frak p)<65$ or $\dim\rho_\frak p=2$ and $\frak p\in S\cup \Ram_f A$. Hence $$\prod_{\dim\rho_\frak p<64}\dim\rho_\frak p\leq \prod_{\N(\frak p)<65}64\prod_{\frak p\in S\cup \Ram_f A}2\leq 2^{6\pi_k(64)+|S|+|\Ram_f A|}.$$ We plug it into (\ref{eq:I2Bound}) to get
$$I_2\leq |\N_{k/\Q}(\Delta(\gamma))|^3(\dim \rho)^{1/2} 2^{3\pi_k(64)+|S|/2+|\Ram_f A|/2}.$$ Finally, we combine the bounds on $I_1,I_2$ and (\ref{eq:I1I2Decomp})
$$\O_\gamma(|\chi_\rho|)\leq I_1I_2\leq |\N_{k/\Q}(\Delta(\gamma))|^{7} (\dim\rho)^{1/2}2^{3\pi_k(64)+|S|/2+|\Ram_f A|/2}.$$
Lemma \ref{lem:CharacterOI} is proved.
\end{proof}

\subsection{Representation zeta functions}\label{sec:RepZetaBound}
In this section we give explicit upper bounds on the twist representation zeta functions (Definition \ref{def:TwistRepZeta}) of $U_\frak p^i,i=0,1,2$ and deduce a bound on $\zeta_U^*(s)$. We recall that in our notation $U=\prod_\frak p U_\frak p$ where $U_\frak p\simeq U_\frak p^1$ if $\frak \in S$, $U_\frak p\simeq U_\frak p^2$ if $\frak p\in \Ram_f A$ and $U_\frak p\simeq U_\frak p^0$ otherwise. 
\begin{lemma}\label{lem:TwistRepZetaBound} We have 
$\zeta_U^*(4)\leq \zeta_k(2)^{24}162^{|S|}2^{|\Ram_f A|}\prod_{\frak p\in S}(\N(\frak p)+1)^3\prod_{\frak p\in \Ram_f A}(N(\frak p)^2+1).$
\end{lemma}
\begin{proof}
$$U=\prod_\frak p U_\frak p\simeq \prod_{\frak p\not\in (S\cup \Ram_f A)}U_\frak p^0\times\prod_{\frak p\in S}U_\frak p^1\times \prod_{\frak p\in \Ram_f A}U_\frak p^2,$$ so $$\zeta_{U_\frak p}^*(4)=\prod_{\frak p\not\in (S\cup \Ram_f A)}\zeta_{U_\frak p^0}^*(4)\prod_{\frak p\in S}\zeta_{U_\frak p^1}^*(4)\prod_{\frak p\in \Ram_f A}\zeta_{U_\frak p^2}^*(4).$$ Desired inequality follows immediately from Lemma \ref{lem:LocRepZetaBound}.
\end{proof} Lemma \ref{lem:RepzetaBound} follows upon noting that $\zeta_U^*(s)$ is decreasing in $s$. 
The remainder of this section is devoted to the bounds on $\zeta^*_{U_\frak p^i}(s)$. 
\begin{lemma}\label{lem:LocRepZetaBound} Let $q$ be the cardinality of the residue field of $k_\frak p$. For every $s\geq 2$:
\begin{enumerate}
 \item $\zeta_{U_\frak p^0}^*(s)\leq (1-q^{2-s})^{-3\cdot 2^{s-1}}$. 
 \item $\zeta_{U_\frak p^1}^*(s)\leq 2\cdot 3^s(q+1)^3(1-q^{2-s})^{-1}.$
 \item $\zeta_{U_\frak p^2}^*(s)\leq 2(q^2+1)(1-q^{2-s})^{-1}.$
\end{enumerate}
\end{lemma}
 Before the proof, few remarks are in order. The twist representation zeta functions of $\GL(2,\frak o_\frak p)$ were studied in \cite{HasaStasinski2018} by H\"as\"a and Stasinski. They give an exact formula for $\zeta^*_{\GL(2,\frak o_\frak p)}$ (c.f. \cite[Theorem 4.14]{HasaStasinski2018}) under the assumption that $q$ is odd. They also give  bounds on $\zeta_{\GL(2,\frak o_\frak p)}^*(s)$ for even $q$. They suffice for establishing the axis of convergence of $\zeta_{\GL(2,\frak o_\frak p)}^*(s)$, which was the goal in \cite{HasaStasinski2018}. Unfortunately, the implicit constants in these bounds depend on the ramification index of $k_\frak p$ which would lead to a substantially weaker bound in Lemma \ref{lem:TwistRepZetaBound}. For odd $q$ the twist representation zeta function of $\zeta^*_{U_\frak p^0},\zeta^*_{U_\frak p^1}$ could be computed, or at least estimated using the list of irreducible representations of $U_{\frak p}^0,U_\frak p^1$  given by Silberger in \cite{Silberger70},\cite{Silberger77}. In our setup we cannot ignore the even primes, so we are forced to estimate the local zeta functions by hand.
\begin{proof}
\begin{enumerate}
\item 
We will identify $\Irr U_\frak p^0= \Irr \PGL(2,\frak o_\frak p)$ with the subset of $\Irr \GL(2,\frak o_\frak p)$ of representations with trivial central character. For any $r\geq 1$ let $K^0_r:=1+\frak p^r M(2,\frak o_\frak p)\subset \GL(2,\frak o_\frak p)$. We say that a non-trivial representation $\rho\in\Irr \GL(2,\frak o_\frak p)$ is of level $r$ if it factors through $\GL(2,\frak o_\frak p)/K^0_r$ but not $\GL(2,\frak o_\frak p)/K^0_{r-1}$. In that case we write $\ell(\rho)=r$ and we adopt the convention that $\ell(1)=0$. An irreducible representation $\rho\in\Irr \GL(2,\frak o_\frak p)$ will be called \textit{primitive} if $\ell(\rho)\leq \ell(\rho\otimes \theta)$ for every one dimensional character $\theta$ that factors through the determinant map. According to \cite[Lemma 4.4]{HasaStasinski2018} for every primitive $\rho$ we have \begin{equation}\label{eq:DimLB}\dim \rho\geq (q-1)q^{\ell(\rho)-1}.\end{equation} 
Let $W^0$ be the set of primitive irreducible representations of $\GL(2,\frak o_\frak p)$. Let $N_r^0:=|\{\rho\in W^0| \ell(\rho)=r\}|$. By (\ref{eq:DimLB}) we have $$N^0_r\leq \frac{|\GL(2,\frak o_\frak p)/K_r^0|}{ (q-1)^2q^{2r-2}}=(q+1)q^{2r-1}.$$ Every $\rho\in \Irr\PGL(2,\frak o_\frak p)$ either factors through the determinant or it admits a twist $\rho\otimes\theta$ which is primitive of positive level. Therefore 
\begin{align*}\zeta_{\PGL(2,\frak o_\frak p)}^*(s)\leq& 1+\sum_{r=1}^\infty \sum_{\substack{\rho\in W^0\\ \ell(\rho)=r}}(\dim\rho)^{-s}\leq 1+\sum_{r=1}^\infty \frac{(q+1)q^{2r-1}}{(q-1)^sq^{(r-1)s}}\leq 1+3\cdot 2^{s-1}\sum_{r=1}^\infty  q^{r(2-s)}\\
\leq& (1-q^{2-s})^{-3\cdot 2^{s-1}}.
\end{align*} 
\item   Let $P=\begin{pmatrix}\frak o_\frak p^\times  & \frak o_\frak p\\ \frak p & \frak o_\frak p^\times\end{pmatrix}$ and let $K_r^0$ be as in the previous point. For $\rho\in \Irr P$ non-trivial we define the level $\ell(\rho)$ to be the smallest $r\geq 1$ such that $\rho$ factors through $P/K^0_r$. We say that $\rho$ is primitive if $\ell(\rho)\leq \ell(\rho\otimes\theta)$ for every one dimensional character $\theta$ that factors through the determinant map. We claim that \begin{equation}\label{eq:DimLB2}\dim\rho\geq \frac{(q-1)q^{\ell(\rho)-1}}{q+1}.\end{equation}Indeed, let $\rho_1\in \Irr \GL(2,\frak o_\frak p)$ be any irreducible representation containing $\rho$. Then $\ell(\rho_1)\geq \ell(\rho)$ so by (\ref{eq:DimLB}) we have 
$\dim \rho_1\geq (q-1)q^{\ell(\rho)-1}$. Desired inequality follows because $\dim\rho_1\leq \dim \Ind_{P}^{\GL(2,\frak o_\frak p)}\rho=(q+1)\dim\rho$.

 Let $W^1$ be the set of primitive irreducible representations $\rho\in \Irr P$. Let $N^1_r:=|\{\rho\in W^1|\ell(\rho)=r\}|$. We have $$N_r^1\leq \frac{|P/K_r^0|(q+1)^2}{(q-1)^2q^{2r-2}}=\frac{q(q-1)^2q^{4r-4}(q+1)^2}{(q-1)^2q^{2r-2}}= q(q+1)^2q^{2r-2}.$$ The lower bound (\ref{eq:DimLB2}) is useless for representations of level $1$. Instead, we describe them explicitly: $P/K_1^1\simeq \F_q\rtimes (\F_q^\times)^2$ where $(\F_q^\times)^2$ acts on $\F_q$ by $(a,b)x=ab^{-1}x$. This group has $(q-1)^2$ one dimensional representations and $q-1$ irreducible representations of dimension $q-1$. We use (\ref{eq:DimLB2}) and the bound on $N_r^1$ to obtain 
\begin{align*}
 \sum_{\rho\in W^1}\dim\rho^{-s}\leq& (q-1)^2 + (q-1)(q-1)^{-s}+\sum_{r=2}^\infty \sum_{\substack{\rho\in W^1 \\ \ell(\rho)=r}}(\dim \rho)^{-s}\\
 \leq& (q-1)^2+(q-1)^{1-s}+\frac{q(q+1)^{2+s}}{(q-1)^s}\sum_{r=1}^\infty q^{r(2-s)}\leq 3^s(q+1)^3(1-q^{2-s})^{-1}.
\end{align*}
It remains to compare the above series with $\zeta_{U_\frak p^1}^*(s)$. The group $P/\frak o_\frak p^\times$ is a normal index $2$ subgroup of $U_\frak p^1$ and the conjugation by $U_\frak p^1$ preserves the levels of irreducible representations. The determinant map on $P$ extends to $U_\frak p^1\frak o_\frak p^\times$ so every irreducible representation $\rho_0\in\Irr U_\frak p^1$ has a twist $\rho_0\otimes \theta$, where $\theta$ factors through the determinant map, such that the irreducible components of $\rho_0|_{P/\frak o_\frak p^\times}$ are primitive as representations of $P$, hence are in $W^1$. Moreover, every $\rho\in W^1$ appears in at most two irreducible representations of $U_\frak p^1$. We infer 
$$\zeta_{U_\frak p^1}^*(s)\leq 2\sum_{\rho\in W^1}\dim\rho^{-s}\leq 2\cdot 3^s(q+1)^3(1-q^{2-s})^{-1}.$$

\item In that case $U_\frak p^2=\mathbf G(k_\frak p)=A_\frak p^\times/ k_\frak p^\times.$ Let $\O_\frak p=\{x\in A_\frak p|\, |\n(x)|_\frak p\leq 1\}$, it is the maximal compact subring of $A_\frak p$. Let $\frak P=\{x\in \O_\frak p|\, |\n(x)|_\frak p<1\}$ be the Jacobson radical of $\O_\frak p$. We define the groups $K_0^2:=\O_\frak p^\times$ and $K_r^2:=1+\frak P^r, r\geq 1$. We define the level $\ell(\rho)$ of $\rho\in \Irr A_\frak p^\times$ as the smallest $r$ such that $\rho$ factors through $A_\frak p^\times/ K_r^2$. Note that this time we can have non-trivial representations of level $0$. The \textbf{conductor} $c(\rho)$ of $\rho$ is given by the relation $c(\rho)=\ell(\rho)+1$ (c.f. \cite[p. 213]{Cara84}). We note that all irreducible representations of level $0$ are one dimensional. Following Carayol we say that $\rho$ is \textit{ of minimal conductor } if $c(\rho)\leq c(\rho\otimes \theta)$ for every character $\theta$ that factors though the norm map. To be consistent with nomenclature of previous points we will say that the representations of minimal conductor are primitive. The dimension formula \cite[Proposition 6.5]{Cara84}, translated in terms of levels, tells us that if $\rho\in \Irr A_\frak p^\times$ is primitive of positive level, then
$$\dim \rho=\begin{cases}(q+1)q^{\frac{\ell(\rho)-2}{2}} &\textrm{ if } \ell(\rho) \textrm{ is even},\\
             2q^{\frac{\ell(\rho)-1}{2}} &\textrm{ if } \ell(\rho) \textrm{ is odd.}
            \end{cases}$$
In particular $\dim\rho\geq q^{\lfloor \frac{\ell(\rho)}{2}\rfloor}$ for every primitive $\rho$.  Choose a uniformizer $\varpi$ of $k_\frak p$. Let $W^2$ be the set of positive level irreducible primitive representations of $A_\frak p^\times$ that are trivial on the group $I:=\langle \varpi\rangle$. For $r\geq 0$ let $N_r^2$ be the number of $\rho\in W^2$ with $\ell(\rho)=r$. We bound the numbers $N_r^2$ using the lower bound on dimensions: 
$$N_r^2\leq \frac{|A_\frak p^\times/IK_r^2|}{q^{2\lfloor r/2\rfloor}}\leq \frac{2(q^2-1)q^{2r-2}}{q^{2\lfloor r/2\rfloor}}<2q^{2r-2\lfloor r/2\rfloor}.$$
It follows that 
$$\sum_{\rho\in W^2}\dim\rho^{-s}\leq \sum_{r=1}^\infty \frac{2q^{2r-2\lfloor r/2\rfloor}}{q^{s\lfloor r/2\rfloor}}=2\sum_{r=1}^\infty q^{r(2-s)}+2q^2\sum_{r=0}^\infty q^{r(2-s)}.$$
Every irreducible representation $\rho$ of $U_\frak p^2$ is either one dimensional or admits a twist $\rho\otimes \theta\in \Irr A_\frak p^\times$ which is primitive and trivial on $I$. Therefore 
$$\zeta_{U_\frak p^2}^*(s)\leq 1+\sum_{\rho\in W^2}\dim\rho^{-s}\leq 2+2\sum_{r=1}^\infty q^{r(2-s)}+2q^2\sum_{r=0}^\infty q^{r(2-s)}=2(1+q^2)(1-q^{2-s})^{-1}.$$
\end{enumerate}
\end{proof}
\subsection{Abelianization of $U$}\label{sec:AbelianizationBound}
Let $V$ be an open subgroup of $U$. As before $U$ is a maximal compact open subgroup of $\mathbf G(\mathbb A_f)$ with product decomposition $U=\prod_{\frak p}U_\frak p$ and $S:=\{\frak p| U_\frak p\sim U_\frak p^1\}$. In this section we prove Lemma \ref{lem:AbelianizationBound}. Let us set up some notations relevant to this part. Let $j\colon U\to U^{\rm ab}$ be the abelianization map. The norm map $\n:A^\times\to k^\times$ induces maps $\n\colon \mathbf G(R)\to R^\times/(R^\times)^2$ for any $k$-algebra $R$. Throughout this section $\n$ will refer to these induced maps. We shall write $U^1,V^1$ and $(U_\frak p^i)^1, i=0,1,2$ for the kernel of $\n$ on $U,V$ and $U_\frak p^i,i=0,1,2$ respectively. The norm map factors through the abelianization and we will write $\n^{\rm ab}(j(u)):=\n(u)$ for $u\in U$. We use the same letter to denote the restriction of any of these maps to the local components $U_\frak p$. 

\begin{lemma}\label{lem:LocAbel}
Let $i\in\{0,1,2\}$ be such that $U_\frak p$ is conjugate to $U_\frak p^i$. We have 
$$\n(U_\frak p)=\begin{cases}
                 \frak o_\frak p^\times/(\frak o_\frak p^\times)^2 & \textrm{ if } i=0,\\
                 k_\frak p^\times/(k_\frak p^\times)^2 & \textrm{ if } i=1,2,\\
                \end{cases} \ \ 
j((U_\frak p)^1)=\begin{cases}
                 \Z/2\Z & \textrm{ if } i=0 \textrm{ and } \N(\frak p)=2,\\
                 \{1\} & \textrm{ otherwise }.\\
                \end{cases}
$$                
\end{lemma}\begin{proof} The part on $\n(U_\frak p)$ is clear. Let $\F_q$ be the residue field of $k_\frak p$. If $i=0$, then the commutator subgroup always contains the kernel of the reduction map mod $\frak p$ so $j((U_\frak p)^1)\simeq \PSL(2,\F_q)^{\rm ab}$. The group $\PSL(2,\F_q)$ is perfect unless $q=2$ in which case the abelianization is $\Z/2\Z$. If $i=1$ we verify by hand that the subgroups $\begin{pmatrix}1 & \frak o_\frak p\\ 0 & 1 \end{pmatrix}, \begin{pmatrix}1 & 0\\ \frak p & 1 \end{pmatrix}$ and $\left\{\begin{pmatrix}a & 0\\ 0 & a^{-1} \end{pmatrix}\mid a\in \frak o_\frak p^\times\right\}$ are in the commutator subgroup. They generate $(U_\frak p^1)^1$, whence $j((U_\frak p^1)^1)=1$. By \cite[1.4.3]{PlaRap} $(U_\frak p^2)^1$ is the commutator subgroup of $U_\frak p^2$ so we also have $j((U_\frak p^2)^1)=1$. 
\end{proof}

\begin{proof}[Lemma \ref{lem:AbelianizationBound}]
The proof consists of two applications of snake lemma and the Dirichlet's unit theorem. First we shall bound $[j(U):j(V)]$ in terms of $[n(U):n(V)]$. Consider the following exact diagram
$$\begin{tikzcd}  1\arrow[r]& j(V^1)\arrow[r]\arrow[d]& j(V)\arrow
[r, "\n^{\rm ab}"]\arrow[d] & \n(V)\arrow[r]\arrow[d]&1 \\  
1\arrow[r]& j(U^1)\arrow[r]& j(U)\arrow
[r, "\n^{\rm ab}"] & \n(U)\arrow[r]&1   ,
  \end{tikzcd}
$$ where columns are induced by the inclusion maps. Snake lemma yields an exact sequence $1\to j(U^1)/j(V^1)\to j(U)/j(V)\to n(U)/n(V)$. By Lemma \ref{lem:LocAbel} we get \begin{equation}\label{eq:Jindex1} [j(U):j(V)]\leq |j(U^1)|[\n(U):\n(V)]\leq 2^{\pi_k(2)}[\n(U):\n(V)]\leq 2^{[k:\Q]}[\n(U):\n(V)].\end{equation}
We proceed to bound $[n(U):n(V)]$. Let $W_U,W_V$ be the images of $\n(U),\n(V)$ in $\mathbb A_f^\times/ (\mathbb A_f^\times)^2 k^\times$. We need to estimate the size of $\ker[\n(U)\to W_U]$. Write $\frak o_1$ for the ring of $(S\cup \Ram_f A)$-integers in $k$ and let $\overline {\frak o_1}$ be its weak closure in $\mathbb A_f$. By Lemma \ref{lem:LocAbel} we recognize $\overline{\frak o_1}^\times/(\overline{\frak o_1}^\times)^2$ as $\n(U)$. The map $\n(U)\to \mathbb A_f^\times/ (\mathbb A_f^\times)^2 k^\times$ is a composition  
$$\begin{tikzcd}\n(U)=\overline{\frak o_1}^\times/(\overline{\frak o_1}^\times)^2\arrow[r,"\alpha"]& (\overline{\frak o_1}^\times/\frak o_1^\times(\overline{\frak o_1}^\times)^2)\arrow[r,"\beta"]&\mathbb A_f^\times/(\mathbb A_f^\times)^2k^\times\end{tikzcd}$$
In our setup $k$ is either quadratic or has at least one real embedding, so the $2$-torsion part of $\frak o_k^\times$ is at most $\Z/2\Z$. Hence, by Hasse-Minkowski's local global principle and Dirichlet's unit theorem $|\ker \alpha|=|\frak o_1^\times/ (\frak o_1^\times)^2|\leq 2^{[k:\Q]+|S|+|\Ram_f A|}$. 
Put $C:=\mathbb A_f^\times/\overline{\frak o_1}^\times k^\times$. We have an exact sequence
$$\begin{tikzcd} {\rm Tor}(C,\Z/2\Z)\arrow[r]& (\overline{\frak o_1}^\times/\frak o_1^\times(\overline{\frak o_1}^\times)^2)\arrow[r,"\beta"]&\mathbb A_f^\times/(\mathbb A_f^\times)^2k^\times\arrow[r]& C/C^2\arrow[r]&1.\end{tikzcd}$$
Therefore $|\ker \beta|\leq |{\rm Tor}(C,\Z/2\Z)|=|C/C^2|$. The group $C$ is a quotient of the class group of $k$ so $|\ker\beta|\leq |\cl(k)/\cl(k)^2|$. We deduce that 
$$|\ker[\n(U)\to W_U]|=|\ker\alpha||\ker\beta|\leq 2^{[k:\Q]+|S|+|\Ram_f A|}|\cl(k)/\cl(k)^2|.$$ Therefore

\begin{equation}\label{eq:NormIndex1}
[n(U):n(V)]\leq |\cl(k)/\cl(k)^2|2^{[k:\Q]+|S|+|\Ram_f A|} [W_U:W_V]
\end{equation}
The group $W_U$ fits into the exact sequence $1\to W_U\to \mathbb A_f^\times/ (\mathbb A_f^\times)^2k^\times\to \cl(U)\to 1$. Indeed, we have shown this in the proof of Lemma \ref{lem:ClassGroup}. Same statement holds for $W_V$.
To estimate $[W_U:W_V]$ we consider the following extact diagram
$$\begin{tikzcd}  1\arrow[r]& W_V \arrow[r]\arrow[d]& \mathbb A_f^\times/ (\mathbb A_f^\times)^2k^\times \arrow
[r]\arrow[d] & \cl(V)\arrow[r]\arrow[d]&1 \\  
1\arrow[r]& W_U\arrow[r]& \mathbb A_f^\times/ (\mathbb A_f^\times)^2k^\times\arrow
[r] & \cl(U)\arrow[r]&1.   
  \end{tikzcd}
$$
Upon applying snake lemma we find that $[W_U:W_V]=\frac{|\cl(V)|}{|\cl(U)|}.$ We combine this with (\ref{eq:Jindex1}) and (\ref{eq:NormIndex1}) to get 
$$[j(U):j(V)]\leq 2^{2[k:\Q]+|S|+|\Ram_f A|}|\cl(k)/\cl(k)^2|\frac{|\cl(V)|}{|\cl(U)|}.$$ Lemma \ref{lem:AbelianizationBound} is proven.
\end{proof}
\section{Volume computations}
\subsection{Co-volumes of congruence lattices}\label{sec:Covolumes}
In this section we give a volume formula for $\Vol(\mathbf G(k)\bs \mathbf G(\mathbb A))$ in a form that will be useful in the proof of Theorem \ref{mthm:TraceEstimate}. We shall deduce it from the Borel volume formula \cite{Bor81} (see also \cite[Chapter 11]{MaRe03}). Let $\O\subset A$ be an order in the quaternion algebra $A$ (c.f. \cite[2.2]{MaRe03}). Write $\Gamma_{\O^1}$ for the lattice of $\PGL(2,\K)$ obtained as the image of $\O^1$ under the map $A^\times\to A^\times_{\nu_1}\to \mathbf G(k_{\nu_1})\simeq \PGL(2,\K).$ For a maximal order $\O$ the  Borel volume formula \cite[11.1.1, 11.1.3]{MaRe03} yields 
\begin{align}
\Vol(\H^2/\Gamma_{\O^1})=&\frac{8\pi\Delta_k^{3/2}\zeta_k(2)\prod_{\frak p\in \Ram_f A}(\N\frak p-1)}{(4\pi^2)^{[k:\Q]}} &\textrm{ if } \K=\R,\\ 
\Vol(\H^3/\Gamma_{\O^1})=&\frac{4\pi^2\Delta_k^{3/2}\zeta_k(2)\prod_{\frak p\in \Ram_f A}(\N\frak p-1)}{(4\pi^2)^{[k:\Q]}} &\textrm{ if } \K=\C. 
\end{align}
We shall derive the corresponding formulas for the $\Vol(\PGL(2,\K)/\Gamma_V)$ where $V\subset \mathbf G(\mathbb A_f)$ is an open compact subgroup. As in previous section we fix a maximal open compact subgroup $U=\prod_{\frak p}U_\frak p\subset \mathbf G(\mathbb A_f)$ and put $S=\{\frak p| U_\frak p\sim U_\frak p^1\}$. We start by establishing a volume formula for $\Gamma_U$ and next we deduce one for $\Gamma_V$ for $V\subset U$. 
\begin{lemma}\label{lem:MaxCovol} With $U$ as above we have:
\begin{align*}
\Vol(\H^2/\Gamma_U)=&\frac{4\pi\Delta_k^{3/2}\zeta_k(2)\prod_{\frak p\in \Ram_f A}\frac{\N\frak p-1}{2}\prod_{\frak p\in S}\frac{\N\frak p+1}{2}}{|\cl(U)|(4\pi^2)^{[k:\Q]}} &\textrm{ if } \K=\R,\\ 
\Vol(\H^3/\Gamma_U)=&\frac{2\pi^2\Delta_k^{3/2}\zeta_k(2)\prod_{\frak p\in \Ram_f A}\frac{\N\frak p-1}{2}\prod_{\frak p\in S}\frac{\N\frak p+1}{2}}{|\cl(U)|(4\pi^2)^{[k:\Q]}} &\textrm{ if } \K=\C. 
\end{align*}
\end{lemma}
\begin{proof}
If $U_\frak p^g$ equals $U_\frak p^0,U_\frak p^1,U_\frak p^2$ for some $g\in \mathbf G(k_\frak p)$ define  $\O_{U_\frak p}^g$ as $\begin{pmatrix} \frak o_\frak p & \frak o_\frak p\\ 
\frak o_\frak p & \frak o_\frak p\end{pmatrix}, \begin{pmatrix} \frak o_\frak p & \frak o_\frak p\\ \frak p & \frak o_\frak p\end{pmatrix}, \{x\in A_\frak p|\, |\n(x)|_\frak p\leq 1\}$ respectively. Define an order $\O_U:=A\cap \prod_{\frak p}\O_{U_\frak p}$. Choose a maximal order $\O$ containing $\O_U$ and let $\overline{\mathcal O}=\prod_{\frak p}\O_\frak p$ be the weak closure in $A(\mathbb A_f)$. At each finite place $\frak p\not\in S$ we have $\O_\frak p=\O_{U_\frak p}$ and for $\frak p \in S$ the pair $\O_{U_\frak p}\subset \O_\frak p$ is conjugate to $\begin{pmatrix} \frak o_\frak p & \frak o_\frak p\\ \frak p & \frak o_\frak p\end{pmatrix}\subset \begin{pmatrix}\frak o_\frak p & \frak o_\frak p\\ \frak o_\frak p & \frak o_\frak p\end{pmatrix}$. By the strong approximation theorem \cite[7.7.5]{MaRe03} the map $ \O^1/\O_U^1\to \overline{\O}^1/ \overline{\O_U}^1\simeq \prod_{\frak p\in S}\O_\frak p^1/\O_{U_\frak p}^1$ is bijective, so $[\Gamma_{\O^1}:\Gamma_{\O_U^1}]=[\O^1:\O_U^1]=\prod_{\frak p\in S}(\N\frak p+1).$  Therefore 
\begin{align}
\Vol(\H^2/\Gamma_{\O_U^1})=&\frac{8\pi\Delta_k^{3/2}\zeta_k(2)\prod_{\frak p\in \Ram_f A}(\N\frak p-1)\prod_{\frak p\in S}(\N\frak p+1)}{(4\pi^2)^{[k:\Q]}} &\textrm{ if } \K=\R,\\ 
\Vol(\H^3/\Gamma_{\O_U^1})=&\frac{4\pi^2\Delta_k^{3/2}\zeta_k(2)\prod_{\frak p\in \Ram_f A}(\N\frak p-1)\prod_{\frak p\in S}(\N\frak p+1)}{(4\pi^2)^{[k:\Q]}} &\textrm{ if } \K=\C. 
\end{align}
By construction $\Gamma_{\O_U^1}\subset \Gamma_U$ and $\mathcal O_{U_\frak p}\subset U_\frak p k_p^\times.$
It remains to compute the index $[\Gamma_U:\Gamma_{\O^1_U}]$. 
We first compute it when $S=\emptyset$ and then deduce the volume formulas in general case. If $S=\emptyset$, then $\O=\mathcal O_U$ and $\Gamma_U$ is a maximal lattice. We know that $\Gamma_U$ normalizes $\Gamma_{\O^1_U}$ so it must be equal to the normalizer of $\Gamma_{\O^1_U}$. The index of $\Gamma_{\O^1_U}$ in the normalizer was computed by Borel \cite[8.4,8.5,8.6]{Bor81}. Following the notation of \cite[11.6]{MaRe03} we have:
$$[\Gamma_U:\Gamma_{\O^1_U}]=[R^\times_{f,\infty}:(R^\times_f)^2][_2J_1:J_2].$$
The group $J_1$ is $\mathbb A_f^\times/ k_A^\times \overline{\frak o}^\times \prod_{\frak p\in \Ram_f A}k_\frak p$ and $_2J_1$ is the subgroup of its $2$-torsion elements. We skip the description of $R_{f,\infty},J_2$ since by \cite[Lemma 2.1]{ChFr86} (where the same notation is used) we have 
$$[R^\times_{f,\infty}:(R^\times_f)^2][_2J_1:J_2]=2^{|\Ram_f A|+a+r_2}|J_1/J_1^2|,$$ where $a$ is the number of real places of $k$ where $A$ is split and $r_2$ is the number of complex places of $k$. In our setting $a+r_2=1$ so using Lemma \ref{lem:ClassGroup} we get 
$$[\Gamma_U:\Gamma_{\O^1_U}]=2^{|\Ram_f A|+1}|J_1/J_1^2|=2^{|\Ram_f A|+1}\left|\mathbb A_f^\times/ (\mathbb A_f^\times)^2k_A^\times \overline{\frak o}^\times \prod_{\frak p\in \Ram_f A}k_\frak p^\times\right|=2^{|\Ram_f A|+1}|\cl(U)|.$$ Therefore, in the case $S=\emptyset$ we have 
\begin{align}
\label{eq:VolFormMax2}\Vol(\H^2/\Gamma_U)=&\frac{4\pi\Delta_k^{3/2}\zeta_k(2)\prod_{\frak p\in \Ram_f A}\frac{\N\frak p-1}{2}}{|\cl(U)|(4\pi^2)^{[k:\Q]}} &\textrm{ if } \K=\R,\\ 
\label{eq:VolFormMax3}\Vol(\H^3/\Gamma_U)=&\frac{2\pi^2\Delta_k^{3/2}\zeta_k(2)\prod_{\frak p\in \Ram_f A}\frac{\N\frak p-1}{2}}{|\cl(U)|(4\pi^2)^{[k:\Q]}} &\textrm{ if } \K=\C. 
\end{align}
In general case when $S$ is non-empty we find a maximal compact subgroup $U'=\prod_{\frak p}U'_\frak p$ where $U_\frak p'\sim U_\frak p^0$ for all $\frak p\not \in \Ram_f A$ and put $V=U\cap U'$.   By Lemma \ref{lem:LatticeIndex} have $$\frac{[\Gamma_{U'}:\Gamma_V]}{[\Gamma_{U}:\Gamma_V]}=\frac{|\cl(U')|}{|\cl(U)|}\frac{[U':V]}{[U:V]}=\frac{|\cl(U')|}{|\cl(U)|}\prod_{\frak p\in S}\frac{\N\frak p+1}{2}.$$
The formulas (\ref{eq:VolFormMax2},\ref{eq:VolFormMax3}) describe the covolume of $\Gamma_{U'}$ so we get
\begin{align*}
\Vol(\H^2/\Gamma_U)=&\frac{4\pi\Delta_k^{3/2}\zeta_k(2)\prod_{\frak p\in \Ram_f A}\frac{\N\frak p-1}{2}\prod_{\frak p\in S}\frac{\N\frak p+1}{2}}{|\cl(U)|(4\pi^2)^{[k:\Q]}} &\textrm{ if } \K=\R,\\ 
\Vol(\H^3/\Gamma_U)=&\frac{2\pi^2\Delta_k^{3/2}\zeta_k(2)\prod_{\frak p\in \Ram_f A}\frac{\N\frak p-1}{2}\prod_{\frak p\in S}\frac{\N\frak p+1}{2}}{|\cl(U)|(4\pi^2)^{[k:\Q]}} &\textrm{ if } \K=\C. 
\end{align*}\end{proof}

\begin{lemma}\label{lem:LatticeIndex}
Let $U$ be a maximal compact subgroup of $\mathbf G(\mathbb A_f)$ and let $V\subset U$ be an open subgroup. Then $[\Gamma_U:\Gamma_V]=[U:V]\frac{|\cl(U)|}{|\cl(V)|}.$ 
\end{lemma}
\begin{proof}
We have $[\Gamma_U:\Gamma_V]=[U:V][U:V\Gamma_U ]^{-1}$ so we need to show that $[U:V\Gamma_U]=\cl(V)/\cl(U).$ Consider the natural map $\alpha\colon \cl(V)\to \cl(U)$ of finite groups of exponent $2$ (c.f. Lemma \ref{lem:ClassGroup}). The map $\alpha$ fits into an exact sequence 
$$\begin{tikzcd} 1\arrow[r]& \Gamma_U\bs U/V\arrow[r]& \mathbf G(k)\bs \mathbf G(\mathbb A_f)/V\arrow[r,"\alpha"]& \mathbf G(k)\bs \mathbf G(\mathbb A_f)/ U\arrow[r]& 1.\end{tikzcd}$$
It follows that $|U/V\Gamma_U|=|\cl(V)|/|\cl(U)|$.
\end{proof}
Combining Lemmas \ref{lem:MaxCovol} and \ref{lem:LatticeIndex} we get:
\begin{lemma}\label{lem:CongVolume}
Let $U=\prod_{\frak p}U_\frak p\subset \mathbf G(\mathbb A_f)$ be a maximal compact subgroup, $S=\{\frak p| U_\frak p\sim U_\frak p^1\}$ and let $V\subset U$ be an open subgroup. We have
\begin{align*}
\Vol(\H^2/\Gamma_V)=&\frac{[U:V]4\pi\Delta_k^{3/2}\zeta_k(2)\prod_{\frak p\in \Ram_f A}\frac{\N\frak p-1}{2}\prod_{\frak p\in S}\frac{\N\frak p+1}{2}}{|\cl(V)|(4\pi^2)^{[k:\Q]}} &\textrm{ if } \K=\R,\\ 
\Vol(\H^3/\Gamma_V)=&\frac{[U:V]2\pi^2\Delta_k^{3/2}\zeta_k(2)\prod_{\frak p\in \Ram_f A}\frac{\N\frak p-1}{2}\prod_{\frak p\in S}\frac{\N\frak p+1}{2}}{|\cl(V)|(4\pi^2)^{[k:\Q]}} &\textrm{ if } \K=\C. 
\end{align*}
\end{lemma}

\subsection{Volumes of anisotropic tori}
\begin{lemma}\label{lem:TorusVolume}
Let $k$ be a number field, let $l/k$ be a quadratic extension with genus character $\chi$ and let $\mathbf T:=\Res^1_{l/k}\mathbb G_m$ be the associated norm torus. Write $a,b$ for the number of real places of $k$ that are split, inert in $l$ respectively and let $c$ be the number of complex places of $k$. We have
$$\Vol(\mathbf T(k)\bs \mathbf T(\mathbb A))\ll \frac{\Delta_k^{1/2}\N_{k/\Q}(\Delta_{l/k})^{1/2}L(1,\chi)}{2^{a}(2\pi)^{b+c}}.$$
\end{lemma}
\begin{proof}
 We are going to obtain this formula by specializing a more general general result of Ullmo and Yafaev \cite{UY}. Since they work with tori over $\Q$, let us replace $\mathbf T$ with $\Res_{k/\Q}\mathbf T$. The standard measures are compatible with Weil restriction so $\Vol(\mathbf T(k)\bs \mathbf T(\mathbb A))=\Vol(\Res_{k/\Q}\mathbf T(\Q)\bs \Res_{k/\Q}\mathbf T(\mathbb A_\Q)).$ For the remainder of this proof let us write $\mathbf T$ instead of $\Res_{k/\Q}\mathbf T$. The proof boils down to translating the language of \cite{UY} to our setting. 

The standard measure on $\mathbf T(\mathbb A_\Q)$ is the same as the measure $\nu_{T}$ defined by \cite[2.1.3 (11)]{UY}. Let $\omega_\mathbf T$ be the Tamagawa measure on $\mathbf T(\mathbb A_\Q)$ (c.f. \cite[2.1.3 (12)]{UY}), the two measures are related by $c_\mathbf T\nu_\mathbf T=\omega_\mathbf T$ where $c_\mathbf T$ is the \textbf{Shyr constant} of $\mathbf T$. We write $D_\mathbf T=\frac{1}{c_\mathbf T^2}$ and call it the \textbf{ quasi-discriminant} of $\mathbf T$.  Let $\rho_\mathbf T$ be the leading coefficient of the Artin L-function $L(s,X^*(\mathbf T)\otimes \C)$ at $s=1$. In our case $\mathbf T$ is anisotropic, so $\rho_\mathbf T$ will be simply $L(1,X^*(\mathbf T)\otimes \C).$
The Tamagawa number $\tau_\mathbf T$ of $\mathbf T$ is defined by $\omega_{\mathbf T}(\mathbf T(\Q)\bs \mathbf T(\mathbb A_\Q))=\rho_\mathbf T\tau_\mathbf T$. 
 Shyr's formula (\cite[(14)]{UY}, \cite{Shyr}) reads
$$R_\mathbf T h_\mathbf T=w_\mathbf T \rho_\mathbf T \tau_\mathbf T D_\mathbf T^{1/2},$$ where $w_\mathbf T$ is the size of the torsion subgroup of $\mathbf T(\Q)$. Rearranging the formula we find that $$\frac{R_\mathbf T h_\mathbf T}{w_\mathbf T}=\frac{\rho_\mathbf T\tau_\mathbf T}{c_\mathbf T}=\nu_\mathbf T(\mathbf T(\Q)\bs\mathbf T(\mathbb A)).$$
Following \cite[2.2.1]{UY} write $a(\mathbf T)$ for the Artin conductor of the Galois representation $X^*(\mathbf T)\otimes \C$. We use \cite[Proposition 3.1]{UY} to get 
$$\nu_\mathbf T(\mathbf T(\Q)\bs\mathbf T(\mathbb A))=\frac{a(\mathbf T)^{1/2}\tau_\mathbf T\rho_\mathbf T}{2^a(2\pi)^{b+c}\prod_{p}|\Phi_{\mathbf T,p}(\F_p)|}.$$ All we need to know about the factors $|\Phi_{\mathbf T,p}(\F_p)|$ is that they are at least $1$. The Tamagawa number $\tau_\mathbf T$ can be computed in terms of cohomological invariants of $\mathbf T$ \cite{Ono2}. We have 
$$\tau_\mathbf T=\frac{|H^1(\Q,X^*(\mathbf T))|}{|\Sh^1(\mathbf T)|},$$ where $|\Sh^1(\mathbf T)|$ is the Shafarevich-Tate group (see \cite[p. 284]{PlaRap}), defined as the kernel of the natural map $H^1(\Q,\mathbf T)\to \prod_{\nu}H^1(\Q_\nu,\mathbf T).$ Since $X^*(\mathbf{T})\simeq \Ind_{\Gal(k)}^{\Gal(\Q)}\chi$, we have $H^1(\Q,X^*(\mathbf T))\simeq H^1(k,\Z)$ where $\Gal(k)$ acts on $\Z$ by $\chi$. We conclude that $\tau_\mathbf T\leq |H^1(\Q,X^*(\mathbf T))|= 2$. 
$$\nu_\mathbf T(\mathbf T(\Q)\bs\mathbf T(\mathbb A))\leq \frac{2a(\mathbf T)^{1/2}\rho_\mathbf T}{2^a(2\pi)^{b+c}}=\frac{2 \Delta_k^{1/2}\N_{k/\Q}(\Delta_{l/k})^{1/2}\rho_\mathbf T}{2^a(2\pi)^{b+c}}.$$ The lemma is proved. 
\end{proof}
 
 \section{Archimedean estimates and consequences of the support condition}\label{sec:Archimedean}

 Let $\B(R), R>0$ be the subset of $\PGL(2,\K)$ defined by the condition 
 \begin{equation}\label{eq:BallDefinition}
  \frac{\tr(g^* g)}{|\det g|}=\frac{|a|^2+|b|^2+|c|^2+|d|^2}{|\det g|}\leq e^R+e^{-R} \textrm{ where } g=\begin{pmatrix}a & b\\ c& d\end{pmatrix}.
\end{equation}
In Section \ref{ssec:AOI} we compute the Archimedean orbital integrals of the characteristic functions of balls $\B(R)$. We shall assume that the radius $R$ is of the form $R=C+\eta[k:\Q]$ where $\eta, C$ are fixed positive constants. For any arithmetic lattice $\Gamma\subset \PGL(2,\K)$ and hyperbolic element $\gamma\in \Gamma$ we prove an inequality relating the minimal translation length $\m(\gamma)$ with the logarithmic Mahler measure of the eigenvalues of $\Ad(\gamma)$. This is the contents of Lemma \ref{lem:TranslationWeil}. We recall that $\m(\gamma)$ is the minimal displacement of $\gamma$ given by $\log|\lambda|$, where $\lambda$ is an eigenvalue of $\Ad \gamma$ of biggest absolute value\footnote{ If $\gamma$ is hyperbolic is it unique, otherwise all eigenvalues are of modulus $1$.}. In Lemma \ref{lem:TranslationWeil} we also give a lower bound on the degree of the extension generated by the eigenvalues of $\gamma$, which serves as an important input in the proofs of Lemmas \ref{lem:ClassCount},\ref{lem:SmallPrimes} and \ref{lem:WDiscBound}. In Section \ref{sec:ConjClasses}, Lemma \ref{lem:ClassCount} we parametrize and count conjugacy classes $[\gamma]$ in $\mathbf G(k)$ that can bring non zero contribution to the trace estimate from Lemma \ref{lem:TFTraceFormula} for test functions $f$ with $\supp f\subset \B(R)$. For $\eta$ small, the fact that $\Gamma_V$ has a non-torsion element with non-vanishing orbital integral $\O_\gamma(\mathbf 1_{\B(R)})$ has very strong consequences for the distribution of small primes in its invariant trace field $k$ (Lemma \ref{lem:SmallPrimes}), as well as for the non-Archimedean orbital integrals $\O_\gamma(\mathbf 1_V)$ of the same element (Lemma \ref{lem:WDiscBound}). This phenomenon is crucial in the proof of Theorem \ref{mthm:TraceEstimate} and we explain its ramifications in Section \ref{ssec:WDSMTF}. In Lemma \ref{lem:TorusVolumeUB} we estimate the adelic volumes of centralizers of non-torsion elements with non-vanishing orbital integrals and in Lemma \ref{lem:ClassGpUB} we estimate the size of the class group of $k$ under the assumption that at least one non-torsion $\gamma\in\Gamma$ has non-vanishing orbital integral.
\subsection{Archimedean orbital integrals}\label{ssec:AOI}  
 
 \begin{lemma}\label{lem:OIBound} Let $\gamma\in \mathbf G(k_{\nu_1})\simeq \PGL(2,\K)$ be a split semisimple regular element. Then 
 $$\O_\gamma(\mathbf 1_{\B(R)})\ll |\Delta(\gamma)|_\K^{-1/2} \max\{0,e^R-e^{\m(\gamma)}\}^{[\K:\R]/2}.$$
 In particular the orbital integral vanishes unless $\m(\gamma)\leq R$. 
 \end{lemma}
 
\begin{proof}
Write $G=\PGL(2,\K)$. Choose a measurable bounded function $\alpha$ on $G$ such that 
$\int_{G_\gamma}\alpha(tg)dt=1$ for all $g\in G$. Then
\begin{align}
\O_\gamma(\mathbf 1_{\B(R)})=\int_{G_\gamma\backslash G}\mathbf 1_{\B(R)}(x^{-1}\gamma x)dx=\int_G \alpha(g)\mathbf 1_{\B(R)}(g^{-1}\gamma g)dg.
\end{align}
Since $G_\gamma$ is split, it is conjugate to the diagonal subgroup $A$. Conjugating $\gamma$ does not change the orbital integrals so we may assume $\gamma=\begin{pmatrix}
a & 0\\ 0 & b
\end{pmatrix}$ and $G_\gamma=A$. We have the Iwasawa decomposition $G=ANK$ where $N$ is the group of unipotent upper triangular matrices. By \cite[Theorem 2.5.1]{LangSL2} for any integrable $h$ we have 
\begin{equation*}
\int_G h(g)dg=\int_A\int_N\int_K h(ank)dadndk. 
\end{equation*}
Hence,
\begin{align*}
\O_\gamma(\mathbf 1_{\B(R)}) =& \int_A\int_N\int_K \alpha(ank)\mathbf 1_{\B(R)}(k^{-1}n^{-1}a^{-1}\gamma ank)dadndk\\
=& \int_A\int_N\int_K \alpha(ank)\mathbf 1_{\B(R)}(n^{-1}\gamma n)dadndk\\
=& \int_N \mathbf 1_{\B(R)}(n^{-1}\gamma n)dn\\
=& \int_\K \mathbf 1_{\leq e^R+e^{-R}}\left(\frac{|a|^2+|b|^2}{|ab|}+|\Delta(\gamma)||x|^2\right)dx\\
\ll & |\Delta(\gamma)|_\K^{-1/2}\max\{0,e^R+e^{-R}-\frac{|a|^2+|b|^2}{|ab|}\}^{[\K:\R]/2}.
\end{align*}  We have $\frac{|a|^2+|b|^2}{|ab|}=e^{\m(\gamma)}+e^{-\m(\gamma)}$ so the resulting inequality can be rewritten as
$$\O_\gamma(\mathbf 1_{\B(R)})\ll |\Delta(\gamma)|_\K^{-1/2}\max\{0,e^R+e^{-R}-e^{\m(\gamma)}-e^{-\m(\gamma)}\}^{[\K:\R]/2}\leq |\Delta(\gamma)|_\K^{-1/2}\max\{0,e^R-e^{\m(\gamma)}\}^{[\K:\R]/2}.$$
\end{proof}

\subsection{Conjugacy classes with non-vanishing contribution}\label{sec:ConjClasses}
Let $f\in C_c(\PGL(2,\K))$ with $\supp f\subset \B(R)$, let $k, \mathbf G$ be as in Section \ref{sec:CongruenceLattices} and let $U$ be a maximal compact subgroup of $\mathbf G(\mathbb A_f)$. Put $f_\mathbb A:=f\times\mathbf 1_{\SO(3)}^{s-1}\times \mathbf 1_U$. In this section we count the number of regular semisimple, non-torsion conjugacy classes $[\gamma]\subset \mathbf G(k)$ that bring non-zero contribution to the sum on the right hand side of the formula in Lemma \ref{lem:TFTraceFormula}: $$\sum_{[\gamma]\subset \mathbf G(k)\cap W}\Vol(\mathbf G_\gamma(k)\bs \mathbf G_\gamma(\mathbb A))\O_\gamma(f_\mathbb A).$$ The main result of this section is Lemma \ref{lem:ClassCount} which asserts that for $R=C +\eta[k:\Q]$ the number of such classes is of order $e^{\kappa_3(\eta)[K;\Q]+o_{C,\eta}([k:\Q])}$ where $\kappa_3:\R_{>0}\to\R_{>0}$ vanishes as $t\to 0$. 

The problem of counting conjugacy classes contributing to the trace formula readily reduces to a purely Diophantine problem of counting algebraic integers with certain properties. The regular semisimple elements in $\mathbf G(k)$ can be parametrized as follows. For $\gamma\in \mathbf G(k)$ choose a lift $\tilde\gamma\in A^\times$. The extension $l:=k(\tilde\gamma)$ does not depend on the choice of the lift. Therefore $\gamma$ corresponds to a well defined element $\tilde\gamma k^\times:=\xi \in (l^\times-k^\times)/k^\times$. By Skolem-Noether theorem \cite[2.9.8]{MaRe03} the $\mathbf G(k)$ conjugacy class is uniquely determined by $\xi$. Writing $\lambda=\xi/\xi^\sigma$ where $\sigma$ is the generator of $\Gal(l/k)$ we obtain a parametrization of the regular semisimple conjugacy classes in $\mathbf G(k)$ by pairs $(l,\lambda)$ where $l$ is a quadratic extension of $k$ contained in $A$ and $\lambda\neq 1\in l^\times$ is an element satisfying $\N_{l/k}(\lambda)=1$, up to the identification of $(l,\lambda)$ and $(l,\lambda^{-1})$. The numbers $\lambda,\lambda^{-1}$ are precisely the non-trivial\footnote{i.e. distinct from $1$.} eigenvalues of $\Ad \gamma$. 
The extension $l/k$ is determined by $\lambda$ unless $\lambda=-1$, in which case $\gamma$ is a torsion element of order $2$. Finally $\lambda$ needs to be a unit if the conjugacy class of $\gamma$ meets a compact subgroup of $\mathbf G(\A_f)$ (c.f. Lemma \ref{lem:TranslationWeil}). Indeed, assume $[\gamma]$ intersects a compact subgroup $V$ of $\mathbf G(\mathbb A_f)$. At each finite place $\frak p$ we have $\lambda^n+\lambda^{-n}+2=\tr \Ad\gamma^n\in \frak o_\frak p$. It follows that $\lambda,\lambda^{-1}$ are algebraic integers.  We condense the above discussion in the form of a lemma:
\begin{lemma}\label{lem:CGParam}
The non-torsion conjugacy classes $[\gamma]\in \mathbf G(k)$ that meet a a compact subgroup of $\mathbf G(\A_f)$ are parametrized by non-torsion units $\lambda\in \overline k$, up to equivalence of $\lambda$ and $\lambda^{-1}$, with the following properties: \begin{enumerate}\item $l=k(\lambda)$ is a quadratic extension of $k$, $\N_{l/k}(\lambda)=1$, \item the infinite places $\nu\neq \nu_1$ become complex in $l/k$ while the infinite place $\nu_1$ splits into two real places if $\K=\R$ (and into two  complex places if $\K=\C$).\end{enumerate}
The parametrization is given explicitly by $[\gamma]\mapsto \lambda$ where $\lambda$ is a non-trivial eigenvalue of $\Ad(\gamma)$.
\end{lemma} 
We remark that not all such units $\lambda$ need to correspond to conjugacy classes in $\mathbf G(k)$ (which ones do will depend on the algebra $A$) but the map from conjugacy classes to units up to equivalence is always an injection.
To parametrize the conjugacy classes that can contribute to the trace formula, we relate the minimal displacement $\m(\gamma)$ with the logarithmic Mahler measure of the eigenvalues of $\Ad \gamma$. In preparation for the counting problem we also give lower bounds on the degrees of the fields generated by the eigenvalues of non-torsion elements. 
\begin{lemma}\label{lem:TranslationWeil}
Let $[\gamma]$ be a non torsion conjugacy class in $\mathbf G(k)$ that intersects an open compact subgroup $V\subset \mathbf G(\mathbb A_f)$. Let $\lambda,\lambda^{-1} \neq 1$ be the non-trivial eigenvalues of $\Ad\gamma$. Then $\lambda$ is a unit and:\begin{enumerate}
\item If $\K=\R$, then $k(\lambda)=\Q(\lambda)$ and $\m(\lambda)= \m(\gamma).$
\item If $\K=\C$, then either  $k(\lambda)=\Q(\lambda)$ or $[k(\lambda):\Q(\lambda)]=2$. In the first case $\m(\lambda)= 2\m(\gamma)$, in the second $\m(\lambda)=\m(\gamma)$. 
\end{enumerate}
In particular $[\Q(\lambda):\Q]\geq [k:\Q]$ and $h(\lambda)\leq \frac{\m(\gamma)}{[k:\Q]}.$
\end{lemma}
\begin{proof}
Put $l:=k(\lambda)$. It is a quadratic extension of $k$.
We enumerate the embeddings $\Omega_k=\{\sigma_1,\ldots,\sigma_d\}, \Omega_l=\{\sigma_1^1,\sigma_1^2,\sigma_2^1,\ldots,\sigma_d^2\}$ in such a way that $\sigma_i^1,\sigma_i^2$ are the extensions of $\sigma_i$. For any embedding $\sigma\in \Omega_k,\Omega_l$ write $|\cdot|_\sigma$ for the corresponding valuation on $k, l$ respectively and $k_\sigma,l_\sigma$ for the completion with respect to these valuations.
\textbf{ Case $\K=\R.$} There is a unique embedding $\sigma_1\in \Omega_k$ such that $\mathbf G(k_{\sigma_1})\simeq \PGL(2,\K)$ and $\mathbf G(k_{\sigma_i})\simeq \SO(3)$ for $i\geq 2$. Since $\gamma$ is assumed to be hyperbolic, we deduce that $|\sigma_1(\lambda+\lambda^{-1})|>2$ and $|\sigma_i(\lambda+\lambda^{-1})|\leq 2$ if $i\geq 2$. 
The numbers $\sigma_i^1(\lambda),\sigma_i^2(\lambda)$ are the roots of the polynomial $X^2-\sigma_i(\lambda+\lambda^{-1})X+1$ so they are real if $i=1$ and complex of modulus $1$ if $i=2,\ldots,d$. It follows immediately that, up to switching the order, $\sigma_1^1(\lambda)>1$, $\sigma_1^2(\lambda)<1$ and $|\sigma_i^1(\lambda)|=|\sigma_i^2(\lambda)|=1$ for $i\geq 2$. Since $\sigma_1^1$ is the unique embedding in $\Omega_l$ such that $|\sigma_1^1(\lambda)|>1$, the stabilizer of $\sigma_1^1(\lambda)$ in $\Gal(\Q)$ coincides with the stabilizer of $\sigma_1^1$ i.e. $\Gal(l)$. Therefore $\Q(\lambda)=l$. We compute the logarithmic Mahler measure
$$\m(\lambda)=\sum_{i=1}^d(\log^+|\sigma_i^1(\lambda)|+\log^+|\sigma_i^2(\lambda)|)=\log|\sigma_1^1(\lambda)|=\m(\gamma).$$
This proves (1).
\textbf{ Case $\K=\C$.} There is a unique pair of conjugate complex embeddings $\sigma_1,\sigma_2$ such that $\mathbf G(k_{\sigma_1})\simeq \mathbf G(k_{\sigma_2})\simeq  \PGL(2,\K)$ and $\mathbf G(k_{\sigma_i})\simeq \SO(3)$ for $i\geq 3$. Since $\gamma$ is assumed to be hyperbolic we deduce that, up to re-ordering, $|\sigma_1^1(\lambda)|=|\sigma_2^1(\lambda)|>1$,$|\sigma_1^2(\lambda)|=|\sigma_2^2(\lambda)|<1$ and $|\sigma_i^1(\lambda)|=|\sigma_i^2(\lambda)|=1$ for $i\geq 3$. We consider two cases according to whether $\sigma_1^1(\lambda)$ is real or not. If $\sigma_1^1(\lambda)\not\in\R$, then $\sigma_1^1(\lambda)\neq \sigma_2^1(\lambda)$. Hence, the stabilizer of $\sigma_1^1(\lambda)$ in $\Gal(\Q)$ is $\Gal(l)$ so $\Q(\lambda)=l$. We have
$$\m(\lambda)=\sum_{i=1}^d(\log^+|\sigma_i^1(\lambda)|+\log^+|\sigma_i^2(\lambda)|)=(\log|\sigma_1^1(\lambda)|+\log|\sigma_2^1(\lambda)|)=2\m(\gamma).$$
If $\sigma_1^1(\lambda)\in \R$, then the stabilizer of $\sigma_1^1(\lambda)$ in $\Gal(\Q)$ contains $\Gal(l)$ as index $2$ subgroup. Hence $[l:\Q(\lambda)]=2$. By the same calculation as above we have $\m(\lambda)=\m(\gamma).$ 
\end{proof}
\begin{definition}\label{def:CkR}
Define $\mathcal C(k,R)$ as the set of all units $\lambda$ such that $\lambda$ is non-torsion, satisfies the conditions (1),(2) from Lemma \ref{lem:CGParam} and $\m(\lambda)\leq R$ if $\K=\R$ or $\m(\lambda)\leq 2R$ if $\K=\C$. 
\end{definition}
Given Lemmas \ref{lem:CGParam}, \ref{lem:TranslationWeil} and \ref{lem:OIBound} we see that for any test function $f\in C_c(\PGL(2,\K))$ with $\supp f_\infty\subset \B(R)$, the non torsion conjugacy classes that can contribute to the adelic trace formula (c.f. Lemma \ref{lem:TFTraceFormula}) are parametrized by $\lambda\in \mathcal C(k,R)$ up to identifying $\lambda$ with $\lambda^{-1}$. In the next section we will show that mere fact that $\mathcal C(k,R)\neq \emptyset$ has important consequences for the arithmetic of the field $k$. We have the following upper bound on the size of $\mathcal C(k,R).$
\begin{lemma}\label{lem:ClassCount}
Let $R=\eta [k:\Q]+C$ with $0<\eta<1/4$ and $C>0$. There exists a function $\kappa_3:\R_{>0}\to \R_{>0}$ (independent of $\eta,C$) with the following properties: $\lim_{t\to 0}\kappa_3(t)=0$ and $$|\mathcal C(k,R)|\leq e^{\kappa_3(\eta)[k:\Q]+o_{C,\eta}([k:\Q])}.$$
\end{lemma}
\begin{proof} Let us sketch the argument first. Let $\mathcal C'(k,R)=\mathcal C(k,R)/\sim $ where $\sim$ is the equivalence relation of the multiplicative dependence i.e. $\lambda_1\sim \lambda_2$ if there exist $n,m\in\Z\setminus \{0\}$ such that $\lambda_1^n=\lambda_2^m$. As the first step we show that $|\mathcal C(k,R)|\ll R[k:\Q]\log [k:\Q]^4|\mathcal C'(k,R)|.$  This is a relatively quick consequence of Dobrowolski's lower bound on Mahler measures of algebraic number in terms of the degree \cite{Dobr79}. For $R= \eta [k:\Q]+C$ the factor $R[k:\Q]\log [k:\Q]^4$ becomes $e^{o_{C,\eta}([k:\Q])}$ so the initial problem is reduced to 
\begin{equation}\label{eq:CkRFirstReduction}
 |\mathcal C'(k,R)|\leq  e^{\kappa_3(\eta)[k:\Q]+o_{C,\eta}([k:\Q])}.
\end{equation}
In the second step we construct a map $v\colon \mathcal C'(k,R)\to \R^{s-1}$ (recall that $s=[k:\Q]$ if $\K=\R$ and $s=[k:\Q]-1$ if $\K=\C$).
Let $H_1\subset \Hom(k,\C)$ be the set of real embeddings corresponding to the places in $\Ram_\infty A$. For $[\lambda]\in \mathcal C'(k,R)$ choose a representative  $\alpha\in [\lambda]$. Put $$v([\lambda])=\left(\frac{\sigma(\alpha+\alpha^{-1})}{\left(\sum_{\sigma\in H_1}\sigma(\alpha+\alpha^{-1})^2\right)^{1/2}}\right)_{\sigma\in H_1}.$$ The map depends on the choice of $\alpha$ so we fix the choice of representatives for the remainder of the argument. 
We shall show that $v$ has the following properties: \begin{enumerate}\item $\langle v([\lambda]), v([\lambda])\rangle=1$ for $[\lambda]\in \mathcal C'(k,R)$, \item $\langle v([\lambda_1]), v([\lambda_2])\rangle\ll \eta^{1/2}+o_{C,\eta}(1)$ for $[\lambda_1]\neq[\lambda_2]\in \mathcal C'(k,R).$\end{enumerate} In the inequality above and the following ones $o_{C,\eta}(1)$ means a quantity that tends to $0$ as $[k:\Q]\to\infty$. To show these inequalities we use the quantitative Bilu equidistribution results developed in Section \ref{sec:Bilu}. 
The vectors  $v([\lambda])$ form an almost orthogonal family in $\R^{s-1}$. In the third step we apply Kabatjanskii-Levenstein bounds on the maximal number of almost orthogonal vectors to bound the size of $\mathcal C'(k,R)$. More precisely, we use a simplified version of Kabatjanskii-Levenstein bounds proved by Terrence Tao. Since the proof of this bound is published only as a blog entry we reproduce the argument for the reader's convenience. In this way we prove (\ref{eq:CkRFirstReduction}) with $\kappa_3(t)=O(-t\log t)$, which concludes the proof.                                                                                                                                                                                                                                                                                                                                                                                                                               
                                                                                                                                                                                                                                                                                                                                                                                                                          
\textbf{ Step 1.} $|\mathcal C(k,R)|\ll R[k:\Q]\log[k:\Q]^4|\mathcal C'(k,R)|.$  We need to show that any equivalence class generated by relation $\sim$ has at most $O( R[k:\Q]\log [k:\Q]^4)$ elements. Fix $\lambda\in \mathcal C(k,R)$, let $E=\{\lambda'\in \mathcal C(k,R)| \lambda'\sim\lambda\}$ and let $l:=k(\lambda)$. Let $\lambda'\in E$ and let $m,n\in\Z\setminus 0$ be such that $\lambda^n=\lambda'^m$. Since $\N_{l/k}(\lambda^n)=1$, we must have $\lambda^n\not\in k$ as otherwise $\lambda$ would be torsion. Therefore $l=k(\lambda)=k(\lambda^n)=k(\lambda')$. We proved that 
$$E\subset \{\lambda'\in \frak o_l^\times|\, \N_{l/k}(\lambda')=1, \m(\lambda')\leq 2R\}.$$
By Dirichlet's unit theorem and the splitting conditions satisfied by $l$ we have $\{\lambda'\in \frak o_l^\times|\, \N_{l/k}(\lambda')=1\}\simeq \Z\times W$, where $W$ is the torsion subgroup. Let $\alpha$ be the generator of the torsion-free part. Every element $\lambda'\sim \lambda$ can be written as $\lambda'=w \alpha^m$ for $m\in \Z, w\in W$. By \cite{Dobr79} $\m(\alpha)\gg \log[\Q(\alpha):\Q]^{-3}$, so $\m(\lambda')\leq 2R$ implies $m\ll R\log[k:\Q]^3$. The size of $W$ can be estimated by the number of roots of identity in $l$, which is at most $O([l:\Q]\log\log[l:\Q])$. Hence, $|E|\ll R[k:\Q]\log [k:\Q]^4$, as desired.

\textbf{ Step 2.} Let $[\lambda_1],[\lambda_2]\in \mathcal C'(k,R)$. Note that if $k(\lambda_1)=k(\lambda_2)$, then $[\lambda_1]=[\lambda_2]$ (cf. the proof of Lemma \ref{lem:IndepProducts}). Let $L=k(\lambda_1,\lambda_2)$ and let $H'=\{\sigma\in \Hom(L,\C)|\quad \sigma|_k\in H_1\}$. We recall that $H_1\subset \Hom(k,\C)$ was defined as the set of real embeddings corresponding to the places in $\Ram_\infty A$, so $|H_1|=s-1$.  We have $|H'|=4|H_1|$ if $[\lambda_1]\neq [\lambda_2]$ and $|H'|=2|H_1|$ otherwise. Let $\alpha_1, \alpha_2$ be the representatives of $[\lambda_1],[\lambda_2]$. Since $\alpha_i,\alpha^{-1}_i$ are Galois conjugate for $i=1,2$ we have 
\begin{equation}
\sum_{\sigma\in H_1}\sigma((\alpha_1+\alpha_1^{-1})(\alpha_2+\alpha_2^{-1}))=\begin{cases}\sum_{\sigma'\in H'}\sigma'(1+\alpha_1^2) & \textrm{ if } [\lambda_1]=[\lambda_2]\\
                                                                                                                          \sum_{\sigma'\in H'}\sigma'(\alpha_1\alpha_2)& \textrm{ if } [\lambda_1]\neq[\lambda_2].
                                                                                                                          
                                                                                                                         \end{cases}
\end{equation}
By Lemma \ref{lem:TranslationWeil} $[\Q(\alpha_i):\Q]\geq [k:\Q]$. Lemma \ref{lem:BiluFinal} yields $$\sum_{\sigma'\in H'}\sigma'(1+\alpha_1^2)=2[k:\Q]+O(\eta^{1/2}[k:\Q])
+o_{\eta,C}([k:\Q]).$$ If $[\lambda_1]\neq[\lambda_2]$, then by Lemma \ref{lem:IndepProducts} we have $[\Q(\alpha_1\alpha_2)]\geq 2[k:\Q]$ so Lemma \ref{lem:BiluFinal} yields $$\sum_{\sigma'\in H'}\sigma'(\alpha_1\alpha_2)\ll \eta^{1/2}[k:\Q]+ o_{\eta,C}([k:\Q]).$$
By inspecting the definition of $v$ we infer that $\langle v([\lambda_1]), v([\lambda_2])\rangle=1$ if $[\lambda_1]=[\lambda_2]$ and  $\langle v([\lambda_1]), v([\lambda_2])\rangle\ll \eta^{1/2}+o_{C,\eta}(1)$ if $[\lambda_1]\neq[\lambda_2]\in \mathcal C'(k,R).$

\textbf{ Step 3.} We apply Lemma \ref{lem:KL} to the collection of vectors $v([\lambda]), [\lambda]\in \mathcal C'(k,R)$, $n=|H_1|$ and $C_1:=\eta^{1/2}[k:\Q]^{1/2}+ o_{C,\eta}([k:\Q]^{1/2})$. As $\eta<1/4$, the assumption $C_1\leq \frac{1}{2}n^{1/2}$ used in Lemma \ref{lem:KL} will be satisfied for $[k:\Q]$ large enough. Since $|H_1|$ is either $[k:\Q]-1$ or $[k:\Q]-2$ we can replace $|H_1|$ with $[k:\Q]$ in the resulting estimates. With $C_2>0$ as in Lemma \ref{lem:KL} we get
\begin{align*}|\mathcal C'(k,R)|&\ll_{C,\eta} \left(\frac{[k:\Q]}{\eta[k:\Q]+ o_{C,\eta}([k:\Q])}\right)^{C_2(\eta^{1/2}[k:\Q]^{1/2}+ o_{C,\eta}([k:\Q]^{1/2}))^2}\\&=e^{-C_2 \eta\log\eta[k:\Q] +o_{C,\eta}([k:\Q])}=e^{\kappa_3(\eta)[k:\Q]+o_{C,\eta}([k:\Q])},\end{align*} with $\kappa_3(t):=-C_2t \log t.$

\end{proof}
\begin{lemma}\label{lem:IndepProducts}
Let $\lambda_1,\lambda_2\in \mathcal{C}(k,R)$. Then either $\lambda_1,\lambda_2$ are multiplicatively dependent or $[\Q(\lambda_1\lambda_2):\Q]\geq 2[k:\Q]$. 
\end{lemma}
\begin{proof}
Let $l_i=k(\lambda_i)$ for $i=1,2$. \textbf{ Case $l_1=l_2=:l$.} According to Definition \ref{def:CkR} $\lambda_i$ is a unit of $\frak o_{l}$ satisfying $\N_{l/k}(\lambda_i)=1.$ By Dirichlet's unit theorem and the splitting conditions satisfied by $l/k$ at Archimedean places we have $\rk \frak o_{l}^\times-\rk \frak o_k^\times=1$. Hence, the subgroup $\{x\in \frak o_l^\times|\, \N_{l/k}(x)=1\}$ is virtually cyclic, so $\lambda_1$ and $\lambda_2$ are multiplicatively dependent. 
\textbf{ Case $l_1\neq l_2$.} Write $L=k(\lambda_1,\lambda_2)$. It is a degree $4$ Galois extension of $k$, with $\Gal(L/k)\simeq (\Z/2\Z)^2$. Let $\Hom(k,\C)=\{\sigma_1,\ldots,\sigma_d\}$ and let $\sigma_{i}^{a,b}, i=1,\ldots,d,$ $ a,b\in\{\pm 1\}$ be an enumeration of $\Hom(L,\C)$ in such a way that $\sigma_{i}^{a,b}|_k=\sigma_i$, $\sigma_{i}^{a,b}(\lambda_1)=\sigma_{i}^{1,1}(\lambda_1)^a, \sigma_{i}^{a,b}(\lambda_2)=\sigma_{i}^{1,1}(\lambda_2)^b$. 
\textbf{ Subcase $\K=\R$.} There is a unique embedding $\sigma_1\in\Hom(k,\C)$ such that $\mathbf G(k_{\sigma_1})\simeq \PGL(2,\K)$ and $\mathbf G(k_{\sigma_i})\simeq \SO(3)$ for $i\geq 2$. For every $i\geq 2, a,b\in\{\pm 1\}$ we have $|\sigma_i^{a,b}(\lambda_1\lambda_2)|=1$. If the values of $|\sigma_1^{a,b}(\lambda_1\lambda_2)|$ are pairwise distinct for $a,b\in \{\pm 1\}$, then the stabilizer of $\sigma_1^{1,1}(\lambda_1\lambda_2)$ in $\Gal(\Q)$ is an index $[L:\Q]$ subgroup so $[\Q(\lambda_1\lambda_2):\Q]=[L:\Q]=4[k:\Q].$ If the values of $|\sigma_1^{a,b}(\lambda_1\lambda_2)|$ are not pairwise distinct, then $\lambda_1,\lambda_2$ or $\lambda_1\lambda_2$ is torsion. First two possibilities are excluded since $\mathcal C(k,R)$ contains no torsion elements and the last implies that $\lambda_1$ and $\lambda_2$ are multiplicatively dependent. 
\textbf{ Subcase $\K=\C$.} There is a unique pair of conjugate complex embeddings $\sigma_1,\sigma_2$ such that $\mathbf G(k_{\sigma_1})\simeq \mathbf G(k_{\sigma_1})\simeq  \PGL(2,\K)$ and $\mathbf G(k_{\sigma_i})\simeq \SO(3)$ for $i\geq 3$. For every $i\geq 3, a,b\in\{\pm 1\}$ we have $|\sigma_i^{a,b}(\lambda_1\lambda_2)|=1$. If the values of $|\sigma_1^{a,b}(\lambda_1\lambda_2)|$ are pairwise distinct, then the stabilizer of $\sigma_1^{1,1}(\lambda_1\lambda_2)$ in $\Gal(\Q)$ is an index $[L:\Q]$ subgroup if $\sigma_1^{1,1}(\lambda_1\lambda_2)\not\in \R$ or an index $[L:\Q]/2$ subgroup if $\sigma_1^{1,1}(\lambda_1\lambda_2)\in \R$. In both cases we deduce that $[\Q(\lambda_1\lambda_2):\Q]\geq 2[k:\Q].$ The last remaining case is when the values of $|\sigma_1^{a,b}(\lambda_1\lambda_2)|$ are not pairwise distinct. Then either $\lambda_1,\lambda_2$ or $\lambda_1\lambda_2$ is torsion, which can happen only if $\lambda_1$ and $\lambda_2$ are multiplicatively dependent. 


\end{proof}
\begin{lemma}\label{lem:KL}[Kabatjanskii-Levenstein, Tao] Let ${v_1,\ldots,v_m}$ be unit vectors in ${{\R}^n}$ such that $|\langle v_i, v_j \rangle| \leq$ $ C_1 n^{-1/2}$ for some $\frac{1}{2} \leq C_1 \leq \frac{1}{2} \sqrt{n}$. Then we have ${m \leq (\frac{n}{C_1 ^2})^{C_2 C_1^2}}$ for some absolute constant ${C_2}$.
\end{lemma}
\begin{proof}
We give the proof due to Terrence Tao, from a blog entry\footnote{\tt{https://terrytao.wordpress.com/2013/07/18/a-cheap-version-of-the-kabatjanskii-levenstein-bound\\-for-almost-orthogonal-vectors/}}. Let $k$ be such that $$C_1n^{-1/2}\leq 2^{-1/k}{n+k-1\choose k}^{-1/2k}.$$ Let $\{w_{i}\}_{i=1,\ldots, m}\in (\R^n)^{\otimes k}$ be given as $w_i:=v_i^{\otimes k}$.  For each $1\leq i\neq j\leq m$ we have $|\langle w_i,w_j\rangle| \leq C_1^{k}n^{-k/2}\leq \frac{1}{2}{n+k-1\choose k}^{-1/2}.$ We claim that $m< 2{n+k-1\choose k}.$ For the sake of contradiction assume $m\geq 2N:=2{n+k-1\choose k}$ and consider the Gramm matrix $M:=(\langle w_i,w_j\rangle)_{i,j=1,\ldots,2N}.$ Since the vectors $w_i$ lie in the $k$-th symmetric power of $\R^n$, which is an $\dim \Sym^k(\R^n)=N$ dimensional space, the rank of $M$ is at most $N$. Hence, $M-{\rm Id}$ has eigenvalue $-1$ with multiplicity at least $N$. Taking the Hilbert-Schmidt norm of $M$ we conclude 
$$\sum_{i,j=1,\ldots, 2N}|\langle w_i,w_j\rangle -\delta_{i,j}|^2\geq N.$$ On the other hand, $|\langle w_i,w_j\rangle| \leq \frac{1}{2}N^{-1/2}$ for $i\neq j$ so the left hand side is bounded above by $2N(2N-1)\frac{1}{4}N^{-1}$. We get $\frac{1}{2}(2N-1)\geq N$ which is a contradiction. Therefore $m<2{n+k-1\choose k}\leq n^{k}C_1^{-2k}.$ Using Stirling approximation we verify that one can choose $k=\lfloor C_2C_1^{2}\rfloor$ for some large constant $C_2$. 
\end{proof}

We end this Section with an upper bound on Archimedean orbital integrals in arithmetic lattices. This complements Lemma \ref{lem:OIBound}.
\begin{lemma}\label{lem:AOIBound}
Let $[\gamma]$ be a non torsion conjugacy class in $\mathbf G(k)$ that intersects non-trivially an open compact subgroup $V\subset \mathbf G(\mathbb A_f)$. Then 
$$\O_\gamma(\mathbf 1_\B(R))\ll e^{[\K:\R]R/2}\log [k:\Q]^{3[\K:\R]}.$$
\end{lemma}
\begin{proof}
 By Lemma \ref{lem:OIBound} $\O_\gamma(\mathbf 1_\B(R))\ll |\Delta(\gamma)|_{\K}^{-1/2}\max\{0,e^R-e^{\m(\gamma)}\}^{[\K:\R]/2}$. Lemma will follow once we prove that $|\Delta(\gamma)|\gg \log[k:\Q]^{-6}.$
 Let $\lambda,\lambda^{-1}$ be the non-trivial eigenvalues of $\Ad\gamma$. By Lemma \ref{lem:TranslationWeil} the logarithmic Mahler measure satisfies $\m(\lambda)\leq 2\m(\gamma)=:|\log|\lambda||$. As $\lambda$ is a non-torsion algebraic integer of degree at most $2[k:\Q]$, we can invoke Dobrowolski's lower bound \cite{Dobr79}: $\m(\lambda)\gg \log[\Q(\lambda):\Q]^{-3}\geq \log [k:\Q]^{-3}.$ Therefore 
 $$|\Delta(\gamma)|=|(1-\lambda)(1-\lambda^{-1})|\gg \log [k:\Q]^{-6}.$$ 
\end{proof}

\subsection{Weyl discriminants and small primes in the invariant trace field}\label{ssec:WDSMTF}
Let $R=\eta[k:\Q]+C$, with $0<\eta<1/4, 0<C$, and let $\mathcal C(k,R)$ be as in Definition \ref{def:CkR}. As explained in the previous section, the set $\mathcal C(k,R)$ modulo inverse parametrizes the conjugacy classes $[\gamma]\subset \mathbf G(k)$ that can intersect $\B(R)\times \SO(3)^{s}\times U$ for some open compact subgroup $U\subset \mathbf G(\mathbb A_f)$. In this section we estimate various arithmetic invariants attached to $[\gamma]$ and $k$. Namely, the norm of the Weyl discriminant, the norm of the relative discriminant of the quadratic extension of $k$ generated by the eigenvalues of $\gamma$ and the standard volume of $\mathbf G_\gamma(k)\bs \mathbf G_\gamma(\mathbb A)$. Under the assumption that $\mathcal C(k,R)\neq \emptyset$ we prove an upper bound on $|\cl(k)|$.  
\begin{lemma}\label{lem:WDiscBound}
Let $[\gamma]\subset \mathbf G(k)$ and assume that a non-trivial eigenvalue $\lambda$ of $\Ad \gamma$ is in $\mathcal C(k,R)$. Let $l=k(\lambda)$ be the quadratic extension generated by $\lambda$. Then
\begin{enumerate}
 \item $|\N_{k/\Q}(\Delta(\gamma))|=\exp({O(-\eta^{1/2}\log \eta [k:\Q])+ o_{C,\eta}([k:\Q])}).$
 \item $\N_{k/\Q}(\Delta_{l/k})=\exp({O(-\eta^{1/2}\log \eta [k:\Q])+ o_{C,\eta}([k:\Q])}).$
\end{enumerate}
\end{lemma}
\begin{proof}
By Lemma \ref{lem:TranslationWeil} $[\Q(\lambda):\Q]\geq [k:\Q]$. We have $\Delta(\gamma)=(1-\lambda)(1-\lambda^{-1})=\N_{l/k}(1-\lambda)$. Therefore, by Lemma \ref{lem:BiluFinal} 
$$\log |\N_{k/\Q}(\Delta(\gamma))|=\log|\N_{l/\Q}(1-\lambda)|=[l:\Q(\lambda)]\log|\N_{\Q(\lambda):\Q}(1-\lambda)|\ll -\eta^{1/2}\log\eta [k:\Q]+o_{C,\eta}([k:\Q]).$$
For the second assertion observe that $\Delta_{l/k}$ divides the principal ideal generated by $\det\begin{pmatrix} 1 & \lambda\\ 1 & \lambda^{-1}\end{pmatrix}^2=-\Delta(\gamma^2)$.  Since $\lambda^2\in \mathcal C(k,2R)$, by the previous point we have $$\log \N_{k/\Q}(\Delta_{l/k})=O(-\eta^{1/2}\log \eta [k:\Q])+ o_{C,\eta}([k:\Q]).$$
\end{proof}

\begin{lemma}\label{lem:SmallPrimes} 
If $\mathcal C(k,R)\neq \emptyset$, then: 
\begin{enumerate}
 \item For every $X>0$ we have $\pi_k(X)\ll_{X} \kappa_1(\eta)[k:\Q]+ o_{X,C,\eta}([k:\Q])$, where $\kappa_1\colon \R_{>0}\to\R_{>0}$ is as in Lemma \ref{lem:BiluZeta}.
 \item For any $s=\sigma+it,\sigma>1$ we have $\log |\zeta_{k}(s)|\leq \kappa_{2,\sigma}(\eta)[k:\Q]+o_{C,\eta}([k:\Q]),$ where $\kappa_{2,\sigma}\colon \R_{>0}\to\R_{>0}$ is as in Lemma \ref{lem:BiluZeta}.
\end{enumerate}
\end{lemma}
\begin{proof}
Let $\lambda\in \mathcal C(k,R)$ and let $l=k(\lambda)$ be the quadratic extension generated by $\lambda$. Each prime $\frak p$ of $\frak o_k$ of norm $\leq X$ splits into at most two primes of $\frak o_l$ of norm at most $X^2$. Each prime $\frak P$ of $\frak o_l$ lies over a prime of $\Q(\lambda)$ of less or equal norm and the map $\frak P\to \frak P|_{\frak o_{\Q(\lambda)}}$ is at most $[l:\Q(\lambda)]$-to-$1$. Therefore, by Lemmas \ref{lem:TranslationWeil} and \ref{lem:BiluZeta} $$\pi_k(X)\leq \pi_l(X^2)\leq [l:\Q(\lambda)]\pi_{\Q(\lambda)}(X^2)\ll_{X} \kappa_1(\eta)[k:\Q]+ o_{C,X,\eta}([k:\Q]).$$
For the second point we mimic the proof Lemma \ref{lem:BiluZeta} (2). For any constant $M>2$ 
\begin{align*}\log |\zeta_k(\sigma+it)|\leq \log |\zeta_k(\sigma)|\leq& 2\sum_{\frak p}\frac{1}{\N(\frak p)^\sigma}\leq 2\sum_{\N(\frak p)\leq M}\frac{1}{N(\frak p)^\sigma}+2[k:\Q]\sum_{p> M} \frac{1}{p^\sigma}\\
\leq &2C_1(M)\kappa_1(\eta)[k:\Q]+ o_{C,M,\eta}([k:\Q])+ \frac{2M^{1-\sigma}}{\sigma-1}[k:\Q].
 \end{align*}
For the second inequality we used (1), with $C_1(M)$ being the implicit constant, and a crude inequality  $\sum_{p> M}\frac{1}{p^\sigma}\leq \int_{M}^\infty \frac{dx}{x^\sigma}=\frac{M^{1-\sigma}}{\sigma-1}.$ There is a function $M_0:\R_{>0}\to\R_{>0}$ with $M_0(\eta)\to\infty$ as $\eta\to 0$ such that $C_1(M_0(\eta))k_1(\eta)\to 0$ as $\eta\to 0$. Put $\kappa_{2,\sigma}(\eta):=2C_1(M_0(\eta))\kappa_1(\eta)+2M_0(\eta)^{1-\sigma}(\sigma-1)^{-1}$. Choosing $M=M_0(\eta)$ we get $$\log |\zeta_{k}(s)|\leq \kappa_{2,\sigma}(\eta)[k:\Q]+o_{C,\eta}([k:\Q]).$$
\end{proof}
We are ready to estimate the adelic volumes of centralizers of elements parametrized by $\mathcal C(k,R)$. 
\begin{lemma}\label{lem:TorusVolumeUB}
There exists a function $\kappa_4:\R_{>0}\to \R_{>0}$, such that $\kappa_4(t)\to 0$ as $t\to 0$, with the following property. Let $\gamma\in \mathbf G(k)$ such that $\Ad \gamma$ has an eigenvalue $\lambda\in \mathcal C(k,R)$. Then
$$\Vol(\mathbf G_\gamma(k)\bs \mathbf G_\gamma(\mathbb A))\ll_{\eta} \frac{\Delta_k^{1/2+\kappa_4(\eta)+o_{C,\eta}(1)}}{(2\pi)^{[k:\Q]}}.$$ 
\end{lemma}
\begin{proof}
Let $l=k(\gamma)\simeq k(\lambda)$. Let $\chi$ be the quadratic character corresponding to $l/k$. By Lemma \ref{lem:TorusVolume} 
$$\Vol(\mathbf G_\gamma(k)\bs \mathbf G_\gamma(\mathbb A))\ll \frac{\Delta_k^{1/2} \N_{k/\Q}(\Delta_{l/k})^{1/2}\rho_{\chi}}{(2\pi)^{[k:\Q]}}.$$
Note that $\rho_\chi=L(1,\chi)$. Lemma will follow once we show $\N_{k/\Q}(\Delta_{l/k})^{1/2}\rho_{\chi}\ll_{\eta} \Delta_{k}^{\kappa_3(\eta)+o_{C,\eta}(1)}.$ In fact, it will be more natural to bound the whole numerator since it ``almost'' appears as a value of the completed L-function attached to $\chi$. We will follow the standard proof of Brauer-Siegel theorem (c.f. \cite[XVI,\S 1]{LangANT}). Let $\Lambda(s,\chi)$ be the completed L-function (c.f. \cite[p. 299]{LangANT}), given by 
$$\Lambda(s,\chi):=(2^{-r_2}\pi^{-[k:\Q]/2}\Delta_k^{1/2}\N_{k/\Q}(\Delta_{l/k})^{1/2})^s\prod_{\nu\in\Sigma_\infty} \Gamma(s_\nu/2)L(s,\chi),$$ where 
$s_\nu=s_\nu(\chi):=N_\nu(s+i\varphi_\nu(\chi))+|m_\nu(\chi)|$ are defined as in \cite{LangANT}. Specializing to our case $\varphi_\nu(\chi)=0$ for all $\nu\in \Sigma_\infty$, $m_\nu(\chi)=1,N_v=1$ for $\nu\in \Ram_\infty A$ and $m_{\nu_1}(\chi)=0, N_{\nu_1}=[\K:\R]$ for the unique Archimedean place $\nu_1\not\in \Ram_\infty A$. We get 
$$\Lambda(s,\chi):=(2^{-r_2}\pi^{-[k:\Q]/2}\Delta_k^{1/2}\N_{k/\Q}(\Delta_{l/k})^{1/2})^s \Gamma([\K:\R]s/2)\Gamma(\frac{s+1}{2})^{[k:\Q]-[\K:\R]}L(s,\chi).$$
The completed L-function satisfies the functional equation $$W(\chi)\Lambda(s,\chi)=\Lambda(1-s,\chi),$$ with $|W(\chi)|=1$ (c.f. \cite[Corollary 2 XIV,\S 8]{LangANT}). 
The function $\Lambda(s,\chi)$ is entire so by Phragm\'en-Lindel\"of principle $|\Lambda(1,\chi)|\leq \sup_{t\in\R}|\Lambda(\sigma+it)|$ for any $\sigma>1$. In the sequel we assume $1<\sigma<2$. 
We have $|L(s,\sigma+it)|\leq \zeta_k(\sigma)$ so 
\begin{align*}|\Lambda(1,s)|\leq & (2^{-r_2}\pi^{-[k:\Q]/2}\Delta_k^{1/2}\N_{k/\Q}(\Delta_{l/k})^{1/2})^\sigma \Gamma([\K:\R]\sigma/2)\Gamma\left(\frac{\sigma+1}{2}\right)^{[k:\Q]-[\K:\R]}\zeta_k(\sigma)\\
 \ll& (\pi^{-[k:\Q]/2}\Delta_k^{1/2}\N_{k/\Q}(\Delta_{l/k})^{1/2})^\sigma \Gamma\left(\frac{\sigma+1}{2}\right)^{[k:\Q]}\zeta_k(\sigma).
\end{align*}
Therefore, by Lemmas \ref{lem:WDiscBound} and \ref{lem:SmallPrimes} the last expression can be bounded by 
\begin{align*}|\Lambda(1,\chi)|\ll & (\pi^{-[k:\Q]/2}\Delta_k^{1/2})^\sigma \Gamma\left(\frac{\sigma+1}{2}\right)^{[k:\Q]}\exp(-\sigma\eta^{1/2}\log \eta [k:\Q]+ \kappa_{2,\sigma}(\eta)[k:\Q]+o_{C,\eta}([k:\Q]))\\
 \leq& \pi^{-[k:\Q]/2}\Delta_k^{\sigma/2}\exp(-\sigma\eta^{1/2}\log \eta [k:\Q]+ \kappa_{2,\sigma}(\eta)[k:\Q]+o_{C,\eta}([k:\Q]))\\
\end{align*}
By Minkowski's lower bound $\log |\Delta_k|\geq C_3[k:\Q]$, for some  $C_3>0$. We get 
$$ |\Lambda(1,\chi)|\ll \Delta_k^{1/2}\pi^{-[k:\Q]/2}\exp\left(\log|\Delta_k|\left( \frac{\sigma-1}{2}-C_3^{-1}\sigma\eta^{1/2}\log \eta+ C_3^{-1}\kappa_{2,\sigma}(\eta)+o_{C,\eta}(1)\right)\right).$$

Since $\lim_{t\to 0}\kappa_{2,\sigma}(t)=0$, there exists a function $\sigma_0\colon \R_{>0}\to \R_{>1}$ with $\lim_{t\to 0}\sigma_0(t)=1$ such that $\lim_{t\to 0}\kappa_{2,\sigma_0(t)}(t)=0$. Put $\kappa_4(t):=\frac{\sigma_0(t)-1}{2}-C_3^{-1}\sigma_0(t)t^{1/2}\log t+ C_3^{-1}\kappa_{2,\sigma_0(t)}(t).$ Then 
$$ |\Lambda(1,\chi)|\ll \Delta_k^{1/2}\pi^{-[k:\Q]/2}\exp\left(\log|\Delta_k|(\kappa_4(\eta)+o_{C,\eta}(1))\right)$$ and $\kappa_4(t)\to 0$ as $t\to 0$.
Since $\Lambda(1,\chi)=2^{-r_2}\pi^{-[k:\Q]/2}\Gamma([\K:\R]/2)\Delta_{k}^{1/2} \N_{k/\Q}(\Delta_{l/k})^{1/2}\rho_{\chi}$ we deduce that the numerator of the volume formula satisfies $\Delta_k^{1/2}\N_{k/\Q}(\Delta_{l/k})^{1/2}\rho_{\chi}\ll \Delta_{k}^{\frac{1}{2}+ \kappa_4(\eta)+o_{C,\eta}(1)}.$ Lemma is proved.
\end{proof} 
Lastly, we give an upper bound on $|\cl(k)|$ in the case when $\mathcal C(k,R)\neq\emptyset$. 
\begin{lemma}\label{lem:ClassGpUB}
Suppose that $\mathcal C(k,R)\neq \emptyset$. Then $|\cl(k)|\leq e^{-1.43[k:\Q]}\Delta_{k}^{\frac{1}{2}+\kappa_5(\eta)+o_{\eta,C}(1)}$ with $\kappa_5:\R_{>0}\to\R_{>0}$ satisfying $\lim_{t\to 0}\kappa_5(t)=0$. 
\end{lemma}
\begin{proof}
By the class number formula the problem reduces to an upper bound on the residue $\rho_k$ and a lower bound on the regulator $R_k$. The argument is completely parallel to the proof of lemma \ref{lem:TorusVolumeUB}. By Phragmen-Lindel\"of principle (c.f. \cite[XVI,\S 1]{LangANT}) for any $\sigma>1$ we have
$$\rho_k \leq \Delta_k^{\frac{\sigma-1}{2}}\left(2\pi(2\pi)^{-\sigma}\Gamma(\sigma)\right)^{r_2}\left( \pi^{-\sigma/2}\Gamma(\sigma/2)\right)^{r_1}\sigma(\sigma-1)\zeta_k(\sigma).$$
Hence, by Lemma \ref{lem:SmallPrimes}, for $\sigma-1$ small
$$\rho_k\leq \Delta_k^{\frac{\sigma-1}{2}}\exp(O([k:\Q](\sigma-1))+\kappa_{2,\sigma}(\eta)[k:\Q]+o_{\eta,C}([k:\Q])).$$ Using Minkowski's lower bound $\log|\Delta_k|\geq C_3[k:\Q]$ we get 
$$\rho_k\leq \Delta_k^{\frac{\sigma-1}{2} +C_4(\sigma-1)+C_4\kappa_{2,\sigma}(\eta)+o_{\eta,C}(1)},$$ for some positive constant $C_4>0$.
We choose $\kappa_5(t)$ in the same way as $\kappa_4$ the proof of lemma \ref{lem:TorusVolumeUB} to get 
$$\rho_k \leq \Delta_k^{\kappa_5(\eta)+o_{C,\eta}(1)}.$$
By \cite[p.620 Corollary]{Friedman89} (see also \cite{Zimmert}) we have a lower bound on the regulator $R_k/w_k\gg \exp(0.241[k:\Q]+0.497 r_1)$. Since all places of $k$ except possibly $1$ are real, we deduce $R_k/w_k\gg \exp(0.738[k:\Q])$. Hence, $$|\cl(k)|=\frac{\Delta_k^{1/2}w_k\rho_k}{2^{r_1}(2\pi)^{r_2}R_k}\ll \Delta_k^{1/2+\kappa_5(\eta)+o_{\eta,C}(1)}\exp(-1.43[k:\Q]).$$
\end{proof}

 \section{Proof of the main estimate}\label{sec:ProofMT}
 Theorem \ref{mthm:TraceEstimate} asserts that for  $\eta>0$ small enough, for every arithmetic lattice $\Gamma\subset G=\PGL(2,\K)$ with invariant trace field $k$, $f\in C_c(G), \|f\|_\infty\leq 1$ and $\supp f\subset \B(R), R=C+\eta[k:\Q]$ (with $\B(R)$ defined by (\ref{eq:BallDefinition})) we have 
 \begin{equation}\label{eq:MainEstP1}
 \left|\sum_{\substack{[\gamma]\subset \Gamma_V\\ \gamma \textrm{ non-torsion}}}\Vol(\Gamma_\gamma\bs G_\gamma)\O_\gamma(f)\right|\ll_{C} \begin{cases}\Vol(\Gamma\bs G)^{11/12}\Delta_k^{-4/9}&\textrm{ if }\Gamma \textrm{ is congruence arithmetic.}\\
                                                                                                                                \Vol(\Gamma\bs G)\Delta_k^{-4/9} & \textrm{ if }\Gamma \textrm{ is just arithmetic.}
                                                                                                                               \end{cases}\end{equation}
This section is devoted to the proof of Theorem \ref{mthm:TraceEstimate}. In Section \ref{ssec:UBMT} we give an upper bound on left hand side of \ref{eq:MainEstP1}. Next, in Section \ref{ssec:LBVC} we compare it with a lower bound on volume and conclude that (\ref{eq:MainEstP1}) indeed holds. 
 \subsection{Upper bound on the trace}\label{ssec:UBMT}
It is enough to work under the assumption that $f\geq 0$ since the LHS of \ref{eq:MainEstP1} for $f$ is obviously bounded by the same expression for $|f|$. Let $k,\mathbf G$ be as in Section \ref{sec:CongruenceLattices}. Let $U$ be a maximal open compact subgroup of $\mathbf G(\mathbb A_f)$. Write $U=\prod_{\frak p}U_\frak p$ where $U_\frak p$ are maximal compact subgroups of $\mathbf G(k_\frak p)$. Let $S$ be the set of primes where $U_\frak p$ is conjugate to $U_\frak p^1$ (see  (\ref{eq:MaxCompDefs})). Let $V$ be an open subgroup of $U$. Finally let $f\in C_c(G)$ be such that $\|f\|_\infty \leq 1,f\geq 0$ and $\supp f\subset \B(R)$ where $R=C+\eta[k:\Q]$. By Lemma \ref{lem:TFTraceFormula} we have 
\begin{equation}\label{eq:MainEstUB1}
I_V:=\sum_{\substack{[\gamma]\subset \Gamma_V\\ \gamma \textrm{ non-torsion}}}\Vol(\Gamma_{V,\gamma}\bs G_\gamma)\O_\gamma(f)\leq \frac{2}{|\cl(V)|}\sum_{\substack{[\gamma]\subset \mathbf G(k)\\ \gamma \textrm{ non-torsion}}}\Vol(\mathbf G_\gamma(k)\bs \mathbf G_\gamma(\mathbb A))\O_\gamma(f_\mathbb A),\end{equation} where $f_\mathbb A:=f\times \mathbf 1_{\SO(3)}^{s-1}\times [U:V]\mathbf 1_{V}.$
Due to the product structure of $f_\mathbb A$ we have $\O_\gamma(f_\mathbb A)=[U:V]\O_\gamma(f)\O_\gamma(\mathbf 1_V).$ We have $|\O_\gamma(f)|\leq \O_\gamma(\mathbf 1_{\B(R)})$ so by Lemma \ref{lem:OIBound} the orbital integral $\O_\gamma(f)$ vanishes unless $\m(\gamma)\leq R$. We parametrize the set of conjugacy classes that can bring non-zero contribution to the right hand side by looking at their eigenvalues. Define ${\rm eig}([\gamma]):=\{\lambda,\lambda^{-1}\}$ where $\lambda$ is a non-trivial eigenvalue of $\Ad \gamma$. In Section \ref{sec:ConjClasses} we explained how $\lambda,\lambda^{-1}$ determine the $\mathbf G(k)$-conjugacy class of any non-torsion element $\gamma$. By Lemma \ref{lem:TranslationWeil} the condition $\m(\gamma)\leq R$ forces ${\rm eig}([\gamma])\subset \mathcal C(k,R)$ (c.f. Definition \ref{def:CkR}). 

Either $\mathcal C(k,R)=\emptyset$ and then the right hand side of (\ref{eq:MainEstUB1}) vanishes or $\mathcal C(k,R)\neq \emptyset$. Clearly, in order to bound $I_V$ (defined in (\ref{eq:MainEstUB1})) it's enough to consider only the second case. 
\begin{align*}
 \sum_{\substack{[\gamma]\subset \mathbf G(k)\\ \gamma \textrm{ non-torsion}}}\Vol(\mathbf G_\gamma(k)\bs \mathbf G_\gamma(\mathbb A))\O_\gamma(f_\mathbb A)=&\frac{1}{2} \sum_{\lambda\in \mathcal C(k,R)}\sum_{{\rm eig}([\gamma])\ni \lambda}\Vol(\mathbf G_\gamma(k)\bs \mathbf G_\gamma(\mathbb A))[U:V]\O_\gamma(f)\O_\gamma(\mathbf 1_V).
\end{align*}
For each $\lambda\in \mathcal C(k,R)$ there is at most one conjugacy class $[\gamma]$ with $\lambda\in {\rm eig}([\gamma]).$ By Lemma \ref{lem:TorusVolumeUB} and Proposition \ref{prop:NonArchBound}, for any $m\geq 6$ we can bound the right hand side by 
$$\ll \sum_{\lambda\in \mathcal C(k,R)}\frac{\Delta_k^{\frac{1}{2}+\kappa_4(\eta)+o_{C,\eta}(1)}}{(2\pi)^{[k:\Q]}}\O_{\gamma}(f)|\N_{k/\Q}(\Delta(\gamma))|^{7}J_1[U:V]^{1-1/2m} J_2^{1/2m}J_3^{1/2m}.$$ 
The function $\kappa_4\colon \R_{>0}\to\R_{>0}$ satisfies $\lim_{t\to 0}\kappa_4(t)=0$ and  the constants $J_1,J_2,J_2$ (c.f. Proposition \ref{prop:NonArchBound}) are defined as \begin{align*}J_1=&2^{3\pi_k(65)+|\Ram_f A|/2+|S|/2},\\ J_2=&2^{2[k:\Q]+|S|+|\Ram_f A|}|\cl(k)/\cl(k)^2|\frac{|\cl(V)|}{|\cl(U)|},\\ J_3=&\zeta_k(2)^{24} 162^{|S|}2^{|\Ram_f A|}\prod_{\frak p\in S\cup\Ram_f A}(N(\frak p)+1)^3.\end{align*} We use Lemma \ref{lem:WDiscBound} to bound $|\N_{k/\Q}(\Delta(\gamma))|$ and Lemma \ref{lem:AOIBound} to bound $|\O_{\gamma}(f)|$. We get 
\begin{align*}&\sum_{\lambda\in \mathcal C(k,R)}\frac{\Delta_k^{\frac{1}{2}+\kappa_4(\eta)+o_{C,\eta}(1)}}{(2\pi)^{[k:\Q]}}\exp(O(-\eta^{1/2}\log \eta [k:\Q])+ o_{C,\eta}([k:\Q]))J_1[U:V]^{1-1/2m} J_2^{1/2m}J_3^{1/2m},\\ 
\end{align*} since $\eta=O(-\eta^{-1/2}\log \eta)$ for $\eta$ small. Note that the summands no longer depend on $\gamma$. By Minkowski's lower bound $\log \Delta_k\geq C_3 [k:\Q], C_3>0$, so we can bound the last expression by  
\begin{align*}&\frac{|\mathcal C(k,R)|}{(2\pi)^{[k:\Q]}}\Delta_k^{\frac{1}{2}+\kappa_4(\eta)-C_4\eta^{1/2}\log \eta+ o_{C,\eta}(1)}J_1[U:V]^{1-1/2m} J_2^{1/2m}J_3^{1/2m},
\end{align*} for some $C_4>0$. Using the upper bound on $|\mathcal C(k,R)|$ from Lemma \ref{lem:ClassCount} and Minkowski's lower bound 
we get $|\mathcal C(k,R)|\leq \Delta_k^{C_3^{-1}\kappa_3(\eta)+o_{C,\eta}(1)}.$ Therefore 
$$I_V\leq  |\cl(V)|^{-1}(2\pi)^{-[k:\Q]}\Delta_k^{\frac{1}{2}+ \kappa_6(\eta)+o_{C,\eta}(1)} [U:V]^{1-1/2m} J_1J_2^{1/2m}J_3^{1/2m},$$
where $\kappa_6(t):=\kappa_4(\eta)-C_4\eta^{1/2}\log \eta+ C_3^{-1}\kappa_3(\eta).$ Note that $\kappa_6(\eta)\to 0$ as $\eta\to 0$. 
 Since $\mathcal C(k,R)\neq \emptyset$, we can use Lemma \ref{lem:SmallPrimes} to bound $J_1$
\begin{align*}J_1=&2^{3\pi_k(65)+|\Ram_f A|/2+|S|/2}\leq 2^{|\Ram_f A|/2+|S|/2}\exp(O(\kappa_1(\eta)[k:\Q])+o_{C,\eta}([k:\Q]))\\ \leq& 2^{|\Ram_f A|/2+|S|/2}\Delta_k^{C_5\kappa_1(\eta)+o_{C,\eta}(1)},\end{align*} for some $C_5>0$.
We have $e^{1.43}>4$ so Lemma \ref{lem:ClassGpUB} gives $|\cl(k)|\leq \Delta_k^{1/2+\kappa_5(\eta)+o_{\eta,C}(1)}4^{-[k:\Q]}$. Hence 
$$J_2\ll 2^{|\Ram_f A|+|S|}\Delta_{k}^{1/2+\kappa_5(\eta)+o_{\eta,C}(1)}\frac{|\cl(V)|}{|\cl(U)|}.$$ By Lemma \ref{lem:SmallPrimes} and Minkowski's lower bound
$$J_3\ll \Delta_k^{C_3^{-1} 24\kappa_{2,2}(\eta)+o_{\eta,C}(1)}162^{|S|}2^{|\Ram_f A|}\prod_{\frak p\in \Ram_f A\cup S}(\N(\frak p)+1)^{3}.$$
We multiply all the bounds above. After some algebraic manipulations  we arrive at:
\begin{equation}\label{eq:TraceUBFinal} I_V\ll \left(\frac{[U:V]}{|\cl(V)|}\right)^{1-\frac{1}{2m}}|\cl(U)|^{-\frac{1}{2m}}\frac{\Delta_k^{\frac{1}{2}+\frac{1}{4m}+\kappa_{7,m}(\eta)+o_{C,\eta}(1)}}{(2\pi)^{[k:\Q]}}\prod_{\frak p\in \Ram_f A\cup S}2^{\frac{1}{2}+\frac{5}{m}}(\N(\frak p)+1)^{\frac{3}{2m}}\end{equation} 
where $\kappa_{7,m}(\eta):=\kappa_6(\eta)+\frac{\kappa_5(\eta)}{2m}+C_5\kappa_1(\eta)+C_3^{-1}\frac{24}{2m}\kappa_{2,2}(\eta).$ Note that $\kappa_{7,m}(\eta)\to 0$ as $\eta\to 0$. To shorten notation we will suppress the dependence of $\kappa_{7,m}$ on $m$ and write $\kappa_7$.
 \subsection{Lower bound on volume and the conclusion}\label{ssec:LBVC}

Let $X=\H^2$ if $\K=\R$ and $X=\H^3$ if $\K=\C$. By Lemma \ref{lem:CongVolume} 
\begin{equation}\label{eq:IneqVol1} \Vol(\Gamma_V\bs X)\gg \frac{[U:V]}{|\cl(V)|}\frac{\Delta_k^{3/2}}{(4\pi^2)^{[k:\Q]}}\prod_{\frak p\in \Ram_f A\cup  S}\frac{\N(\frak p)-1}{2}.\end{equation}
The proof of Theorem \ref{mthm:TraceEstimate} is divided into two stages. First we treat the case of a congruence subgroup $\Gamma_V$, same as in the previous section and then deduce a version for non-congruence lattices. If $\mathcal C(k,R)=\emptyset$ we have $I_V=0$ so there is nothing to prove. Assume that $\mathcal C(k,R)\neq\emptyset$ and consider the product $I_V \Vol(\Gamma_V\bs X)^{\frac{1-2m}{2m}}$ for $m\geq 6$. 

By (\ref{eq:TraceUBFinal}) and (\ref{eq:IneqVol1}) \begin{align*}I_V\Vol(\Gamma_V\bs X)^{\frac{1-2m}{2m}}\ll& |\cl(U)|^{-\frac{1}{2m}} \Delta_k^{\frac{1-m}{m}+\kappa_7(\eta)+o_{\eta,C}(1)}(2\pi)^{^{\frac{m-1}{m}}[k:\Q]}\prod_{\frak p\in \Ram_f A\cup S}\frac{2^{\frac{3m+9}{2m}}(\N(\frak p)+1)^{\frac{3}{2m}}}{(\N(\frak p)-1)^{\frac{2m-1}{2m}}}.\end{align*}
Observe that for every $m\geq 6$ we have 
\begin{equation*}2^{\frac{3m+9}{2m}}\frac{(\N(\frak p)+1)^{\frac{3}{2m}}}{(\N(\frak p)-1)^{\frac{2m-1}{2m}}}\leq 2^{9/4}(\N(\frak p)-1)^{-1/2}\frac{(\N(\frak p)+1)^{3/2m}}{(\N(\frak p)-1)^{5/2m}}\leq \begin{cases}
                                                                                                                 1 & \textrm{ if } \N(\frak p)\geq 24,\\
                                                                                                                 7 & \textrm{ otherwise.}
                                                                                                                \end{cases}
\end{equation*}
Therefore, by Lemma \ref{lem:SmallPrimes} \begin{align*}I_V\Vol(\Gamma_V\bs X)^{\frac{1-2m}{2m}}\ll& \Delta_k^{\frac{1-m}{m}+\kappa_7(\eta)+o_{\eta,C}(1)}(2\pi)^{\frac{m-1}{m}[k:\Q]}7^{\pi_k(23)}\\ \ll& \Delta_k^{\frac{1-m}{m}+\kappa_8(\eta)+o_{\eta,C}(1)}(2\pi)^{\frac{m-1}{m}[k:\Q]},\end{align*} with $\kappa_8(\eta)\to 0$ as $\eta\to 0$.

We specialize $m=6$ (lowest possible parameter, due to the assumptions of Proposition \ref{prop:NonArchBound}). By Odlyzko's lower bound \cite{Odlyzko} $\Delta_k\gg 60^{[k:\Q]}$, so $\Delta_k^{0.449}\gg (2\pi)^{[k:\Q]}.$ We infer that
\begin{equation} I_V\Vol(X\bs \Gamma_V)^{-\frac{11}{12}}\ll ((2\pi)^{[k:\Q]} \Delta_k^{-0.449})^{\frac{5}{6}}\Delta_k^{-0.551\cdot \frac{5}{6}+\kappa_8(\eta)+o_{\eta,C}(1)}\ll \Delta_k^{-0.459+\kappa_8(\eta)+o_{\eta,C}(1)}.\end{equation}
If $\eta$ is small enough, then $-0.459+\kappa_8(\eta)<-0.45$. For such $\eta$ we have 
$$ I_V\Vol(X\bs \Gamma_V)^{-\frac{11}{12}}\ll \Delta_k^{-4/9}.$$
This proves Theorem \ref{mthm:TraceEstimate} for the congruence arithmetic lattices. 

Now, let $\Gamma$ be any arithmetic lattice in $G:=\PGL(2,\K)$ and let $f\in C_c(G)$ satisfy $\|f\|_\infty\leq 1, f\geq 0, \supp f\subset \B(C+\eta[k:\Q])$. Since every maximal arithmetic lattice is a congruence lattice, we can find a congruence lattice $\Gamma_V$ such that, up to conjugacy, $\Gamma\subset \Gamma_V$. 
Formal calculation shows that 
$$\Vol(\Gamma\bs X)^{-1}I_\Gamma:=\Vol(\Gamma\bs X)^{-1}\sum_{\substack{[\gamma]\subset \Gamma \\ \textrm{non-torsion}}}\Vol(\Gamma_\gamma\bs G_\gamma))\O_\gamma(f)\leq \Vol(\Gamma_V\bs X)^{-1} I_V,$$ where $I_V$ is as before. Hence, $\Vol(\Gamma\bs X)^{-1} I_\Gamma\leq \Vol(\Gamma_V\bs X)^{-1}I_V\ll \Vol(\Gamma_V\bs X)^{-\frac{11}{12}}I_V\ll \Delta_k^{-4/9}.$ In the second inequality we used a well known fact that the volume of arithmetic quotient $\Gamma_V\bs X$ is bounded away from $0$. This proves Theorem \ref{mthm:TraceEstimate} for non-congruence arithmetic lattices.

\section{Benjamini-Schramm Convergence}\label{sec:BSConvergence}
The aim of this section is to prove Theorem \ref{mthm:BSconv}. Let $\Gamma$ be a torsion-free arithmetic lattice in $G=\PGL(2,\K)$. The case of non-uniform arithmetic lattices\footnote{If $\K=\R$ they are defined over $\Q$ and if $\K=\C$ the are defined over a quadratic imaginary number field.} was treated in \cite[Theorem A]{Raim13} so we may assume that $\Gamma$ is a co-compact lattice. Let $\B(R)$ be defined by (\ref{eq:BallDefinition}) and let $K\subset G$ be the unique maximal compact subgroup of $G$ stabilizing $\B(R)$ from both sides. Write $X=G/K$ ($X=\H^2$ if $\K=\R$, $X=\H^3$ if $\K=\C$). 
\begin{proof}[Theorem \ref{mthm:BSconv}.] Let us start with the congruence case $\Gamma=\Gamma_V$. We need to show that for $R=C+\eta[k:\Q]$, with $\eta>0$ sufficiently small, we have 
\begin{equation}
\Vol((\Gamma\backslash X)_{<R})\ll_C \Vol(\Gamma\backslash X)^{11/12}\Delta_k^{-4/9}
\end{equation}
To this end, we apply Theorem \ref{mthm:TraceEstimate} to the lattice $\Gamma$ and $f=\mathbf 1_{\B(R)}$ - the characteristic function of an $R$-ball, defined as in Section \ref{sec:Archimedean}. This function is not continuous but it can be approximated from above by continuous compactly supported functions so the estimate from Theorem \ref{mthm:TraceEstimate} still applies. We have 
\begin{equation}
|\tr R_\Gamma f-\Vol(\Gamma\backslash G)|\ll_C\Vol(\Gamma\backslash G)^{11/12}\Delta_k^{-4/9}.
\end{equation}
Unfolding the proof of Selberg's trace formula for compact quotients (see e.g. \cite[p. 9, second equality]{Clay03}) gives
\begin{align}
\tr R_\Gamma f=&\int_{\Gamma\backslash G}\sum_{\gamma\in \Gamma}f(x^{-1}\gamma x)dx\\
=& \int_{\Gamma\bs G}|\{\B(R)\cap x^{-1}\Gamma x\}|dx.\\
=& \Vol(\Gamma\backslash X)+\int_{\Gamma\backslash X}\left(|\{\B(R)\cap x^{-1}\Gamma x\}|-1\right)dx
\end{align}
The set of points $x\in \Gamma\bs X$ whose injectivity radius is smaller than $R$ can be described as $\{\Gamma x\in \Gamma\bs X\mid |\B(R)\cap x^{-1}\Gamma x|\geq 2\}$. Hence
\begin{equation}
\Vol((\Gamma\backslash X)_{<R})\leq \tr R_\Gamma f-\Vol(\Gamma\backslash G)
\end{equation}
By Theorem \ref{mthm:TraceEstimate}  $\Vol((\Gamma\backslash X)_{<R})\ll_{C}\Vol(\Gamma\backslash X)^{11/12}\Delta_k^{-4/9}$ This ends the proof in the congruence case. In the general case we choose a congruence arithmetic lattice $\Gamma_V$ containing $\Gamma$.  We have 
$$\Vol((\Gamma\bs X)_{\leq R})\leq \sum_{\substack{[\gamma]\subset \Gamma\\ \textrm{ non-torsion }}}\Vol(\Gamma_\gamma\bs G_\gamma)\O_\gamma(f)\leq \frac{I_V}{[\Gamma_V:\Gamma]}\leq \Vol(\Gamma\bs X)\Delta_k^{-4/9},$$ with $I_V$ defined in (\ref{eq:MainEstUB1}). Theorem \ref{mthm:BSconv} is proved.

\end{proof}
\begin{remark} Same argument shows that in any semisimple Lie group $G$ if $(\Gamma_n)_{n\in\mathbb N}$ has the limit multiplicity property, then $(\Gamma_n\bs X)_{n\in\mathbb N}$ converges Benjamini-Schramm to $X$. 
\end{remark}
In a subsequent joint work with Jean Raimbault \cite[Theorem A]{FraRaim} we managed to lift the torsion-free assumption. Since \cite{FraRaim} refers to the old version of this paper, we indicate how \cite[Theorem A]{FraRaim} can be shown using Theorem \ref{mthm:TraceEstimate}. Let $(\Gamma_n)_{n\in\NN}$ be a sequence of non-conjugate arithmetic lattices that are either pairwise non-commensurable or are all congruence. In the first case let us replace $(\Gamma_n)_{n\in \NN}$ by a sequence of maximal lattices $\Gamma_n'\supset \Gamma_n$. We can do that because Benjamini-Schramm convergence for $(\Gamma_n'\bs X)_{n\to\NN}\to X$ will imply the convergence of $(\Gamma_n\bs X)_{n\to\NN}\to X$. These lattices are always congruence, and by assumption they are pairwise non-commensurable . The problem is reduced to the case of a sequence of pairwise non-conjugate congruence arithmetic lattices. Theorem \ref{mthm:TraceEstimate} immediately implies that \cite[Theorem 2.5]{FraRaim} holds for such sequences. We combine this with \cite[Prop. 2.4]{FraRaim} to get \cite[Theorem A]{FraRaim}.

\section{Applications}
 \subsection{Gelander conjecture}\label{sec:Gelander}
This section is devoted to the proof of Theorem \ref{mt.Gelander} and Corollaries \ref{mc.Gelander},\ref{cor:TorsionBd}.
\begin{proof}[Theorem \ref{mt.Gelander}]
Let $\Gamma\subset \PGL(2,\C)$ be a torsion-free arithmetic lattice with the invariant trace field $k$. Put $M=\Gamma\backslash \mathbb H^3$. Gelander already proved the conjecture for non-uniform arithmetic lattices so we shall assume that $\Gamma$ is uniform\footnote{ The key feature used for non-uniform lattices is that they are all defined over a quadratic imaginary field. This implies a uniform lower bound on the lengths of closed geodesics on the quotients.} . Following the method from \cite{Gelander1} we will construct a simplicial complex $\mathcal N$ homotopic to $M$ as the nerve of a covering of $M$ by certain balls. We are able to bound the number of simplices in $\mathcal N$ because both the size of the thin part of the manifold and its injectivity radius can be controlled by the degree $[k:\Q]$. Let $\varepsilon$ be the Margulis constant for $\mathbb H^3$.

  Let ${\rm V}(R)$ be the volume of an $R$-ball in $\H^3$. For $x\in \Gamma\backslash  \mathbb H^3$ or $x\in \mathbb H^3$ we will write $\B(x,R)$ for the closed ball of radius $R$ centered in $x$. Define  $i(x)=\min\{\injrad x, 1\}$ for $x\in M$. 
Let $\mathcal B$ be a maximal with respect to inclusion set of points in $M$ satisfying the following conditions: For any distinct $x,y\in \mathcal B$ we have $\B(x,i(x)/16)\cap \B(y,i(y)/16)=\emptyset$. 
   
\textbf{ Claim 1.} $$\bigcup_{x\in\mathcal B}\B(x,i(x)/5)=M.$$
\begin{proof}
Let $y\in M$. By maximality of $\mathcal B$ there exists $x\in \mathcal B$ such that $\B(x,i(x)/16)\cap \B(y,i(y)/16)\neq \emptyset$. Hence $d(x,y)<\frac{i(x)+i(y)}{16}$.
If $i(x)\geq i(y)$, then $d(x,y)<i(x)/8$ so $y\in \B(x,i(x)/5)$. Otherwise 
$$i(y)-i(x)\leq d(x,y)< \frac{i(x)+i(y)}{16},$$ because $i(x)$ is $1$-Lipschitz. In particular, $i(y)<\frac{17}{15}i(x)$. It follows that $$d(x,y)<\frac{i(x)+i(y)}{16}<\frac{2}{15}i(x),$$ so $y\in \B(x,i(x)/5)$.
\end{proof}
\textbf{Claim 2.} For every $y\in \mathcal B$ the number of $x\in\mathcal B$ such that $\B(x,i(x)/5)\cap \B(y,i(y)/5)\neq \emptyset$ is at most $3104$.
\begin{proof}
Let $S_y:=\{x\in \mathcal B|\, \B(x,i(x)/5)\cap \B(y,i(y)/5)\neq \emptyset\}$. Let $x\in S_y$. Note that $$i(x)\leq i(y)+d(x,y)<i(y)+\frac{i(x)+i(y)}{5},$$ so $i(x)<\frac{3}{2}i(y)$ and $d(x,y)< i(y)/2$. Similarly $i(y)<\frac{3}{2}i(x)$ and $d(x,y)<i(x)/2$. Hence $\B(x,i(x)/16)\subset \B(y,i(x)/16+d(x,y))\subset \B(y, 19/32 i(y))$. The balls $\B(x,i(x)/16), x\in S_y$ are pairwise disjoint, each of volume ${\rm V}(i(x)/16)>{\rm V}(1/24 i(y))$. By comparing the volumes of ${\rm V}(1/24 i(y))$ and ${\rm V}(19/32 i(y))$ we get
$$|S_y|\leq \frac{{\rm V}(19/32 i(y))}{{\rm V}(1/24 i(y))}\leq \frac{{\rm V}(19/32)}{{\rm V}(1/24)}\approx 3103.573<3104.$$
The inequalities follow from the formula for the volume of a ball in hyperbolic $3$-space ${\rm V}(R)=\pi(\sinh 2R-2R)$ \cite[p.83 Ex 3.4.5]{Ratcliffe}. 
\end{proof}

Let $\mathcal U$ be the open cover $\bigcup_{x\in\mathcal B} \B(x,i(x)/5)=M$. By the first claim it is indeed a cover of $M$. Any nonempty intersection of sets in $\mathcal U$ is a convex set so it is contractible. It follows that the cover $\mathcal U$ is "good" in the terminology of \cite{BotT}. By \cite[Theorem 13.4]{BotT} the nerve $\mathcal N$ of $\mathcal U$ is homotopy equivalent to $M$. By definition, the vertices in $\mathcal N$ correspond to the open sets in $\mathcal U$ and $k$-simplices correspond to unordered $k$-tuples in $\mathcal{U}$ with nonempty intersection. Using the second claim we deduce that the degree of vertices in $\mathcal N$ is bounded by $3104$. 

It remains to bound $|\mathcal B|$, which is the number of vertices in $\mathcal N$. We will bound the size separately for $\mathcal B_1:=\mathcal B\cap M_{\geq 1}$ and $\mathcal B_2:=\mathcal B\cap M_{<1}$. The union $\bigsqcup_{x\in B_1}\B(x,1/16)$ is disjoint so
$$|\mathcal B_1|\leq \frac{\Vol(M)}{{\rm V}(1/16)}\ll \Vol(M).$$

By Lemma \ref{lem:TranslationWeil} and Dobrowolski's Theorem \cite{Dobr79} we get that for $\gamma\neq 1$
$$\m(\gamma)\gg(\log[k:\Q])^{-3}.$$
We deduce that the injectivity radius of $M$ is bounded below by $C(\log[k:\Q])^{-3}<1$ for some absolute positive constant $C$. The disjoint union $\bigsqcup_{x\in B_2}\B(x,C\log[k:\Q])^{-3}/16)$ lies in $M_{<17/16}$ so 
$$|\mathcal B_2|\leq \frac{\Vol(M_{<17/16})}{{\rm V}(C\log[k:\Q])^{-3}/16)}\ll \Vol(M_{<17/16})(\log[k:\Q])^9).$$
By Theorem \ref{mthm:BSconv} and Odlyzko's lower bound \cite{Odlyzko} (Minkowski's weaker bound would also suffice) we get $|\mathcal B_2|\ll \Vol(M)60^{-[k:\Q]4/9}(\log[k:\Q])^3\ll \Vol(M)$. Hence 
$|\mathcal B|\ll \Vol(M)$. This proves that the number of vertices in $\mathcal N$ is at most linear in the volume of $M$.
\end{proof}
To prove Corollary \ref{mc.Gelander} one just has to repeat the steps of the proof of \cite[Theorem 11.2]{Gelander1}. 
The explicit presentation of $\Gamma$ from Corollary \ref{mc.Gelander} implies the following bound on the size of the torsion part of $H_1(\Gamma\bs \mathbb H^3,\Z)$:
$$\log |H_1(\Gamma\bs \mathbb H^3,\Z)_{\rm tors}|\ll \Vol(\Gamma\bs \mathbb H^3).$$ This is Corollary \ref{cor:TorsionBd}.

 \subsection{Growth of Betti numbers}\label{sec:BettiNumbers}
 \begin{proof}[Corollary \ref{mc.Betti}]Let $\Gamma$ be a co-compact lattice in $G=\PGL(2,\C)$. 
By \cite[VI 3.2, VI 4.2 VII 4.9]{BW}  $\dim_\C H^1(\Gamma\bs X,\C)= m(I_{P,\delta_1,0}, \Gamma)$, where $I_{P,\delta_1,0}$ is the irreducible representation of $G$ defined in \cite[VI,\S 3]{BW}. Write $I_1:=I_{P,\delta_1,0}.$ 
Consider a sequence of congruence arithmetic co-compact torsion-free lattices in $G$ that are either congruence or pairwise non-commensurable. By Theorem \ref{mthm:LMP} $(\Gamma_i)_{i\in\NN}$ has the limit multiplicity property. Since $I_1$ is not discrete series ($\PGL(2,\C)$ has none), we have $\lim_{n\to\infty} \frac{m(I_1,\Gamma_n)}{\Vol(\Gamma_n\bs G)}=0$. 
\end{proof}
\begin{proof}[Theorem \ref{thm:MultiUB}]
As in the proof above $\dim_{\C} H^1(\Gamma\bs X,\C)= m(I_1, \Gamma)$, so the problem reduces to an estimate on $m(I_1,\Gamma).$ Our estimate for multiplicity will only use the fact that $I_1$ in tempered but not discrete series. We will use the argument of DeGeorge and Wallach from \cite{DeWa78}. For that purpose we will need Harish-Chandra's expansion to control the decay of the matrix coefficients of $I_1$. We borrow the notation from \cite[IV]{BW}. Let $K$ be a maximal compact subgroup of $G$. Let $H$ be the Hilbert space underlying $I_1$. Let $H_0$ be the subspace of $K$-finite vectors of $H$. Let $P,M,A\subset G$ be the subgroups of upper triangular, diagonal and real positive diagonal matrices respectively. 
Let $\frak a$ be the Lie algebra of $A$, let $\Phi(P,A)$ be the set of roots of $P$ relative to $A$. Let $\rho$ be the half-sum of roots in $\Phi(P,A)$ and let $\frak a^+=\{h\in \frak a| \alpha(h)>0, \alpha\in \Phi(P,A)\}.$ 
By \cite[Theorem 1.2, 1.3 IV \S]{BW} there exists a finite set o weights $E^0(P,I_1)\subset \frak a^*_\C$ and non-zero functions $P_\Lambda\colon \frak a\times H_0\times H_0\to \C, \Lambda\in E^0(P,I_1)$ such that for every $H\in \frak a^+$ 
\begin{equation}\label{eq:MCAsymptotics}
\lim_{t\to\infty} \frac{\sum_{\Lambda\in E^0(P,I_1)}e^{\Lambda(tH)}P_\Lambda(tH; v_1,v_2)}{\langle I_1(\exp(tH)v_1,v_2\rangle}=1.
\end{equation}
The functions $P_\Lambda(H,v_1,v_2)$ are linear in $v_1$, anti-linear in $v_2$ and polynomial in $H$. 
Since $I_1$ is tempered \cite[3.2 III]{BW}, but not discrete series we must have $\Re\, \Lambda\leq -\rho$ for any $\Lambda\in E^0(P,I_1)$, with equality for at least one element. Let $d$ be the maximal degree of $P_\Lambda$ for which $\Re\,\Lambda=-\rho$. Fix a unit length vector $v_0\in H_0$ and put $\phi(g):=\langle I_1(g)v_0,v_0\rangle$. From (\ref{eq:MCAsymptotics}) we deduce that for sufficiently big $R$ we have 
\begin{align}\label{eq:MCUB2} \int_{\B(R)}|\phi(g)|^2 dg\gg& \int_{K}\int_K\int_{\frak a^+}^{\rho(H)\leq 2R}e^{-2\rho(H)}|\phi(k_1\exp(H)k_2)|^2dH dk_1dk_2 \gg R^{2d+1}.
\end{align}
For integration formula in spherical coordinates we refer to \cite[p.125 (5)]{Shalom}. 

We are ready to apply the method of DeGeorge and Wallach. Put $u:=\phi\cdot \mathbf 1_{\B(R)}$ and let $u\ast \hat u$ be the convolution of $u$ with $\hat u(g):=\overline u(g^{-1})$. By \cite[3.2]{DeWa78} we have 
$$ m(I_1, \Gamma)\leq \tr R_\Gamma (u\ast \hat u) \|u\|_2^{-4}.$$
We also observe that $\|u\ast \hat u\|_\infty=u\ast \hat u(1)=\|u\|_2^2$. Assume now that $\Gamma$ is a torsion-free arithmetic lattice with invariant trace field $k$. Let $\eta$ be as in Theorem \ref{mthm:TraceEstimate}. Put $R:=\frac{\eta}{2}[k:\Q].$ Then, $\supp (u\ast \hat u)\subset \B(\eta[k:\Q])$.  By Theorem \ref{mthm:TraceEstimate} $$\tr R_\Gamma (u\ast \hat u)\leq \left((u\ast \hat u)(1)+ O(\Delta_k^{-4/9})\|u\ast \hat u)\|_\infty\right)\Vol(\Gamma\bs G)\ll \Vol(\Gamma\bs G)\|u\|_2^2.$$ 
Hence, by (\ref{eq:MCUB2}) $m(I_1, \Gamma)\ll \Vol(\Gamma\bs G)\|u\|_2^{-2}\ll \Vol(\Gamma\bs G)[k:\Q]^{-2d-1}.$ The degree $d$ is non-negative so we get the theorem. 


\end{proof}

\begin{acknowledgement}
This began as a part of author's PhD thesis at the Universit\'e Paris-Sud. I would like to thank my supervisor Emmanuel Breuillard for suggesting this problem as well as for many useful remarks. I am grateful to Nicolas Bergeron and Erez Lapid for careful reading the first version of the manuscript. I acknowledge the support of ERC Consolidator Grant No. 648017 during the last stages of work. I am thankful to the Institute for Advanced Study for providing excellent working conditions when I wrote the current version of the manuscpript. Finally I thank anonymous referees whose valuable remarks and suggestions led to a much improved exposition and improvement of results.
\bibliographystyle{plain}
\end{acknowledgement}
\bibliography{ref}
\end{document}